\documentclass[a4paper,10pt]{article}  
\usepackage{amsmath,amsthm,amssymb}
\usepackage{mathtools}
\usepackage{graphicx}  
\usepackage{thmtools}  
\usepackage[linktocpage=true, colorlinks=true, urlcolor=blue, linkcolor=blue, citecolor=blue]{hyperref} 

\DeclareFontFamily{OT1}{pzc}{}
\DeclareFontShape{OT1}{pzc}{m}{it}%
              {<-> s * [1.000] pzcmi7t}{}
\DeclareMathAlphabet{\mathpzc}{OT1}{pzc}%
                                 {m}{it}
\newcommand{\bd}{{\ell^\ast_\BD}}
\newcommand{\bda}{{\ell^{\ast a}_\BD}}
\newcommand{\bdtil}{{\ell^\ast_{\tilde\BD}}}

\newcommand{\PRbig}{\mathcal{P}}
\newcommand{\ELbig}{\mathcal{E}}
\newcommand{\Fbig}{\mathcal{F}}
\newcommand{\BDbig}{\mathcal{B}}

\newcommand{\APbig}{\mathcal{A}_\kappa} 
\newcommand{\BDkappabig}{\mathcal{B}_\kappa}
\newcommand{\PPbig}{\mathcal{P}_\delta}

\newcommand{\PR}{{\scriptscriptstyle{\PRbig}}}
\newcommand{\EL}{{\scriptscriptstyle{\ELbig}}}
\newcommand{\F}{{\scriptscriptstyle{\Fbig}}}
\newcommand{\BD}{{\scriptscriptstyle{\BDbig}}}

\newcommand{\cc}{{cc}}

\newcommand{\refPR}{\ref{EQ:intro-primal}}
\newcommand{\refF}{\ref{EQ:intro-fenchel}}
\newcommand{\refBD}{\ref{EQ:intro-BD}}
\newcommand{\refAP}{\ref{EQ:APinner}}
\newcommand{\refBDkappa}{\ref{EQ:BDkappa}}
\newcommand{\refPP}{\ref{EQ:PP}}
\newcommand{\refEL}{\ref{EQ:EL}}
\newcommand{\refPRtilde}{$\tilde\PRbig$}
\newcommand{\refBDtilde}{$\tilde\BDbig$}

\newcommand{\SOL}{S}
\newcommand{\card}{\operatorname{card}}

\newcommand{\A}{\mathcal{X}}
\newcommand{\Aa}{\A^a}
\newcommand{\Ap}{\A^p}

\newcommand{\sx}[1]{\Delta_{#1}}

\renewcommand{\sp}[2]{\left\langle #1, #2 \right\rangle}

\newcommand{\abs}[1]{\lvert #1 \rvert}
\newcommand{\RRR}{\mathbb{R}}
\newcommand{\NNN}{\mathbb{N}}
\newcommand{\RRRex}{\bar{\RRR}}
\newcommand{\IIi}{\mathcal{I}}
\newcommand{\HHh}{\mathcal{H}}

\newcommand{\dom}{\operatorname{dom}}
\newcommand{\epi}{\operatorname{epi}}

\newcommand{\E}{\operatorname{E}}

\newcommand{\eps}{\varepsilon}
\newcommand{\supp}{\operatorname{supp}}

\newcommand{\conv}{\operatorname{conv}}
\newcommand{\norm}[1]{\|#1\|}

\newcommand{\rank}{\operatorname{rank}}
\newcommand{\diag}{\operatorname{diag}}
\newcommand{\adj}{\operatorname{adj}}

\newcommand{\ri}{\operatorname{ri}}
\newcommand{\aff}{\operatorname{aff}}

\newcommand{\hatmu}{\bar{\mu}}
\newcommand{\Lip}{\operatorname{Lip}}

\newtheorem{theorem}{Theorem}
\newtheorem{definition}[theorem]{Definition}
\newtheorem{lemma}[theorem]{Lemma}
\newtheorem{corollary}[theorem]{Corollary}
\newtheorem{proposition}[theorem]{Proposition}

\theoremstyle{remark}

\newtheorem{example}[theorem]{Example}
\newtheorem*{assumption}{Standing Assumption}
\numberwithin{equation}{section}

\DeclareMathOperator*{\argmin}{argmin}
\DeclareMathOperator*{\argmax}{argmax}
\DeclareMathOperator*{\elim}{e-lim}
\DeclareMathOperator*{\eliminf}{e-liminf}
\DeclareMathOperator*{\elimsup}{e-limsup}

\newcommand{\w}{\vartheta}
\newcommand{\cC}{\sigma}
\newcommand{\B}{M}

\providecommand{\keywords}[1]{\noindent\emph{{Keywords and phrases:}} #1}
\providecommand{\AMScls}[1]{\noindent\emph{AMS 2010 subject classifications:} #1}

\begin{document}
\bibliographystyle{abbrv}
\renewcommand\thmcontinues[1]{cont'd}

\title{Multinomial and empirical likelihood under convex constraints: directions of recession, Fenchel~duality, perturbations}

\author{M. Grend\'ar, V. \v Spitalsk\'y}
\date{}
\maketitle

\begin{abstract}
The primal problem of multinomial likelihood maximization restricted to a convex closed subset of~the probability simplex is studied. Contrary to widely held belief, a solution of this problem may assign a positive mass to an outcome with zero count. Related flaws in the simplified Lagrange and Fenchel dual problems, which arise because the recession directions are ignored, are identified  and corrected.

A solution of the primal problem can be obtained by the PP (perturbed primal) algorithm, that is, as the limit of a sequence of solutions of perturbed primal problems. The PP algorithm may be implemented by the simplified Fenchel dual.

The results permit us to specify linear sets and data such that the empirical~likelihood-maximizing distribution exists and is the same as the multinomial likelihood-maximizing distribution. The multinomial likelihood ratio reaches, in general, a different conclusion than the empirical likelihood ratio.

Implications for minimum discrimination information, compositional data analysis, Lindsay geometry, bootstrap with auxiliary information, and Lagrange multiplier test are discussed.
\end{abstract}

\medskip\medskip 


\medskip\medskip 
\AMScls{Primary $\mathrm{62H12}$, $\mathrm{62H17}$; secondary $\mathrm{90C46}$}
%
\setcounter{tocdepth}{2}
\tableofcontents

\section{Introduction}\label{S:Intro}

Zero counts are a source of difficulties in the maximization of the multinomial likelihood
for log-linear models. A considerable literature has been devoted to this issue,
culminating in the recent papers by Fienberg and Rinaldo~\cite{Rinaldo}
and Geyer~\cite{geyer09}. In these studies, convex analysis considerations play a key role.

Less well recognized is that the zero counts also cause difficulties in the maximization
of the multinomial likelihood under linear constraints, or, in general,
when the cell probabilities are restricted to a convex closed subset
of the probability simplex; see~Section~\ref{St:P} for a formal statement
of the considered primal optimization problem~\refPR{}.
Though in this case the nature of the difficulties is different than in the log-linear case,
the convex analysis considerations are
important here as well, because they permit
developing a correct solution of~\refPR{}
-- one of the main objectives of the present work.

The problem of finding the maximum multinomial likelihood
under linear constraints dates back to, at least, Smith~\cite{Smith},
and continues through the work of Aitchison and Silvey~\cite{Aitchison},
Gokhale~\cite{Gokhale}, Klotz~\cite{Klotz}, Haber~\cite{Haber},
Stirling~\cite{Stirling}, Pelz and Good~\cite{Pelz}, Little and Wu~\cite{LittleWu},
El~Barmi and Dykstra~\cite{elbarmi-dykstra,BD2}, to the recent studies by
Agresti and Coull~\cite{Agresti2}, Lang~\cite{Lang}, and Bergsma et al.~\cite{BergsmaBook}, among others.
Linear constraints on the cell probabilities appear naturally in
marginal homogeneity models, isotonic cone models, mean response models,
multinomial-Poisson homogeneous models, and many others; cf.~Agresti~\cite{Agresti}, Bergsma et al.~\cite{BergsmaBook}.
They also arise in the context of estimating equations.

Contrary to widely-held belief,
a solution of \refPR{} may assign a positive weight to an
outcome with zero count; cf.~Theorem~\ref{T:P}, Examples~\ref{EX:nonunique-min} and
\ref{EX:positive-passive}, as well as the examples in~Section~\ref{St:BD-examples}.
This fact also affects the Lagrange and Fenchel dual problems to \refPR{}.

The restricted maximum of the multinomial likelihood defined through
the primal problem \refPR{} is not
amenable to asymptotic analysis, and the primal form is not ideal for numerical optimization.
Thus, it is common to consider the Lagrange
dual problem instead of the primal problem.
This permits the asymptotic analysis (cf.~Aitchison and Silvey~\cite{Aitchison}),
and reduces the dimension of the optimization problem,
because the number of linear constraints is usually much smaller than the
cardinality of the sample space.
Smith \cite[Sects.~6, 7]{Smith} has developed a solution of the
Lagrange dual problem, under the hidden assumption that every outcome
from the sample space appears in the sample at least once; that is,~$\nu > 0$,
where $\nu$ is the vector of the observed relative frequency of outcomes.
The same solution was later considered by several authors;
see, in particular,~Haber~\cite[p.~3]{Haber}, Little and Wu~\cite[p.~88]{LittleWu},
Lang~\cite[Sect.~7.1]{Lang}, Bergsma et al.~\cite[p.~65]{BergsmaBook}.
It remained unnoticed that, if the assumption $\nu > 0$ is
not satisfied, then the solution of Smith's Lagrange dual problem does not
necessarily lead to a solution of the primal problem
\refPR{}.

El~Barmi and Dykstra~\cite{elbarmi-dykstra} studied the maximization
of the multinomial likelihood under more general, convex set constraints,
where it is natural to replace the Lagrange duality with the Fenchel duality.
When the feasible set is defined by the linear constraints, El~Barmi and Dykstra's (BD) dual \refBD{}
reduces to Smith's Lagrange dual.
The BD-dual~\refBD{} leads to a solution of the primal \refPR{} if $\nu > 0$.
The authors overlooked that this is not necessarily the case if a zero count occurs.

Taken together, the decisions obtained from El~Barmi and Dykstra's simplified Fenchel dual \refBD{}
can be severely compromised. It is thus important to know
the correct Fenchel dual \refF{} to \refPR{}.
This is provided by Theorem~\ref{T:F}, which also characterizes the solution set of \refF{}.
It is equally important to know the conditions under which
the BD-dual \refBD{} leads to a solution of \refPR{}.
The answer is provided by Theorem~\ref{T:BD}.
An analysis of directions of recession is crucial for establishing the theorem.

Since obtaining a solution of the Fenchel dual is numerically demanding,
a simple algorithm for solving the primal is proposed.
The PP algorithm forms a sequence of \emph{p}erturbed \emph{p}rimal problems.
Theorem~\ref{T:PP}
demonstrates that the PP algorithm epi-converges to a solution of~\refPR{}.
Even stronger, pointwise convergence can be established for a linear constraint set;
see Theorem~\ref{T:PP-linear}.
The convergence theorems imply that the common practice
of replacing the zero counts by a small, arbitrary value can be
supplanted by a sequence of perturbed primal problems,
where the $\delta-$perturbed relative frequency vectors $\nu(\delta) > 0$
are such that $\lim_{\delta\searrow 0} \nu(\delta) = \nu$.
Because each $\nu(\delta)$ is strictly positive, the  PP algorithm
can be implemented through the BD-dual to the perturbed primal, by the Fisher scoring algorithm,
Gokhale's algorithm \cite{Gokhale}, or similar methods.

The findings have implications  for the empirical likelihood.
Recall that `in most settings, empirical likelihood is a multinomial likelihood on the sample';
cf.~Owen~\cite[p.~15]{Owen}.
As the empirical likelihood inner problem \refEL{} (cf.~Section~\ref{St:EL})
is a convex optimization problem,
it has its Fenchel dual formulation.
If the feasible set is linear, then the Fenchel dual to~\refEL{} is equivalent to
El~Barmi and Dykstra's dual~\refBD{} to~\refEL{}. Thanks to this
connection,
Theorem~\ref{T:BD} provides conditions under which the solution set
$\SOL_\PR$ of the multinomial likelihood primal problem \refPR{} and the solution set $\SOL_\EL$ of the
empirical likelihood inner problem \refEL{} are the same, and the maximum
$\hat L$ of the multinomial likelihood  is equal to the maximum $\hat L_\EL$
of the empirical likelihood.
Consequently:
\begin{itemize}
  \item If $C$ is an H-set or a Z-set with respect to the type $\nu$
  (for the definition, see Section~\ref{St:BD-thm}),
  the maximum empirical likelihood does not exist, though
  the maximum multinomial likelihood exists.
  The notion of H-set corresponds to the convex hull problem (cf.~Owen~\cite[Sect.~10.4]{Owen1988})
  and the notion of Z-set corresponds to the zero likelihood problem (cf.~Bergsma et al.~\cite{Bergsma}).
  By Theorem~\ref{T:BD},
  these are the only ways the empirical likelihood inner problem may fail to have a solution; cf.~Section~\ref{St:EL}.

  \item If any of conditions (i)--(iv) in Theorem~\ref{T:BD}(b) are not satisfied,
    then $\hat L_\EL < \hat L$, and the empirical likelihood may lead to
    different inferential and evidential conclusions than those suggested by the multinomial likelihood.
\end{itemize}

Fisher's \cite{Fisher} original concept of the likelihood carries the discordances
between the multinomial  and empirical likelihoods also into
the continuous \emph{iid} setting; cf.~Section~\ref{St:ELcontinuous}.

The findings also affect other methods, such as the minimum discrimination information,
compositional data analysis, Lindsay geometry of multinomial mixtures,
bootstrap with auxiliary information, and Lagrange multiplier test,
which explicitly or implicitly ignore information
about the support and are restricted to the observed outcomes.

\subsection{Organization of the paper}
The multinomial likelihood primal problem \refPR{}
and its characterization (cf.~Theorem~\ref{T:P}) are presented in Section~\ref{St:P}.
The Fenchel dual problem \refF{} to \refPR{} is introduced in Section~\ref{St:F}.
A Lagrange dual formulation of the convex conjugate (cf.~Theorem~\ref{T:cc}) serves as a ground for
Theorem~\ref{T:F}, one of the main results,  which provides a relation between the solutions of \refPR{} and \refF{}.
If the feasible set $C$ is polyhedral, a solution of \refF{} can be obtained
also from a different Lagrange dual to \refPR{}; cf.~Section~\ref{S:Fenchel-Polyhedral-C}.
A special case of the single inequality constraint is discussed in detail in Section~\ref{St:Fsingle},
where a flaw in Klotz's~\cite{Klotz} Theorem~1 is noted.
In Section~\ref{St:BD}, El~Barmi and Dykstra's \cite{elbarmi-dykstra} dual \refBD{} is recalled;
Section~\ref{St:Smith} introduces its special case, the Smith dual problem.
Theorem~2.1 of El~Barmi and Dykstra \cite{elbarmi-dykstra}, and its flaws are
presented in Section~\ref{St:BD-examples}, where they are also illustrated by simple examples.
Section~\ref{St:BD-thm} studies the scope of validity of the BD-dual \refBD{}; cf.~Theorem~\ref{T:BD}.
Sequential, active-passive dualization is proposed and analyzed in Section~\ref{St:AP}.
Perturbed primal problem \refPP{} is introduced in Section~\ref{St:PP},
where the epi-convergence of a sequence of the perturbed primals for a general, convex $C$,
and the pointwise convergence for the linear $C$ are formulated (cf.~Theorems~\ref{T:PP}, \ref{T:PP-linear}) and illustrated.
Implications of the results for the empirical likelihood method are discussed in Section~\ref{St:EL}.
A brief discussion of implications of the findings for the minimum discrimination information, compositional data analysis,
Lindsay geometry of multinomial mixtures, bootstrap with auxiliary information
and Lagrange multiplier test is contained in~Section~\ref{St:other-implications}.
Finally, Section~\ref{St:proofs} comprises detailed proofs of the results.

An \texttt{R} code and data to reproduce the numerical examples can be found
in~\cite{Rcode}.

\section{Multinomial likelihood primal problem \refPR}\label{St:P}

Let $\A$ denote a finite \emph{alphabet} (sample space)
consisting of $m$ \emph{letters} (outcomes) and $\sx{\A}$
denote the \emph{probability simplex}
\begin{equation*}
    \sx{\A} \triangleq \left\{q\in\RRR^m:\ q\ge 0, \sum q = 1 \right\};
\end{equation*}
identify $\RRR^m$ with $\RRR^\A$.
Suppose that $(n_i)_{i\in\A}$ is a \emph{realization} of the \emph{closed
multinomial distribution} $\mathrm{Pr}((n_i)_{i\in\A}; n, q) = n!\prod{q_i^{n_i}/n_i!}$
with parameters $n\in\NNN$ and $q=(q_i)_{i\in\A}\in\sx{\A}$.
Then the \emph{multinomial likelihood kernel} $L(q) = L_\nu(q) \triangleq e^{-n\,\ell(q)}$,
where $\ell=\ell_\nu:\sx\A\to \RRRex$, \emph{Kerridge's inaccuracy} \cite{Kerridge}, is
\begin{equation}\label{EQ:ell(q)}
    \ell(q) \triangleq  -\sp{\nu}{\log q},
\end{equation}
and $\nu\triangleq (n_i/n)_{i\in\A}$ is the \emph{type}
(the vector of the relative frequency of outcomes).
The conventions $\log 0=-\infty$, $0\cdot (-\infty)=0$ apply; $\RRRex$ denotes
the extended real line $[-\infty,\infty]$ and $\sp{a}{b}$ is
the \emph{scalar product} of $a$, $b\in\RRR^{m}$.
Functions and relations on vectors are taken component-wise;
for example, $\log q=(\log q_i)_{i\in\A}$.
For $x\in\RRR^m$, $\sum x$ is a shorthand for $\sum_{i\in\A} x_i$.

Consider the problem \refPR{} of minimization of $\ell$,
restricted to a convex closed set $C\subseteq\sx{\A}$:
\begin{equation}\label{EQ:intro-primal}
    \hat{\ell}_{\PR} \triangleq \inf_{q\in C} \ell(q),
    \qquad
    \SOL_{\PR} \triangleq \{\hat{q}\in C:\ \ell(\hat q)=\hat\ell_\PR\}.
  \tag{$\PRbig$}
\end{equation}
The goal is to find the \emph{solution set} $\SOL_{\PR}$ as well as the infimum
$\hat\ell_\PR$ of the objective function $\ell$ over $C$.
The problem \refPR{} will be called the \emph{multinomial likelihood primal problem},
or \emph{primal}, for short.

Special attention is paid to the class of \emph{polyhedral} feasible sets $C$
\begin{equation}\label{EQ:C-polyhedral}
   C=\{q\in\sx{\A}:  \sp{q}{u_h}\le 0 \text{ for } h=1,2,\dots,r\},
\end{equation}
or to its subclass of sets $C$ given by (a finite number of) linear equality constraints
\begin{equation}\label{EQ:C-linear}
   C=\{q\in\sx{\A}:  \sp{q}{u_h} = 0 \text{ for } h=1,2,\dots,r\},
\end{equation}
where $u_h$ are vectors from $\RRR^m$.
These feasible sets are particularly interesting from the applied point of view and
permit to establish stronger results.

Without loss of generality
it is assumed that $\A$ is the \emph{support} $\supp(C)$ of $C$,
that is, for every $i\in\A$ there is $q\in C$ with $q_i>0$;
in other words, the structural zeros (cf.~Baker et al.~\cite[p.~34]{Baker}) are excluded.
Due to the convexity of $C$
this is equivalent to the existence of $q\in C$ with $q>0$.
Under this assumption, Theorem~\ref{T:P} gives a basic characterization of
the solution set of \refPR{}.
Before stating it, some useful notions are introduced.

\begin{definition} For a type $\nu$ (or, more generally,
for any $\nu\in\sx\A$), the \emph{active} and
\emph{passive alphabets} are
\begin{equation*}
    \Aa=\Aa_\nu \triangleq  \{i\in\A: \nu_i>0\}
    \qquad\text{and}\qquad
    \Ap=\Ap_\nu \triangleq  \{i\in\A: \nu_i=0\}.
\end{equation*}
The elements of $\Aa,\Ap$ are called \emph{active, passive letters}, respectively.
\end{definition}

Put $m_a\triangleq\card \Aa>0$, $m_p\triangleq\card\Ap\ge 0$.
Let $\pi^a:\RRR^m\to \RRR^{m_a}$, $\pi^p:\RRR^m\to \RRR^{m_p}$ be the
natural projections; identify $\RRR^{m_a}$ with $\RRR^{\Aa}$ and $\RRR^{m_p}$ with $\RRR^{\Ap}$.
Note that if $\Ap=\emptyset$ then $m_p=0$ and
$\RRR^{m_p}=\{0\}$.
For $x\in\RRR^m$, $x^a$ and $x^p$ are the shorthands for $\pi^a(x)$ and $\pi^p(x)$, respectively.
(If no ambiguity can occur, the elements of $\RRR^{m_a},\RRR^{m_p}$ will be denoted
also by $x^a,x^p$.)
Identify $\RRR^m$ with $\RRR^{m_a}\times\RRR^{m_p}$, so that it is possible to write $x=(x^a,x^p)$
for every $x\in\RRR^m$.
Finally, for a subset $M$ of $\RRR^m$ and $x\in M$ let
\begin{equation*}
    M^a \triangleq \pi^a(M)
    \qquad\text{and}\qquad
    M^a(x^p) \triangleq \{x^a\in\RRR^{m_a}: (x^a,x^p)\in M\}
\end{equation*}
be the \emph{active projection} and the \emph{$x^p$-slice} of $M$; analogously define $M^p$ and $M^p(x^a)$.

\begin{theorem}[Primal problem]\label{T:P}
Let $\nu\ge 0$ be from $\sx\A$.
Let $C$ be a convex closed subset of $\sx\A$ with support $\A$.
Then $\hat\ell_\PR$ is finite, $\SOL_{\PR}$ is compact and there is $\hat q^a_\PR\in C^a$,
$\hat q^a_\PR>0$, such that
\begin{equation*}
    \SOL_{\PR} = \{\hat{q}^a_\PR\} \times C^p(\hat q^a_\PR).
\end{equation*}
Moreover,
\begin{enumerate}
  \item[(a)]
    If $\sum \hat{q}^a_\PR=1$ then $C^p(\hat q^a_\PR)=\{0^p\}$ and $\SOL_{\PR}=\{(\hat{q}^a_\PR, 0^p)\}$ is a singleton.
  \item[(b)]
    If $\sum \hat{q}^a_\PR<1$ then $0^p\not\in C^p(\hat q^a_\PR)$, and $\SOL_\PR$ is a singleton
     if and only if the $\hat q^a_\PR$-slice $C^p(\hat q^a_\PR)$ of $C$ is a singleton.
\end{enumerate}
\end{theorem}

Thus the primal has always a solution $\hat q_{\PR}$.
Its active coordinates $\hat q^a_\PR$ are unique
and the passive coordinates $\hat q^p_\PR$ are arbitrary such that $\hat q_\PR\in C$.
It is worth stressing that $C^p(\hat q^a_\PR)$ need not be equal to $\{0^p\}$,
that is, a solution of~\refPR{} may put positive mass to passive letter(s).
The following couple of simple examples illustrates the points; see also the examples in~Section~\ref{St:BD-examples}.
Hereafter $X$ denotes a random variable supported on $\A$.

\begin{example}\label{EX:nonunique-min}
Take $\A=\{-1,0,1\}$ and $C=\{q\in\sx{\A}:\ \E_q(X^2)=\sum_{i\in \A  } i^2 q_i=1/2\}$.
Let $\nu=(0,1,0)$, so that $\Aa=\{0\}$ and $\Ap=\{-1,1\}$.
Then $C=\{q\in\sx{\A}:\ q_0=1/2\}$,
the minimum of $\ell$ over $C$ is $\hat\ell_\PR{}=\log2$, and
$\SOL_\PR=C$.
Here, $\hat q^a_\PR=1/2$ is (trivially) unique and $C^p(\hat q^a_\PR)=\{(a,1/2-a): a\in[0,1/2]\}$.
\end{example}

\begin{example}\label{EX:positive-passive}
Motivated by Wets~\cite[p.~88]{Wets}, let $\A = \{1,2,3\}$, $C = \{q\in\sx{\A}: q_1 \le q_2 \le q_3\}$ and $\nu = (0, 1, 0)$.
Then $\Aa = \{2\}$, $\Ap = \{1, 3\}$. Since $\SOL_\PR = \{(0, 1, 1)/2\}$, the positive weight $1/2$
is assigned to the passive, unobserved letter $3$.
\end{example}

The Fisher scoring algorithm which is commonly used to solve \refPR{}
when $C$ is a linear set may fail to converge when the zero counts
are present; cf.~Stirling~\cite{Stirling}.
Other numerical methods, such as the augmented Lagrange multiplier methods,
which are used to solve the convex optimization problem under polyhedral
and/or linear $C$ may have difficulties to cope with large $m$. Thus, it is desirable to approach \refPR{}
also from another direction.

\section{Fenchel dual problem \refF{} to \refPR}\label{St:F}

Consider the \emph{Fenchel dual problem}~\refF{} to the primal \refPR{}:
\begin{equation}\label{EQ:intro-fenchel}
    \hat{\ell}_\F \triangleq \inf_{y\in C^\ast} \ell^\ast(-y),
    \qquad
    \SOL_{\F} \triangleq \{\hat{y}\in C^\ast:\ \ell^\ast(-\hat y)=\hat\ell_\F\},
  \tag{$\Fbig$}
\end{equation}
where
\begin{equation*}
    C^\ast \triangleq \{y\in\RRR^m: \sp{y}{q}\le 0 \text{ for every } q\in C\}
\end{equation*}
is the \emph{polar cone} of $C$ and
\begin{equation*}
    \ell^\ast:\RRR^m\to \RRRex,\qquad
    \ell^\ast(z) \triangleq \sup_{q\in\sx\A} (\sp{q}{z} - \ell(q))
\end{equation*}
is the \emph{convex conjugate} of $\ell$ (in fact, the convex conjugate
of $\tilde\ell:\RRR^m\to\RRRex$ given by
$\tilde\ell(x)=\ell(x)$ for $x\in\sx\A$ and $\tilde\ell(x)=\infty$ otherwise).

The Fenchel dual \refF{} is often more tractable than the primal \refPR{},
in particular when $C$ is given by linear equality and/or inequality constraints.
This is the case of the models for contingency tables, mentioned in Introduction.
Also an \emph{estimating equations} model $C_\Theta$ leads to a linear feasible set $C_\theta$,
when $\theta\in\Theta$ is fixed.
The model is $C_\Theta \triangleq \bigcup_{\theta\in\Theta} C_\theta$, where
\begin{equation}\label{EQ:EE}
  C_\theta \triangleq \bigcap_{h=1}^r \left\{q\in\sx{\A}:
  \sp{q}{u_h(\theta)} = 0\right\},
\end{equation}
and $u_h: \Theta \rightarrow \RRR^m$, $h=1,2,\dots,r$, are the
\emph{estimating functions}. There $\theta\in\Theta\subseteq\RRR^d$ and $d$
need not be equal to $r$.
Since~$r$ is usually much smaller than~$m$, \refF{} may be
easier to solve numerically than \refPR{}.

Observe that the convex conjugate itself is defined through an optimization
problem, the \emph{convex conjugate primal problem} (\emph{cc-primal}, for short),
whose solution set is
\begin{equation}\label{EQ:F-cc-finite-sol}
 S_{\cc}(z) \triangleq \{q\in\sx{\A}: \sp{q}{z} - \ell(q) = \ell^\ast(z)\}.
\end{equation}
The structure of $S_{\cc}(z)$ is described by Proposition~\ref{P:cc-SOL}.
The conjugate $\ell^\ast$ can be evaluated by means of the Lagrange duality, where
the Lagrange function is
$$
 K_z(x,\mu)=\sp{x}{z} - \ell(x) -\mu\left(\sum x-1 \right).
$$
It holds that (cf.~Lemma~\ref{L:F-cc-via-Lagrange-dual})
$$
 \ell^\ast(z) = \inf_{\mu\in\RRR} k_z(\mu),
 \qquad\text{where}\quad
 k_z(\mu) = \sup_{x\ge 0} K_z(x,\mu).
$$

For every $\mu\in\RRR$ and $a,b\in\RRR^{m_a}$, $a>0$, $b> -\mu$, define
\begin{equation*}
  I_\mu(a\,\|\,b) \triangleq \sp{a}{\log\frac{a}{\mu + b}}.
\end{equation*}

\begin{theorem}[Convex conjugate by Lagrange duality]\label{T:cc}
Let $\nu\in\sx{\A}$ and $z\in\RRR^m$. Then
\begin{equation*}
    \ell^\ast(z) = -1 + \hat\mu(z) + I_{\hat\mu(z)}(\nu^a\,\|\,-z^a),
\end{equation*}
where
\begin{equation}\label{EQ:mu-hat}
\hat\mu(z)\triangleq \max\{ \hatmu(z^a), \max(z^p) \},
\end{equation}
and $\hatmu(z^a)$ is the unique solution of
\begin{equation}\label{EQ:mu-bar}
    \sum \frac{\nu^a}{\mu-z^a} = 1,
    \qquad
    \mu\in(\max(z^a),\infty).
\end{equation}
\end{theorem}

The key point is that the $\hatmu(z^a)$, which solves (\ref{EQ:mu-bar}),
is not always the $\hat\mu(z)$ which minimizes $k_z(\mu)$.
They are the same if and only if $\hatmu(z^a)\ge\max(z^p)$.
As it will be seen in Theorem~\ref{T:F}, if $z\in\SOL_\F$ then
this inequality decides whether the solution of primal~\refPR{} is supported
only on the active letters, or some probability mass is placed also to the passive letter(s).

Theorem~\ref{T:cc} serves as a foundation for Theorem~\ref{T:F}, which states
the relation between the Fenchel dual and the primal.

\begin{theorem}[Relation between \refF{} and \refPR{}]\label{T:F}
    Let $\nu,C,\hat{q}^a_\PR$ be as in Theorem~\ref{T:P}.
    Then
    \begin{equation*}
        \hat\ell_\F = - \hat\ell_\PR,
    \end{equation*}
    $\SOL_\F$ is a nonempty convex compact set and
    $\hat\mu(-\hat y_\F)=1$ for every $\hat y_\F\in\SOL_\F$.
    Moreover, if we put
    \begin{equation}\label{EQ:yhat-F}
        \hat{y}^a_\F \triangleq \frac{\nu^a}{\hat q^a_\PR}-1^a,
    \end{equation}
    then $(\hat{y}^a_\F, -1^p)\in\SOL_\F$ and the following hold:
\begin{enumerate}
  \item[(a)] If $\sum \hat q^a_\PR=1$ then $\hatmu(-\hat y^a_\F)=1$ and
    \begin{equation*}
        \SOL_\F = \{\hat{y}^a_\F\} \times \{y^p\in C^{\ast p}(\hat{y}^a_\F): \min(y^p)\ge -1\}.
    \end{equation*}
  \item[(b)] If $\sum \hat q^a_\PR<1$ then $\hatmu(-\hat y^a_\F)<1$ and
    \begin{equation*}
        \SOL_\F = \{\hat{y}^a_\F\} \times \{y^p\in C^{\ast p}(\hat{y}^a_\F): \min(y^p)= -1\}.
    \end{equation*}
\end{enumerate}
\end{theorem}

Thus, in the active coordinates, the solution of \refF{} is unique
and is related to~$\hat q^a_\PR$ by (\ref{EQ:yhat-F}).
Together with Theorem~\ref{T:P} this yields
$$
  \SOL_\PR =
  \left\{\frac{\nu^a}{1 + \hat y^a_\F}\right\}
  \times
  C^p\left(\frac{\nu^a}{1 + \hat y^a_\F}\right).
$$

The structure of the solution set of \refF{} in the passive letters is determined
by the relation of $\hatmu(-\hat y^a_\F)$ to $1=\hat\mu(-\hat y_\F)$.

In the case (a), $\hatmu(-\hat y^a_\F)=\hat\mu(-\hat y_\F)$, and the passive projections $\hat y^p_\F$
 of the solutions satisfy $\hat y^p_\F \ge -1$.
Then it suffices to solve the primal
solely in the active letters and $\SOL_\PR$ is a singleton.
This happens for instance if $\nu>0$.

In the case (b), $\hatmu(-\hat y^a_\F)<\hat\mu(-\hat y_\F)$, and
$\hat y^p_\F$ satisfy $\min(\hat y^p_\F) = -1$.
Then every solution of the primal assigns the positive probability to at least one passive letter.

To sum up, the Fenchel dual problem, once solved, permits to find $\hat q^a_\PR$
through~(\ref{EQ:yhat-F}) and this way it incorporates $C^\ast$ into $\hat q^a_\PR$.
In the case of linear or polyhedral $C$ the dual may reduce dimensionality of the
optimization problem, yet the numerical solution of \refF{} is somewhat hampered
by the need to obey (\ref{EQ:mu-hat}). Moreover, $\hat q^p_\PR$ remains to be found;
in this respect see Proposition~\ref{P:F-Scc-SP}.

\subsection{Polyhedral and linear $C$}\label{S:Fenchel-Polyhedral-C}
In the polyhedral case (\ref{EQ:C-polyhedral})
\begin{equation*}
   C=\{q\in\sx{\A}:  \sp{q}{u_h}\le 0 \text{ for } h=1,2,\dots,r\},
\end{equation*}
a solution of the Fenchel dual \refF{} can also be obtained from the saddle points
of the following Lagrange function
\begin{equation*}
    L(q,\alpha) \triangleq -\sum_h \alpha_h \sp{q}{u_h} - \ell(q)
    \qquad
    (q\in\sx{\A}, \alpha\in\RRR^r_+).
\end{equation*}
Indeed, it holds that
\begin{equation*}
    \hat\ell_\F
    = \adjustlimits \inf_{\alpha\ge 0} \sup_{q\in\sx{\A}} L(q,\alpha)
    = \adjustlimits \sup_{q\in\sx{\A}} \inf_{\alpha\ge 0} L(q,\alpha)
    = -\hat\ell_\PR.
\end{equation*}
Hence $(\hat q,\hat \alpha)$ is a saddle point of $L(q, \alpha)$ if and only if
$\hat q\in\SOL_\PR$ and $\sum_h \hat\alpha_h u_h \in\SOL_\F$; cf. Bertsekas~\cite[Prop.~2.6.1]{bertsekas}.

The vectors from $\SOL_\F$ of the form $\sum_h \hat\alpha_h u_h$
will be called the \emph{base solutions} of~\refF{}.
From the Farkas lemma (cf.~Bertsekas~\cite[Prop.~3.2.1]{bertsekas})
and the monotonicity of $\ell^\ast$ (Lemma~\ref{L:cc-nondecreasing})
it follows that, for a polyhedral $C$,
there always exists a base solution,
and every solution of \refF{}
is a sum of a base solution and a vector from
$\RRR^m_{-}\triangleq\{z\in\RRR^m:z\le 0\}$; that is,
\begin{equation*}
    \SOL_\F = \left\{
      \hat y_\F{} + (0^a,z^p):
      \hat y_\F{} \text{ is a base solution, }
      z^p\le 0, \  \hat y_\F^p + z^p \ge -1
    \right\}.
\end{equation*}
There may exist many base solutions.
However, if $u_1^a,\dots,u_r^a$ are linearly independent
then the base solution is unique,
since then the system of equations
$\sum_h\hat\alpha_h u_h^a=\hat y_\F^a$ has a unique solution $\hat\alpha$.

Analogous claim holds for a linear $C$ (cf.~(\ref{EQ:C-linear})),
just in this case $\alpha\in\RRR^r$, by the Farkas lemma.

\subsection{Single inequality constraint and Klotz's Theorem~1}\label{St:Fsingle}
To illustrate the base solution in connection with Theorem~\ref{T:F}, consider
$$
  C = \{q\in\sx{\A}: \sp{q}{u} \le 0\},
$$
given by a single inequality constraint.
By the Farkas lemma, $C^\ast=\{\alpha u:\alpha\ge 0 \}+\RRR^m_-$.
If $u^p\ge 0$ the case (a) of Theorem~\ref{T:F} applies
(for if not, there is a base solution $\hat \alpha_\F u$ and, by Theorem~\ref{T:F}(b),
$\min(\hat\alpha_\F u^p)$ should be $-1$; this is not possible
since $\min(\hat\alpha_\F u^p)\ge 0$).

Assume that $\min(u^p)<0$ and the case (b) of Theorem~\ref{T:F} applies.
Take any base solution $\hat \alpha_\F u$. Then, by Theorem~\ref{T:F}(b),
$\min(\hat\alpha_\F u^p)=-1$; so
\begin{equation}\label{E:qa-base-F}
    \hat\alpha_\F = -\frac{1}{\min(u^p)}
    \qquad\text{and}\qquad
    \hat q^a_\PR = \frac{\nu^a}{1 - \frac{u^a}{\min(u^p)}}  \,.
\end{equation}
Further, by Theorem~\ref{T:P}, $\hat q^p_\PR$ is arbitrary from $C^p(\hat q^a_\PR)$.
If $u^p$ attains the minimum at a single letter, then the solution of the primal \refPR{} is always unique.

To sum up, the case (b) of Theorem~\ref{T:F} happens if and only if
\begin{equation}\label{E:qa-base-F-condition-b}
    \min(u^p)<0
    \qquad\text{and}\qquad
    \sum \frac{\nu^a}{1 - \frac{1}{\min(u^p)} u^a}  <1.
\end{equation}

The primal problem \refPR{} with this $C$ was considered by Klotz \cite{Klotz}.
Theorem~1 of~\cite{Klotz} asserts that the solution of \refPR{} takes the
form
$$
  \hat q_K \triangleq \frac{\nu}{1 + \hat\alpha_K u}
$$
if $\sp{\nu}{u} > 0$
and $\min(u^a) < 0$; see Klotz's condition (3.1b).
There, $\hat\alpha_K$ is the unique root of
\begin{equation*}
  \sum \frac{\nu}{1 + \alpha u} = 1
  \qquad\text{such that}\quad
  \alpha \in (0, - 1/\min(u^a)).
\end{equation*}
Thus, under Klotz's condition (3.1b), the solution of \refPR{} should assign zero weight
to any passive letter. This is not the case, as the following example demonstrates.

\begin{example}[Base solution of \refF{}, one inequality constraint]\label{E:base-solution-F}
Take $\A=\{-2,-1,0,$ $1,2\}$, $u = (-2,-1,0,1,2)$ and $\nu = (0, 3, 0, 0, 7)/10$.
Here Klotz's condition (3.1b) is satisfied and an easy computation yields
$\hat\alpha_K=11/20$; thus $\hat q_K = (0, 2, 0, 0, 1)/3$ and $\ell(\hat q_K) = 0.890668$.
However,  $\hat q_\PR = (1, 12, 0, 0, 7)/20$, so a positive mass is placed
also to the passive letter $-2$; $\ell(\hat q_\PR) = 0.888123$.
Since $\min(u^p)=-2$ and $\sum \hat q^a_\PR < 1$,
the condition (\ref{E:qa-base-F-condition-b}) is satisfied and
$\hat q^a_\PR$ is just that given by (\ref{E:qa-base-F}).
\end{example}

In a similar way it is possible to analyze a single equality constraint.
For more than one equality/inequality constraint the condition guaranteeing that $\sum \hat q^a_\PR < 1$
cannot be given in such a simple form as (\ref{E:qa-base-F-condition-b}).

\section{El~Barmi-Dykstra dual problem \refBD{} to \refPR}\label{St:BD}

El~Barmi and Dykstra~\cite{elbarmi-dykstra} consider
a \emph{simplified Fenchel  dual problem}~\refBD{} (the \emph{BD-dual}, for short)
\begin{equation}\label{EQ:intro-BD}
    \hat{\ell}_{\BD} \triangleq
      \inf_{y\in C^\ast} \bd(y),
    \qquad
    \SOL_{\BD} \triangleq \{\hat{y}\in C^\ast:
      \bd(\hat y) = \hat\ell_\BD\},
  \tag{$\BDbig$}
\end{equation}
where
\begin{equation*}
    \bd:\RRR^m\to \RRRex,\qquad
    \bd(y) \triangleq
    \begin{cases}
        I_1(\nu^a\,\|\,y^a)  =\sp{\nu^a}{\log\dfrac{\nu^a}{1+y^a}}
        &\text{if } y^a>-1,
    \\
        \infty &\text{otherwise},
    \end{cases}
\end{equation*}
will be called the \emph{BD conjugate} of $\ell$.
Note that the BD-dual \refBD{}
is easier to solve than~\refF{}, as in the former $\hat\mu$ is fixed to $1$.

If $C$ is polyhedral then, in analogy with the concept of the base solution of~\refF{},
the vectors from $\SOL_\BD$  of the form
$\sum \hat\alpha_{\BD{},h} u_h$
are called the \emph{base solutions} of~\refBD{}.
As above, every solution of \refBD{} can in this case be written as a sum of a base solution and
a vector from $\RRR^m_-$.

\subsection{Smith dual problem}\label{St:Smith}

By the Farkas lemma,
for the feasible set
$C = \{q\in\sx{\A}: \sp{q}{u_h} = 0, h = 1,2,\dots,r\}$ given
by $r$ linear equality constraints, the polar cone is
${C}^\ast = \{y = \sum_h \alpha_h u_h: \alpha\in\RRR^r\} + \RRR^m_{-}$.
Then the BD-dual becomes
\begin{equation}\label{EQ:Smith}
  \inf_{\alpha\in\RRR^r}\, I_1\left(\nu^a\,\Bigl\|\,\sum_h \alpha_h u^a_h \right),
\end{equation}
which is (equivalent to) the \emph{simplified Lagrange dual problem} (the \emph{Smith dual}, for short)
considered by Smith~\cite[Sect.~6]{Smith} and many other authors; see, in particular,
Haber~\cite[p.~3]{Haber}, Little and Wu~\cite[p.~88]{LittleWu}, Lang~\cite[Sect.~7.1]{Lang},
and Bergsma et al.~\cite[p.~65]{BergsmaBook}.
It is worth noting that (4.2) is an unconstrained optimization problem.

\section{Flaws of BD-dual \refBD{}}\label{St:BD-examples}
In ~\cite[Thm.~2.1]{elbarmi-dykstra}
El~Barmi and Dykstra state
the following relationship between~\refPR{} and~\refBD{}.
\begin{theorem}[Theorem~2.1 of \cite{elbarmi-dykstra}]\label{T:BD2.1}
Let $C$ be a convex closed subset of $\sx{\A}$.
\begin{itemize}
\item[] ($\BDbig \rightarrow \PRbig$): If $\hat\ell_\BD$ is finite, then
$\SOL_\BD$ is nonempty and
  $\hat\ell_\BD  = -\hat\ell_\PR$.
Moreover, for every $\hat y_\BD\in \SOL_\BD$,
 $\tilde q_\BD \triangleq \frac{\nu}{1 + \hat y_\BD}$
 belongs to
 $\SOL_\PR$.

\item[] ($\PRbig \rightarrow \BDbig$): If $\hat\ell_\PR$ is finite, then
  $\hat\ell_\BD = - \hat\ell_\PR$.
Moreover, if $\hat q_\PR\in\SOL_\PR$, then
  $\tilde y_\BD \triangleq \frac{\nu}{\hat q_\PR} - 1$
  belongs to
  $S_\BD$. (There, $0/0 = 0$ convention applies).
\end{itemize}
\end{theorem}

Though $\hat\ell_\PR$ is always finite, $\hat\ell_\BD$ may be infinite,
leaving the solution set $\SOL_\PR$ inaccessible through ($\BDbig \rightarrow \PRbig$) of
\cite[ Thm.~2.1]{elbarmi-dykstra}.
Moreover, the claims ($\BDbig \rightarrow \PRbig$) and ($\PRbig \rightarrow \BDbig$) of the theorem
are not always true.
In fact,  there are
three possibilities:
\begin{enumerate}
  \item $\hat{\ell}_{\BD}=-\infty$;
  \item $\hat{\ell}_{\BD}$ is finite, but there is a \emph{BD-duality gap}, that is,
        $\hat{\ell}_{\BD} < - \hat\ell_\PR$;
  \item $\hat{\ell}_{\BD}= - \hat{\ell}_{\PR}$ (\emph{no BD-duality gap}).
\end{enumerate}
To illustrate them, we give below
six simple examples: one where the BD-duality works and the other five
where it fails.
In Examples~\ref{E:no-gap}, \ref{E:E}, \ref{E:Z}, \ref{E:gap} the set~$C$ is linear,
whereas Example~\ref{E:nonlin} presents a nonlinear $C$, and
in Example~\ref{E:Wets} the set $C$ is defined by linear inequalities.

\begin{example}[No BD-duality gap]\label{E:no-gap}
Let $\A = \{-1, 0, 1\}$ and $C = \{q\in\sx{\A}: \mathrm{E}_q X = 0\}$;
that is, $C$ is given by (\ref{EQ:C-linear}) with $u = (-1, 0, 1)$.
Let $\nu = (1, 0, 1)/2$. Thus $\Aa = \{-1,1\}$, $\Ap = \{0\}$,
and $C^a(0^p)=\{(1,1)/2\}$. By the Farkas lemma,
$C^\ast = \{y = \alpha u: \alpha\in\mathbb{R}\} + \RRR^3_-$.
For the considered optimization problems, the following hold:
\begin{itemize}
  \item \refPR:~$\SOL_\PR = \{(1, 0, 1)/2\}$ and $\hat\ell_\PR = \log 2$.
  \item \refF:~Since $\hat\alpha_\F = 0$, the base solution of \refF{} is $(0,0,0)$ and
      $\hat\ell_\F = -\log 2$. Further, (\ref{EQ:yhat-F}) implies that
      $\hat q^a_\PR = \nu^a/(1 + \hat y^a_\F)$; thus, $\hat q^a_\PR = \nu^a$.
\end{itemize}
In this setting, no BD-duality gap occurs:
\begin{itemize}
  \item \refBD$\rightarrow$\refPR:~$\bd(\alpha u) \propto -1/2[\log(1-\alpha)
    + \log(1+\alpha)]$.
      Thus $\hat\alpha_\BD=0$,
      $\hat\ell_\BD = -\log 2=-\hat\ell_\PR$.
      Since $\tilde q_\BD = (1, 0, 1)/2$, it indeed solves \refPR.
  \item \refPR$\rightarrow$\refBD:~$\hat\ell_\BD = -\hat\ell_\PR$ by the previous case.
      Since the base solution of \refBD{}
      is $(0,0,0)$ and $\tilde y_\BD = (0, 0, 0)$,  $\tilde y_\BD \in \SOL_\BD$.
\end{itemize}

\end{example}

\begin{example}[name=H-set, label=E:E]
Let $\A$, $C$ be as in Example~\ref{E:no-gap} and $\nu = (1, 0, 0)$. Thus $\Aa = \{-1\}$, $\Ap = \{0, 1\}$,
and $C^a(0^p)=\emptyset$.
\begin{itemize}
  \item \refPR: $\SOL_\PR = \{(1, 0, 1)/2\}$ and $\hat\ell_\PR = \log 2$.
  \item \refF:~Since $\hat\alpha_\F = -1$, the base solution of \refF{} is $(1,0,-1)$ and
      $\hat\ell_\F = -\log 2$. Thus, from (\ref{EQ:yhat-F}) it follows that $\hat q^a_\PR = (1/2)$.
  \item \refBD$\rightarrow$\refPR:~$\bd(\alpha u) \propto -\log(1 - \alpha)$, which does
      not have finite infimum.
  \item \refPR$\rightarrow$\refBD:
      $\hat\ell_\BD$ is not finite and $\SOL_\BD=\emptyset$.
      Thus $\tilde y_\BD = (1,0,-1)\not\in\SOL_\BD$.
\end{itemize}
\end{example}

\begin{example}[nonlinear $C$, H-set]\label{E:nonlin}
Let $\A = \{2, 3, 4\}$, $C = \{q\in\sx{\A}: \sum_{i\in\A} i\,q^2_i \le 1\}$,
and $\nu = (1, 0, 0)$.
Thus $\Aa = \{2\}$, $\Ap = \{3, 4\}$, and $C^a(0^p)=\emptyset$.

\begin{itemize}
  \item \refPR:~Write $C=\bigcup_{k\in[0,1]}C_k$, where $C_k\triangleq \{q\in\sx{\A}: \sum_{i\in\A} i\,q^2_i =k \}$.
                On $C_k$, $7q_3 = 4 - 4 q_2 \pm \sqrt{-12+24q_2-26q_2^2+7k}$  and the
                nonnegativity of the term under the square-root implies that
                $q_2$ should belong to the interval $[6/13 - 1/26\sqrt{-168+182k}, 6/13 + 1/26\sqrt{-168+182k})]$.
                This in turn implies that $\hat q_2(k) = 6/13 + 1/26\sqrt{-168+182k}$.
                The function $\hat q_2(k)$ attains its maximum at $k=1$, for which $\hat q_2 = 0.6054$.
                The other two elements of $\hat q_\PR$ are determined uniquely.
                Thus, $\hat q_\PR = (0.6054, 0.2255, 0.1691)$.
  \item \refBD:~As shown above, $q_2\le (12+\sqrt{24})/26\triangleq c<1$ for every $q\in C$.
                Put $d=c/(1-c)$. Then, for every $\alpha\ge 0$, $y\triangleq\alpha (1,-d,-d)\in C^\ast$ and
                $\bd(y)=-\log(1+\alpha)$. Hence the BD-dual problem has $\hat\ell_\BD = -\infty$.
\end{itemize}

\end{example}

\begin{example}[monotonicity, H-set]\label{E:Wets}  
Let $\A = \{1,2,3\}$, $C = \{q\in\sx{\A}: q_1 \le q_2 \le q_3\}$, and $\nu = (0, 1, 0)$,
as in Example~\ref{EX:positive-passive}.
Thus $\Aa = \{2\}$, $\Ap = \{1, 3\}$, and $C^a(0^p)=\emptyset$.
The polyhedral set $C$ is given by (\ref{EQ:C-polyhedral}) with $u_1 = (1, -1, 0)$, $u_2 = (0, 1, -1)$.
\begin{itemize}
  \item \refPR{}: $\SOL_\PR = \{(0, 1, 1)/2\}$ and $\hat\ell_\PR = \log 2$.
  \item \refF{}:~Since $\hat\alpha_\F = (0, -1)$, the base solution is $\hat y_\F = (0, -1, 1)$.
                Thus $\hat q^a_\PR = 1/2$ by~(\ref{EQ:yhat-F}).
  \item \refBD{} $\rightarrow$ \refPR{}:
    For every $\alpha\ge 0$, $\alpha u_2\in C^\ast$ and $\bd(\alpha u_2)=-\log(1+\alpha)$.
    Thus $\hat\ell_\BD=-\infty$.
  \item \refPR{} $\rightarrow$ \refBD{}: Since $\SOL_\BD=\emptyset$,
    $\tilde y_\BD = (0, 1/2, 0)\not\in\SOL_\BD$.
\end{itemize}
\end{example}

\begin{example}[name=Z-set, label=E:Z]
Motivated by Example~4 of Bergsma et al.~\cite{Bergsma},
let $\A$, $C$ be as in Example~\ref{E:no-gap} and let $\nu = (1, 1, 0)/2$.
Thus $\Aa = \{-1, 0\}$, $\Ap = \{1\}$, and $C^a(0^p)=\{(0,1)\}$.
\begin{itemize}
  \item \refPR:~$\SOL_\PR = \{(1, 2, 1)/4\}$ and $\ell_\PR = (1/2) \log 8$.
  \item \refF:~Since $\hat\alpha_\F = -1$,  the base solution of \refF{} is $(1,0,-1)$
  and $\hat\ell_\F = -(1/2)\log 8$. Thus, by (\ref{EQ:yhat-F}), $\hat q^a_\PR = (1, 2)/4$.
  \item \refBD$\rightarrow$\refPR:~$\bd(\alpha u) \propto -(1/2)\log(1 - \alpha)$,
    hence $\hat\ell_\BD=-\infty$.
  \item \refPR$\rightarrow$\refBD:
   Since $\SOL_\BD=\emptyset$, $\tilde y_\BD = (1,0,-1)\not\in\SOL_\BD$.
\end{itemize}
\end{example}

Observe that Theorem~2.1 of \cite{elbarmi-dykstra} implies that $\tilde q^p_\BD = 0^p$,
provided that $\hat\ell_\BD$ is finite.
However, $C^p(\hat q^a)$ may be different than $\{0^p\}$;
in such a case $\hat q^p_\PR$ has a strictly positive coordinate and
the BD-duality gap occurs.

\begin{example}[name=BD-duality gap, label=E:gap]
Let $\A = \{-1, 1, 10\}$, $C = \{q\in\sx{\A}: \mathrm{E}_q X = 0\}$, so that $u = (-1, 1, 10)$,
and $\nu = (3, 2, 0)/5$. Thus $\Aa = \{-1,1\}$, $\Ap = \{10\}$, and $C^a(0^p)=\{(1,1)/2\}$.
\begin{itemize}
  \item \refPR:~$\SOL_\PR = \{(54, 44, 1)/99\}$ and $\hat\ell_\PR = 0.6881$.

  \item \refF:~Since $\hat\alpha_\F = -1/10$, the base solution of \refF{}
     is $(1,-1,-10)/10$ and $\hat\ell_\F = -0.6881$.
     From (\ref{EQ:yhat-F}), it follows that $\hat q^a_\PR = (54, 44)/99$.

  \item \refBD$\rightarrow$\refPR:~$\bd(\alpha u) \propto -[\nu_{-1}\log(1 - \alpha)
    + \nu_1\log(1+\alpha)]$.
    Since $\hat\alpha_\BD = -1/5$ and $\hat y_\BD = \hat\alpha_\BD u$,
    thus $\tilde q_\BD = \nu/(1 + \hat y_\BD) = (1,1,0)/2$, $\hat\ell_\BD = -0.6931$, and $\hat\ell_\BD<-\hat\ell_\PR$.

  \item \refPR$\rightarrow$\refBD:~Since $\hat\ell_\PR$ is finite, it should hold that
      $\hat\ell_\PR = - \hat\ell_\BD$, but it does not.
      Moreover, $\tilde y_\BD = (1, -1, -10)/10 = -(1/10) u \not\in \SOL_\BD$.
\end{itemize}
\end{example}

\section{Scope of validity of BD-dual \refBD{}}\label{St:BD-thm}

The correct relation of the BD-dual \refBD{} to the primal \refPR{} is stated in
Theorem~\ref{T:BD}. In order to formulate the conditions under which $\hat\ell_\BD$ is infinite
(cf.~Theorem~\ref{T:BD}(a)) the notions of H-set and Z-set are introduced.
They are implied by the recession cone considerations of \refBD.
\begin{definition}
If a nonempty convex closed set $C\subseteq \sx\A$ and
a type $\nu$ are such that $C^a(0^p)=\emptyset$, then we say that
$C$ is an \emph{H-set with respect to $\nu$}.
The set $C$ is called a \emph{Z-set with respect to $\nu$} if
$C^a(0^p)$ is nonempty but its support is strictly smaller than~$\Aa$.
\end{definition}

Note that $C^a(0^p)$ comprises those $q\in C$
which are supported on the active letters.
Thus, $C$ is neither an H-set nor a Z-set if and only if there is $q\in C$
with $q^a>0$, $q^p=0$.

Clearly, in Example~\ref{E:E}, $C$ is an H-set with respect to the $\nu$.
The same set~$C$ becomes a Z-set with respect to the $\nu$ considered in Example~\ref{E:Z}.
And it is neither an H-set nor a Z-set with respect to the $\nu$
studied in Example~\ref{E:no-gap}.
Further, the feasible set $C$ considered in Example~\ref{E:gap}
is neither an H-set nor a Z-set with respect to the particular $\nu$.
In Examples~\ref{E:nonlin} and \ref{E:Wets}, $C$ is an H-set with respect to the~$\nu$.

\begin{theorem}[Relation between \refBD{} and \refPR]\label{T:BD}
    Let $\nu,C,\hat{q}^a_\PR,\hat{y}^a_\F$ be as in Theorems~\ref{T:P} and~\ref{T:F}.
\begin{enumerate}
  \item[(a)] If $C$ is either an H-set or a Z-set with respect to $\nu$
     then
     \begin{equation*}
        \hat{\ell}_{\BD}=-\infty
        \qquad\text{and}\qquad
        \SOL_\BD=\emptyset.
     \end{equation*}
  \item[(b)] If $C$ is neither an H-set nor a Z-set then $\hat\ell_\BD$ is finite,
     $\hat\ell_\BD\le \hat\ell_\F$,
     and there is $\hat{y}^a_\BD\in C^{\ast a}$ such that $\hatmu(-\hat{y}^a_\BD)=1$ and
     \begin{equation*}
        \SOL_\BD = \{\hat{y}^a_\BD\} \times C^{\ast p}(\hat{y}^a_\BD).
     \end{equation*}
     Moreover, there is no BD-duality gap, that is,
     \begin{equation*}
        \hat\ell_\BD=-\hat\ell_\PR,
     \end{equation*}
     if and only if any of the following (equivalent) conditions hold:
     \begin{enumerate}
       \item[(i)]   $\sum \hat q^a_\PR=1$ (that is, $\SOL_\PR=\{(\hat q^a_\PR,0^p)\}$);
       \item[(ii)]  $\hat{y}^a_\BD=\hat{y}^a_\F$
         (that is, $\nu^a/(1+\hat{y}^a_\BD)=\hat q^a_\PR$);
       \item[(iii)]
          $\ell(\hat q_{\BD}) + \ell^\ast(-\hat{y}_\BD)=0$
          (\emph{extremality relation}) for some $\hat y_\BD\in\SOL_\BD$, where
          $$
           \hat q_{\BD}\triangleq  \left(
              \frac{\nu^a}{1+\hat y^a_{\BD}}, 0^p
           \right);
          $$
       \item[(iv)] $(\hat{y}^a_\BD, -1^p)\in C^\ast$.
     \end{enumerate}
\end{enumerate}
\end{theorem}

Informally put, Theorem~\ref{T:BD}(a) demonstrates that the BD-dual breaks down
if $C$ is either an H-set or a Z-set with respect to the observed type $\nu$.
Then the (\refBD{} $\rightarrow$ \refPR) part of Theorem~\ref{T:BD2.1} does not apply.
At the same time the (\refPR{} $\rightarrow$  \refBD) part of Theorem~\ref{T:BD2.1}
does not hold, as $\SOL_\PR \neq\emptyset$, yet $\SOL_\BD = \emptyset$.
This is illustrated by Examples~\ref{E:E}--\ref{E:Z}.

Part (b) of Theorem~\ref{T:BD} captures the other discomforting fact about the BD-dual:
if the solution of \refBD{} exists, it may not solve the primal problem \refPR.
By {(i)} this happens whenever the solution $\hat q_\PR$ of \refPR{}
assigns a positive weight to at least one of the passive letters
(provided that $\SOL_\BD \neq\emptyset$). See Example~\ref{E:gap}.

At least, for $\nu > 0$ the BD-dual works well.
\begin{corollary}\label{C:P+F+B:nu-positive}
If $\nu>0$ then $\SOL_\PR=\{\hat q_\PR\}, \SOL_\F=\{\hat y_\F\}, \SOL_\BD=\{\hat y_\BD\}$
are  singletons,
\begin{equation*}
    \hat\ell_\BD=\hat\ell_\F=-\hat\ell_\PR,
    \qquad
    \hat{y}_\BD=\hat{y}_\F=\frac{\nu}{\hat{q}_\PR}-1,
    \qquad\text{and}\qquad
    \hat q_\PR   \,\bot \, \hat y_\BD.
\end{equation*}
\end{corollary}

The corollary justifies the use of the BD-dual for solving \refPR{} when $\nu > 0$.
Recall that in the case of linear $C$ the BD-dual
is just Smith's simplified Lagrange dual problem (\ref{EQ:Smith}), which is an unconstrained optimization problem.
It can be solved numerically by standard methods for unconstrained optimization
or by El~Barmi \& Dykstra's \cite{elbarmi-dykstra} cyclic ascent algorithm.

Finally, it is worth noting that the solution sets $\SOL_\PR$ and $\SOL_\F$ are always compact
but $\SOL_\BD$, if nonempty, is compact if and only if $\nu>0$, that is,
there is no passive letter.
If $\nu\not >0$ and $\SOL_\BD\ne\emptyset$, then $\SOL_\BD$ is unbounded from below.

\subsection{Base solution of \refBD{} and no BD-duality gap}

The case (iv) of Theorem~\ref{T:BD}(b) provides a way to find out
whether a solution of \refPR{} assigns the zero weights to the passive letters or not.
First, determine whether $C$ is neither an H-set nor a Z-set with respect to $\nu$.
Then, solve the BD-dual \refBD{} and find a solution $\hat y_\BD$ of it.
Finally, verify that $(\hat{y}^a_\BD,-1^p)$ belongs to the polar cone $C^\ast$.
For example, if $\hat y^p_\BD\ge -1$ then this is satisfied automatically.

On the other hand, in order to have $\hat q_\PR^p=0^p$
(that is, to have no BD-duality gap), $\hat y^p_\BD\ge -1$ must be satisfied
by some $\hat y^p_\BD\in\SOL_\BD$.
In the case when $C$ is polyhedral, there must exist a {base} solution
$\hat y^p_\BD=\sum_h \hat\alpha_h u_h$ of \refBD{} with $\hat y^p_\BD\ge -1$.
The next example illustrates the point.

\begin{example}[Base solution of \refBD{}]\label{EX:when-passive-needed}
Let $\A = \{-1, 1, a, b\}$, where $b > a > 1$. Let $C = \{q\in\sx{\A}: \E_q X = 0\}$,
so that $u = (-1, 1, a, b)$. Let $\nu = (\nu_{-1}, \nu_1, 0, 0)$; hence $\Aa = \{-1, 1\}$,
$\Ap = \{a, b\}$, and $C^a(0^p)=\{(1,1)/2\}$.
For what values of $\nu_1$ the solution of \refPR{}
assigns zero weights to the passive letters? First, $\hat\alpha_\BD$
solves $\sp{\tilde q_\BD^a}{u^a} = 0$, where $\tilde q_\BD^a = \frac{\nu^a}{1 + \hat\alpha_\BD u^a}$.
This gives $\hat\alpha_\BD = 2\nu_1 - 1$. Thus, $\hat\alpha_\BD < 0$ when $\nu_1 < 1/2$.
Then, in order to have $\hat q^p_\PR = 0^p$,
the condition $\hat y^p_\BD\ge -1$ gives that $\nu_1 \ge \frac{b-1}{2b}$.
In the other case ($\hat\alpha_\BD \ge 0$)
the condition is not binding, hence $\hat q^p_\PR = 0^p$ for $\nu_1 \ge 1/2$.
Taken together, for $\nu_1 < \frac{b-1}{2b}$ it holds that $\tilde q_\BD \neq \hat q_\PR$,
since $\hat q^p_\PR$ has a positive coordinate and $\tilde q_\BD^p = 0^p$.
To give a numeric illustration,
let $a = 2$, $b = 5$, and $\nu_1=3/10 <4/10$.
Then $\hat q_\PR = (14, 9, 0,1)/24$, that is, a positive weight is
assigned to a passive letter.  For $\nu_1 = 9/20$, which is above the threshold,
$\hat q_\PR = (1, 1, 0, 0)/2$.
\end{example}

\smallskip

To sum up, El~Barmi and Dykstra's dual may fail to lead to the solution of
the multinomial likelihood primal \refPR{} in different ways.
For a particular type $\nu$ the feasible set $C$ may be an H-set or a Z-set,
and then the BD-dual fails to attain finite infimum.
Even if this is not the case \refBD{} may fail to provide a solution of \refPR{},
due to the BD-duality gap.
Theorem~\ref{T:BD}(b) states equivalent conditions under which the BD-dual
is in the extremality relation with \refPR{} and leads to a solution of \refPR{}; see also Lemma~\ref{L:BD-duality:dual-to-primal3}.

In the next two sections other possibilities of solving \refPR{} are explored.
First, an active-passive dualization is considered. Then, a perturbed primal problem and the PP algorithm are studied.
Interestingly, a solution of the perturbed primal problem may be obtained from the BD-dual problem.

\section{Active-passive dualization}\label{St:AP}

The \emph{active-passive dualization}
is based on a reformulation of the primal \refPR{}
as a sequence of partial minimizations
\begin{equation*}
    \hat\ell_\PR = \inf_{q^p\in C^p} \inf_{q^a\in C^a(q^p)} \ell(q^a,q^p).
\end{equation*}
Assume that $q^p$ is such that the slice $C^a(q^p)$ has support $\Aa$
(this is not a restriction, since otherwise the inner infimum is $\infty$).
Since $\nu^a > 0$, Corollary~\ref{C:P+F+B:nu-positive} gives that
a solution of the inner (active) primal problem~\refAP{}
\begin{equation}\label{EQ:APinner}
\begin{split}
    \hat{\ell}_\PR(q^p) &\triangleq \inf_{q^a\in C^a(q^p)} \ell(q^a,q^p),
  \\
    \SOL_{\PR}(q^p) &\triangleq \{\hat{q}^a\in C^a(q^p):\
    \ell(\hat{q}^a, q^p) = \hat\ell_\PR(q^p)\},
\end{split}
\tag{$\APbig$}
\end{equation}
can be obtained from its BD-dual problem \refBDkappa{}
\begin{equation}\label{EQ:BDkappa}
\begin{split}
    \hat{\ell}_\BD(q^p) &\triangleq
      \sup_{y^a\in (C^a(q^p))^\ast} I_\kappa(\nu^a\,\|\,y^a),
  \\
    \SOL_\BD(q^p) &\triangleq \{\hat{y}^a\in (C^a(q^p))^\ast:
      I_\kappa(\nu^a\,\|\,\hat{y}^a) = \hat\ell_\BD(q^p)\},
\end{split}
\tag{$\BDkappabig$}
\end{equation}
where $\kappa=\kappa(q^p) \triangleq 1/(1-\sum q^p)$.

\begin{theorem}[Relation between \refBDkappa{} and \refAP{}]\label{T:AP}
Let $q^p\in C^p$ be such that the support of $C^a(q^p)$ is $\Aa$.
Then there is a unique solution $\hat y^a(q^p)$ of $\BDkappabig$,
and
\begin{equation}\label{EQ:BD-duality-active}
  \hat{q}^a(q^p) \triangleq \frac{\nu^a}{\kappa(q^p) + \hat{y}^a(q^p)}
\end{equation}
is the unique member of $\SOL_\PR(q^p)$.
Moreover, ${\hat q^a(q^p)}\bot \,{\hat y^a(q^p)}$ and $\hat\ell_\BD(q^p)=-\hat\ell_\PR(q^p)$.
\end{theorem}

Thanks to (\ref{EQ:BD-duality-active}) and the extremality relation between \refAP{} and \refBDkappa,
the \emph{active-passive (AP) dual form} of the active-passive primal is
$$
  \adjustlimits\sup_{q^p\in C^p}\sup_{y^a\in (C^a(q^p))^\ast} \, I_\kappa(\nu^a\,\|\,{y}^a).
$$
The active-passive dualization is illustrated by the following example.

\begin{example}[continues=E:Z]
Here $C^p = [0,1/2]$ and the AP dual can be written in the form
$$
  \hat q^p_1 = \argmax_{q^p_1\in C^p}\,
     \sup_{\alpha\in\RRR}\, I_{\kappa(q^p_1)}(\nu^a\,\|\,\alpha v^a),
$$
where $v^a = u^a + (\kappa(q_1^p)\,q^p_1 u^p_1)  1^a$.
The inner optimization gives
$$
  \hat\alpha(q^p_1) = \frac{1-3q^p_1}{2q^p_1(2q^p_1-1)}.
$$
This $\hat\alpha(q^p_1)$, plugged into
$I_{\kappa(q^p_1)}(\nu^a\,\|\,\alpha v^a)$, yields
$$
  (1/2)\log\left[1 + \hat\alpha(q^p_1)(2q^p_1 - 1)\right]
  \ + \
  (1/2)\log\left[1 + \hat\alpha(q^p_1) q^p_1 \right]
  \ - \
  \log[1-q^p_1],
$$
which has to be maximized over $q^p_1\in C^p = [0,1/2]$.
The maximum is attained at $\hat q^p_1 = 1/4$.
Thus $\hat\alpha(\hat q^p_1) = -1$, $\kappa(\hat q^p_1) = 4/3$, $v^a_{-1} = -2/3$,
and $v^a_0 = 1/3$;
since $\hat q^a(\hat q^p_1) = \nu^a/(\kappa(\hat q^p_1) + \hat\alpha(\hat q^p_1) v^a)$
by~(\ref{EQ:BD-duality-active}), $\hat q^a_{-1} = 1/4$
and $\hat q^a_0 = 1/2$. Hence, $\hat q = (1, 2, 1)/4$.
\end{example}

In the outer, passive optimization,
it is possible to exploit the structure of $\SOL_\PR$
(cf.~Theorem~\ref{T:P}), and this way reduce the dimension of the problem.
This is the case, for instance, when $C$ is polyhedral.

\section{Perturbed primal \refPP{} and PP algorithm}\label{St:PP}
For $\delta > 0$ let $\nu(\delta)\in\sx{\A}$ be a perturbation of the type $\nu$; we assume that
\begin{equation}\label{EQ:PP-nu(delta)}
    \nu(\delta) >0
    \qquad\text{and}\qquad
    \lim_{\delta\searrow 0} \nu(\delta) = \nu.
\end{equation}
The perturbation activates passive, unobserved letters.
For every $\delta>0$ consider the \emph{perturbed primal problem} \refPP{}
\begin{equation}\label{EQ:PP}
\begin{split}
    \hat\ell_{\PR}(\delta)
    &\triangleq
    \inf_{q\in C} \ell_{\delta}(q)
    = \inf_{q\in C}  -\sp{\nu(\delta)}{\log q},
\\
    \SOL_{\PR}(\delta)
    &\triangleq
    \{\hat{q}\in C:\ \ell_{\delta}(\hat q)=\hat\ell_\PR(\delta)\},
\end{split}
    \tag{$\PPbig$}
\end{equation}
where $\ell_\delta\triangleq \ell_{\nu(\delta)}$.

Since the activated type $\nu(\delta)$ has no passive coordinate,
the perturbed primal problem \refPP{}
can be solved, for instance, via the BD-dualization; recall Corollary~\ref{C:P+F+B:nu-positive}.
Thus, for every $\delta>0$,
\begin{equation}\label{EQ:PP-solutions}
    \SOL_{\PR}(\delta) = \{\hat{q}_\PR(\delta)\},
    \quad
    \SOL_{\F}(\delta)=\SOL_{\BD}(\delta) = \{\hat{y}_\BD(\delta)\},
    \quad
    \hat{q}_\PR(\delta) = \frac{\nu(\delta)}{1+\hat{y}_\BD(\delta)} \,.
\end{equation}

How is $\hat{q}_\PR(\delta)$ related to $\SOL_\PR$, and $\hat\ell_\PR(\delta)$ to $\hat\ell_\PR$?
Theorem~\ref{T:PP} asserts that $\ell_{\delta}^C$
epi-converges to $\ell_\nu^C$ when $\delta\searrow 0$.
(There, for a map $f:\RRR^m\to\RRRex$ and a set $C\subseteq\RRR^m$, the map $f^C:\RRR^m\to\RRRex$
is given by $f^C(x)=f(x)$ if $x\in C$, and $f^C(x)=\infty$ if $x\not\in C$.)
The epi-convergence (cf.~Rockafellar and Wets~\cite[Chap.~7]{rockafellar2009variational})
is used in convex analysis to study a limiting behavior of perturbed optimization problems.
It is an important modification of the uniform convergence (cf.~Kall~\cite{Kall},
Wets~\cite{Wets}).

\begin{theorem}[Epi-convergence of \refPP{} to \refPR{}]\label{T:PP}
 Assume that $C$ is a convex closed subset of $\sx{\A}$ with support $\A$, 
 $\nu\in\sx{\A}$, and $(\nu(\delta))_{\delta>0}$
 is such that (\ref{EQ:PP-nu(delta)}) is true.
 Then
 \begin{equation*}
    \ell_{\delta}^C
    \quad\text{epi-converges to}\quad
    \ell^C_{\nu}
    \qquad\text{for}\quad
    \delta\searrow 0.
 \end{equation*}
 Consequently,
 \begin{equation*}
    \lim_{\delta\searrow 0} \inf_{\hat q_\PR\in\SOL_\PR} d(\hat q_\PR(\delta), \hat q_\PR) = 0
 \end{equation*}
 and the active coordinates of solutions of \refPP{} converge
 to the unique point $\hat q^a_\PR$ of~$\SOL_\PR^a$:
 \begin{equation*}
    \lim_{\delta\searrow 0} \hat{q}^a_\PR(\delta) = \hat{q}^a_\PR \in \SOL_\PR^a.
 \end{equation*}
 Moreover, if $\SOL_\PR$ is a \emph{singleton} (particularly, if $\nu>0$)
 then also the passive coordinates converge
 and
 \begin{equation*}
    \lim_{\delta\searrow 0} \hat{q}_\PR(\delta) \in \SOL_\PR.
 \end{equation*}
\end{theorem}

Thus, if $\delta$ is small enough, the (unique) solution $\hat{q}_\PR(\delta)$
of the perturbed primal~\refPP{} is close to a solution of the primal problem \refPR{}.
Theorem~\ref{T:PP} also states that in the active coordinates the convergence
is pointwise, to the unique $\hat{q}^a_\PR$ of $\SOL_\PR^a$.
The next theorem demonstrates that if $C$ is given by linear constraints and $\nu(\delta)$ is defined in a `uniform' way
in $\delta$, also the passive coordinates of $\hat{q}_\PR(\delta)$ converge pointwise.

\begin{theorem}[Convergence of \refPP{} to \refPR{}, linear $C$]\label{T:PP-linear}
    Let $C$ be given by (\ref{EQ:C-linear}),
    $\nu\in\sx{\A}$, and $(\nu(\delta))_{\delta>0}$
    be such that (\ref{EQ:PP-nu(delta)}) is true.
    Further, assume that $\nu(\cdot)$ is continuously differentiable
    and that there is a constant $c>0$ such that, for every $i\in\Ap$,
    \begin{equation}\label{EQ:PP-nu-bounded-differ}
      \abs{\w_i'(\delta)}\le c\,\w_i(\delta),
      \qquad\text{where}\quad
      \w_i(\delta) \triangleq \frac{\nu_i(\delta)}{\sum_{j\in\Ap} \nu_j(\delta)}\,.
    \end{equation}
    Then
    \begin{equation*}
        \lim_{\delta\searrow 0} \hat{q}_\PR(\delta)
        \quad\text{exists and belongs to }
        \SOL_\PR.
    \end{equation*}
\end{theorem}

Notice that the condition (\ref{EQ:PP-nu-bounded-differ})
is satisfied if $\nu_i(\delta)=\nu_j(\delta)$
for every $i,j\in\Ap$; for example if
\begin{equation}\label{EQ:PP-nu(delta)-linear}
    \nu(\delta) = \frac{1}{1+\delta m_p} (\nu + \delta \xi),
\end{equation}
where $\xi\triangleq (0^a,1^p)$ is the vector
with $\xi_i=1$ if $i\in\Ap$ and $\xi_i=0$ if $i\in\Aa$.
This corresponds to the case when every passive coordinate is `activated' by equal weight.

The following example demonstrates that without the assumption
(\ref{EQ:PP-nu-bounded-differ}) the convergence in passive letters need not occur.

\begin{example}[Divergent $\hat q_\PR(\delta)$]
Consider the setting of Example~\ref{EX:nonunique-min}. That is,
$\A=\{-1,0,1\}$ and
\begin{equation*}
    C=\{q\in\sx{\A}: \sp{q}{u}=0\}
    =\{(a,1/2,1/2-a): a\in[0,1/2]\},
\end{equation*}
where $u=(1,-1,1)$.
For $\nu=(0,1,0)$, $\SOL_\PR=C$. Define the perturbed types $\nu(\delta)$ in such a way that
\begin{equation*}
    \nu(\delta)=\begin{cases}
      (\delta,1-3\delta,2\delta)  &\text{if } \delta\in\{1/(2n): n\in\NNN\},
    \\
      (2\delta,1-3\delta,\delta)  &\text{if } \delta\in\{1/(2n+1): n\in\NNN\},
    \end{cases}
\end{equation*}
and $\nu(\cdot)$ is $C^1$ on $(0,1)$. Then, for every $n\in\NNN$,
\begin{equation*}
    \hat q_\PR(1/(2n)) = (1/6,1/2,1/3)
    \qquad\text{and}\qquad
    \hat q_\PR(1/(2n+1)) = (1/3,1/2,1/6);
\end{equation*}
so the limit $\lim \hat q_\PR(\delta)$ does not exist. Note that in this case
the condition~(\ref{EQ:PP-nu-bounded-differ}) is violated. Indeed, $\w_1(1/(2n)) = 2/3$ and
$\w_1(1/(2n+1)) = 1/3$ for every $n$; hence, by the mean value theorem, for every $n$ there is
$\zeta_n\in(1/(2n+1),1/(2n))$ such that $\w_1'(\zeta_n)=2n(n+1)/3$. Since $0<\w_1(\zeta_n)\le 1$,
$\lim_n {\w_1'(\zeta_n)/\w_1(\zeta_n)}=\infty$.
\end{example}

In Examples~\ref{E:E}, \ref{E:Z}, \ref{E:gap},
with $\nu(\delta)$ given by (\ref{EQ:PP-nu(delta)-linear}),
the pointwise convergence can be demonstrated analytically.

\begin{example}[continues=E:E]
Here $\hat y_\BD(\delta) = \hat\alpha_\BD(\delta) u$,
where $\hat\alpha_\BD(\delta) = \frac{\delta  -1}{\delta + 1}$.
So, $\lim \hat{\alpha}_\BD(\delta) = -1$ and
$\lim \hat q_\PR(\delta) = (1, 0, 1)/2 \in \SOL_\PR$.
\end{example}

\begin{example}[continues=E:Z]
First, $\hat y_\BD(\delta) = \hat\alpha_\BD(\delta) u$ and
$$
 \hat\alpha_\BD(\delta) = \argmax_{\alpha\in\RRR}
 \left[
  \nu_{-1}(\delta)\log(1-\alpha)
  + \nu_0(\delta)\log 1 + \nu_1(\delta)\log(1+\alpha)
 \right]
  ;
$$
this leads to
$\hat\alpha_\BD(\delta) =  \frac{2\delta-1}{2\delta+1}$.
Thus, $\lim \hat{\alpha}_\BD(\delta) = -1$ and, since
$\hat q_\PR(\delta) = \frac{\nu(\delta)}{1 + \hat\alpha_\BD(\delta) u}$,
$\lim \hat q_\PR(\delta) = (1, 2, 1)/4 \in \SOL_\PR$.
\end{example}

\begin{example}[continues=E:gap]
Here $\hat y_\BD(\delta) = \hat\alpha_\BD(\delta) u$ with
$\hat\alpha_\BD(\delta) = -\frac{3 - \sqrt{1 + 392\delta + 400\delta^2}}{20(1 + \delta)}$.
Thus, $\lim \hat\alpha_\BD(\delta) = -1/10$, and
$\lim \hat q_\BD(\delta) = (54, 44, 1)/99\in \SOL_\PR$.
\end{example}

The next example provides a numeric illustration of the pointwise convergence
of a sequence of perturbed primals  to~\refPR{}.
The perturbed primal solutions are obtained through their BD-duals.
It is worth stressing that \refBD{} to \refPP{} is, for a linear $C$, an unconstrained optimization problem; cf.~Section~\ref{St:Smith}.

\begin{example}[Qin and Lawless \cite{QL}, Ex.~1]\label{E:PP-convergence}
Consider a discrete-case analogue of Example~1 from Qin and Lawless~\cite{QL}.
Let $\A = \{-2,-1, 0, 1, 2\}$ and
$$
 C_\theta = \{q\in\sx{\A}: \mathrm{E}_q (X-\theta) = 0, \mathrm{E}_q (X^2-2\theta^2-1) = 0\},
$$
where $\theta\in\Theta = [-2,2]$.
Then $C_\Theta = \bigcup_{\theta\in\Theta} C_\theta$ is the estimating equations model;
cf.~(\ref{EQ:EE}). Clearly,
$u_{1}(\theta) = (-2-\theta, -1-\theta, -\theta, 1-\theta, 2 -\theta)$
and $u_{2}(\theta) = (3-2\theta^2,-2\theta^2,-1-2\theta^2,-2\theta^2, 3-2\theta^2)$.
Let $\nu = (0,0,7,3,0)/10$.

For a fixed $\theta\in\Theta$ and a perturbed type $\nu(\delta)$,
the BD-dual to the perturbed primal \refPP{} is (equivalent to)
$$
  \hat{\alpha}_\BD(\delta) = \argmin_{\alpha\in\RRR^2} \,
  I_1(\nu^a\,\|\, \sp{\alpha}{u(\theta)})
$$
and the corresponding
$\hat q_\BD(\delta) = \frac{\nu(\delta)}{1 + \sp{\hat\alpha(\delta)}{u(\theta)}}$.
For $\theta = 0$ and $\nu(\delta)$ given by (\ref{EQ:PP-nu(delta)-linear})
with $\delta = 10^{-j}$ ($j = 3, 5, 7, 9)$,
Table~\ref{tab:PP} illustrates the pointwise convergence of $\hat q_\BD(\delta)$ to $\hat q_\PR$,
$-\hat\ell_\BD(\delta)$ to $\hat\ell_\PR$.

The optimal $\hat{\alpha}_\BD(\delta)$'s were computed by \texttt{optim} of \texttt{R}~\cite{R}.
The rightmost  three columns in Table~\ref{tab:PP} state orders of the precision $10^{-\gamma}$ of satisfaction
of the constraints $\E_{\hat q(\delta)} X = 0$, $\E_{\hat q(\delta)} (X^2 - 1) = 0$, and $\sum\hat q(\delta) - 1 = 0$.
The solution of~\refPR{}, obtained by \texttt{solnp} from the \texttt{R}
library \texttt{Rsolnp} (cf.~Ghalanos and Theussl~\cite{Rsolnp}; based on Ye~\cite{Ye}), is
$\hat q_\PR  = (0.1625, 0, 0.525, 0.3, 0.0125)$ and
$\hat\ell_\PR = 0.812242$.

\begin{table}[h!]
\centerline{
\resizebox{\linewidth}{!}{
\begin{tabular}{ |r|c|c|c|c|c|c|c|c|c| }
  \hline\rule{0pt}{1.25em}
  $j$ & \multicolumn{5}{|c|}{$\hat q_\BD(\delta)$} & $-\hat\ell_\BD(\delta)$
      & $\gamma_1$ & $\gamma_2$ & $\gamma_3$\\
  \hline\rule{0pt}{1.05em}
  3 & 0.161439 & 0.001013 &  0.528326 & 0.294553 & 0.014669 & 0.823788 & 7 & 7 & 8\\
  5 & 0.162488 & 0.000010 &  0.525041 & 0.299936 & 0.012525 & 0.812404 & 7 & 7 & 6\\
  7 & 0.162501 & 1e-7 &      0.525000 & 0.299999 & 0.012500 & 0.812242 & 6 & 6 & 6\\
  9 & 0.162501 & 1e-8 &      0.525000 & 0.300000 & 0.012502 & 0.812242 & 5 & 5 & 6\\
  \hline
\end{tabular}
}}
\caption{The pointwise convergence of \refPP{} to \refPR{}.}
\label{tab:PP}
\end{table}

For $j > 9$
the numerical effects become noticeable.
For instance, for $j = 20$, the precision of the constraints satisfaction is of the order $10^{-1}$.

As an aside, note that for this type $\nu$ the BD-dual to the original,
unperturbed primal \refPR{} breaks down, since $C_\theta$ is an H-set.
In fact, it is an H-set with respect to this $\nu$ for any $\theta\in\Theta$;
cf.~the empty set problem in Grend\'ar and Judge~\cite{ESP}.
\end{example}

The convergence theorems suggest that the practice
of replacing the zero counts by an `ad hoc' value can be
superseded by the \emph{PP algorithm}; i.e., by a sequence of the perturbed primal problems,
for $\nu(\delta) > 0$ such that $\lim_{\delta\searrow 0} \nu(\delta) = \nu$.
Since each $\nu(\delta) > 0$, the  PP algorithm
can be implemented through the BD-dual to~\refPP{}, by the Fisher scoring algorithm, or
by the Gokhale algorithm \cite{Gokhale}, among other methods.

\section{Implications for empirical likelihood}\label{St:EL}

Let us point out some of the consequences of the presented results
for the empirical likelihood (EL) method; cf.~Owen~\cite{Owen1988}. In most settings, including the discrete one,
empirical likelihood is `a multinomial likelihood on the sample',~Owen~\cite[p.~15]{Owen}.
It is usually applied to an \emph{empirical estimating equations} model,
which is in the discrete setting
defined as $C_{\Theta,\nu^a} \triangleq \bigcup_{\theta\in\Theta} C_{\theta,\nu^a}$,
where
$$
 C_{\theta,\nu^a} \triangleq \{p\in\sx{\Aa}: \sp{p}{u_h^a(\theta)} = 0
 \text{ for } h = 1, 2, \dots, r\}
$$
and $u_h^a:\Theta\to\RRR^{m_a}$ are the \emph{empirical estimating functions}.
The empirical likelihood estimator is defined through
\begin{equation}\label{EQ:MEL}
  \inf_{\theta\in\Theta}\inf_{p \in C_{\theta,\nu^a}} \,-\sp{\nu^a}{\log p}.
\end{equation}
For a fixed $\theta\in\Theta$  the data-supported feasible set
$C_{\theta,\nu^a}$ is a convex set and the inner optimization in (\ref{EQ:MEL})
is the \emph{empirical likelihood inner problem}
\begin{equation}\label{EQ:EL}
  \hat\ell_\EL \triangleq \inf_{p\in C_{\theta,\nu^a}} \ell(p),
  \qquad
  \SOL_\EL \triangleq \{\hat p_\EL \in C_{\theta,\nu^a}: \ell(\hat p_\EL) = \hat\ell_\EL\}.
  \tag{$\ELbig$}
\end{equation}
Since $C_{\theta,\nu^a}$ is just the $0^p$-slice $C^a_\theta(0^p)$
of $C_\theta$ (given by~(\ref{EQ:EE})), the EL inner problem~\refEL{} can equivalently be expressed as
\begin{equation*}
  \hat\ell_\EL = \inf_{q^a\in C^a_\theta(0^p)} \ell(q^a).
\end{equation*}
Its dual is
\begin{equation}\label{EQ:Fenchel-to-EL}
  \inf_{\alpha\in \RRR^r} I_1\left(\nu^a\,\Bigl\|\, \sum_h \alpha_{h} u_h^a(\theta)\right).
\end{equation}
Note that~(\ref{EQ:Fenchel-to-EL})
is just Smith's simplified Lagrangean (\ref{EQ:Smith}),
that is, the BD-dual~\refBD{} to the multinomial likelihood primal problem \refPR{}, for the linear set $C_\theta$.
This connection implies, through Theorem~\ref{T:BD}, that the maximum of
empirical likelihood does not exist if $C_\theta$ is
either an H-set or a Z-set with respect to $\nu$.
The two possibilities are recognized in the literature on EL, where an H-set
is referred to as the \emph{convex hull condition} (cf.~Owen~\cite[Sect.~10.4]{Owen}),
and a Z-set is known as the \emph{zero likelihood problem} (cf.~Bergsma et al.~\cite{Bergsma}).
Theorem~\ref{T:BD} also implies that these are the only ways the EL inner problem may fail to have a solution.
Note that the inner empirical likelihood problem may fail to have a solution for any $\theta\in\Theta$;
cf.~the \emph{empty set problem}, Grend\'ar and Judge~\cite{ESP}.

In addition, Theorem~\ref{T:BD} implies that, besides failing to exist,
the empirical likelihood inner problem may have different solution than the
multinomial likelihood primal problem. If $C_\theta$ is
neither an H-set nor a Z-set then, by Theorem~\ref{T:BD}(b),  it is possible that
\begin{enumerate}
\item either $\hat\ell_\PR = -\hat\ell_\BD$ (no BD-duality gap), or
\item $\hat\ell_\PR < -\hat\ell_\BD$ (BD-duality gap).
\end{enumerate}
Since $\hat\ell_\BD = -\hat\ell_\EL$,
in the latter case $\hat\ell_\PR < \hat\ell_\EL$ and $\SOL_\PR \neq \SOL_\EL$.
This happens when any of the conditions (i)--(iv) from Theorem~\ref{T:BD}(b) is not satisfied.
Then the distribution that maximizes empirical likelihood differs from
the distribution that maximizes multinomial likelihood. Moreover,
the multinomial likelihood ratio may lead to different inferential and
evidential conclusions than the empirical likelihood ratio. The next example illustrates the points.

\begin{example}[LR vs. ELR]\label{E:LR-ELR}
Let $\A = \{-2,-1, 0, 1, 2\}$, $\Theta=\{\theta_1, \theta_2\}$ with $\theta_1 = 1.01$ and $\theta_2 = 1.05$. Let
$C_{\theta_j} = \{q\in\sx{\A}: \mathrm{E}_q (X^2)  = \theta_j\}$.
Clearly, $u({\theta_j}) = (4 - \theta_j, 1 - \theta_j, -\theta_j, 1 - \theta_j, 4 - \theta_j)$
for $j=1,2$. Let $\nu = (6, 3, 0, 0, 1)/10$.

The solution of \refPR{} is
\begin{itemize}
\item $\hat q_\PR(\theta_1) = (0.1515, 0.3030, 0.52025, 0, 0.02525)$, for $\theta_1$,
\item $\hat q_\PR(\theta_2) = (0.1575,  0.3150, 0.50125,  0, 0.02625)$, for $\theta_2$.
\end{itemize}
As the two solutions are very close, the multinomial likelihood ratio is
\begin{itemize}
  \item $\mathrm{LR}_{21} = \exp\left(n[\ell(\hat q_\PR(\theta_1)) - \ell(\hat q_\PR(\theta_2))]\right) = 1.4746$,
\end{itemize}
which indicates {inconclusive} evidence.

However, the empirical likelihood ratio leads to a very different conclusion.
Note that the active letters are $\Aa = \{-2, -1, 2\}$, and $C_\theta$ is neither an
H-set nor a Z-set with respect to the observed type $\nu$, for the considered $\theta_j$ ($j=1,2$).
Hence for both $\theta$'s the solution of \refEL{} exists
and it is
\begin{itemize}
  \item $\hat q_{\EL}(\theta_1) = (0.00286, 0.99\bar{6}, 0.00048)$, for $\theta_1$,
  \item $\hat q_\EL(\theta_2) = (0.01429, 0.98\bar{3}, 0.00238)$, for $\theta_2$.
\end{itemize}

The weights given by EL to $-2$ are very different in the two models; the same holds for $2$.
The empirical likelihood ratio is
\begin{itemize}
  \item $\mathrm{ELR}_{21} = \exp\left(n[\ell(\hat q_\EL(\theta_1)) - \ell(\hat q_\EL{}(\theta_2))]\right) = 75031.31$,
\end{itemize}
which indicates {decisive} support for $\theta_2$; cf.~Zhang~\cite{Zhang}.
\end{example}

The BD-duality gap thus implies that in the discrete \emph{iid} setting, when $C$ is linear,
EL-based inferences from finite samples may be grossly misleading.

\subsection{Continuous case and Fisher likelihood}\label{St:ELcontinuous}

As far as the continuous random variables are concerned, due to the finite precision of any measurement
`all actual sample spaces are discrete, and all observable random variables have
discrete distributions', Pitman~\cite[p.~1]{Pitman}.
Already Fisher's original notion of the likelihood~\cite{Fisher}
(see also Lindsey~\cite[p.~75]{Lindsey}) reflects the finiteness of the sample space.
For an \emph{iid} sample $X_1^n\triangleq (X_1, X_2, \dots, X_n)$ and a finite partition
$\mathcal A = \{A_l\}_1^m$ of a sample space $\mathcal X$,
the \emph{Fisher likelihood} $L_\mathcal{A}(q; X_1^n)$ which the data $X_1^n$ provide to a pmf $q\in\sx{\A}$ is
$$
  L_\mathcal{A}(q; X_1^n) \triangleq \prod_{A_l\in\mathcal{A}}  e^{n(A_l)\log q(A_l)},
$$
where $n(A_l)$ is the number of observations in $X_1^n$ that belong to $A_l$.
Thus, this view carries the discordances between the multinomial  and empirical likelihoods
also to the continuous \emph{iid} setting.

\begin{example}[FL with estimating equations]
To give an illustration of the Fisher likelihood as well as yet another example
that $\hat q_\EL$ may be different than $\hat q_\PR$, consider the setting of Example~\ref{E:PP-convergence}
with $\A = \{-4, -3.9, \dots, 3.9, 4\}$ and $\theta\in\Theta = [-4,4]$.
The letters of the alphabet are taken to be the representative points of the partition
$\mathcal A \triangleq \{(-\infty, -3.95)$, $[-3.95,-3.85)$, $\dots$, $[3.85, 3.95)$, $[3.95, \infty)\}$ of $\RRR$.
This way the alphabet captures the finite precision of measurements of a continuous random variable. The type $\nu$ exhibited at the panel c) of~Figure~\ref{Fig:P-EL} is induced by a random sample of size $n=100$
from the $\mathcal A$-quantized standard normal distribution. 
The EL estimate of $\theta$ is $-0.052472$ and the associated EL-maximizing distribution $\hat q_\EL$ is different than the multinomial likelihood maximizing distribution $\hat q_\PR$, which
is associated with the estimated value $0.000015$ and assigns a positive weight also to the passive letters $-4$ and $4$.
\begin{figure}[h!]\label{Fig:P-EL}
	\centering
	\includegraphics[width=\linewidth]{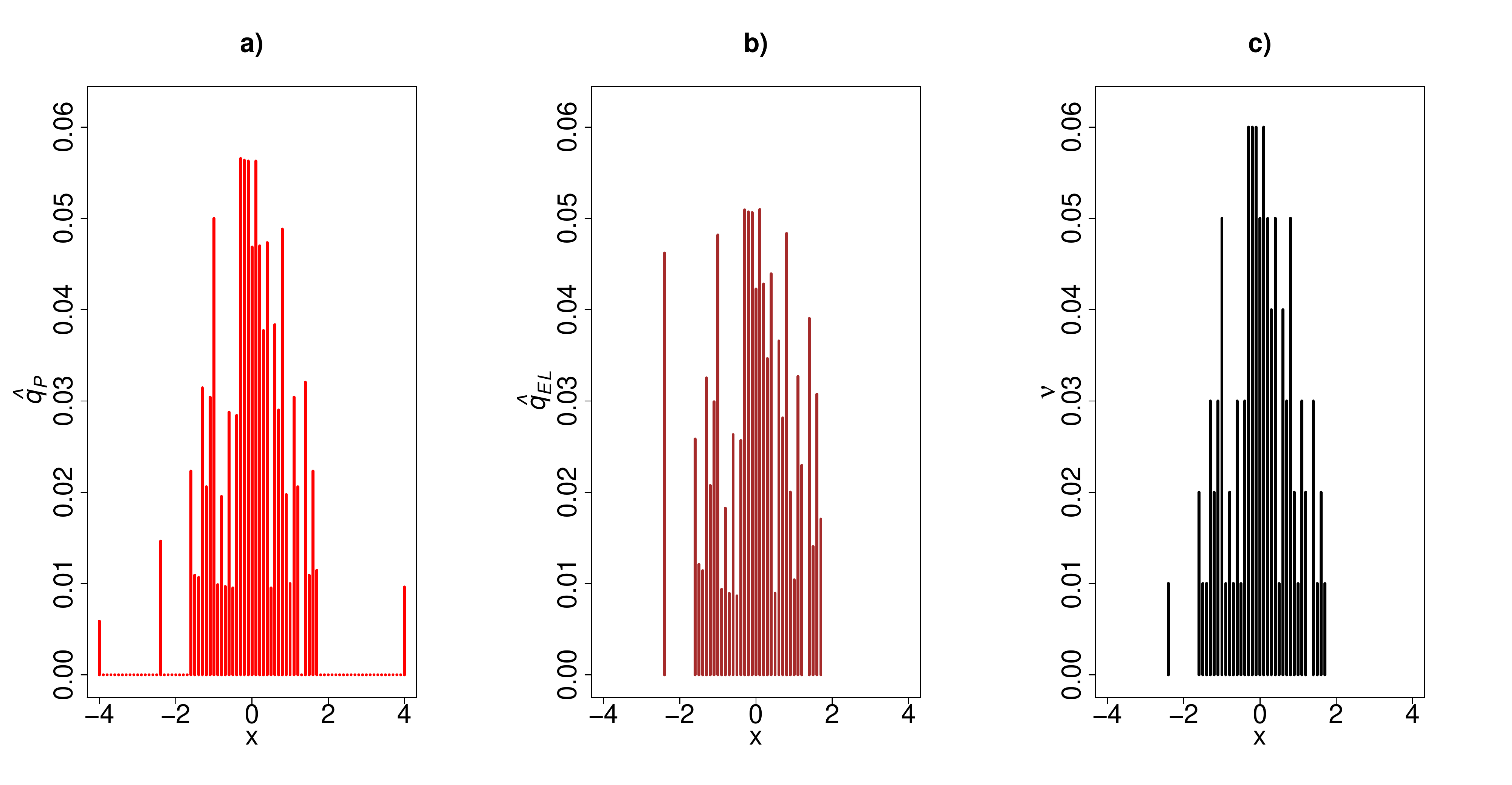}
	\caption{Panel a) the multinomial likelihood maximizing distribution $\hat q_\PR$; panel~b) the empirical likelihood maximizing distribution $\hat q_\EL$; and panel c) the observed type~$\nu$.}
	\label{fig:Fig_MmLred_MELbrown_NUblack__Aug3}
\end{figure}

\end{example}

\section{Implications for other methods}\label{St:other-implications}

Besides the empirical likelihood,
the minimum discrimination information,
compositional data analysis, Lindsay geometry, bootstrap in the presence of auxiliary information,
and Lagrange multiplier test ignore information about
the alphabet, and are restricted to the observed data.
Thus, they are affected by the above findings.

1) In the analysis of contingency tables with given marginals,
the minimum discrimination information (MDI)
method (cf.~Ireland and Kullback~\cite{IrelandKullback})
is more popular than the maximum multinomial likelihood method.
This is because the former is more computationally tractable,
thanks to the generalized iterative scaling algorithm (cf.~Ireland et al.~\cite{IrelandKuKullback}).
MDI minimizes $I_0(q\,\|\,\nu)$
over $q$, so that a solution of the MDI problem must assign a zero mass to a passive, unobserved letter.
Thus, MDI is effectively an empirical method.  This implies that the MDI-minimizing distribution
restricted to a convex closed set~$C$ may not exist;
however, the multinomial likelihood-maximizing distribution always exists
(cf.~Theorem~\ref{T:P}), and may assign a positive mass to an unobserved outcome(s).

\begin{example}[Contingency table with given marginals]
Consider a $3\times 3$ contingency table with given marginals.
Let $\A = \{1,2,3\}\times\{1,2,3\}$, and let the observed bi-variate type $\nu$ have all the mass
concentrated to $(1,1)$; the remaining eight possibilities have got zero counts.
Let the column and raw marginals be $f_c = (1, 2, 7)/10$, $f_r = (5, 4, 1)/10$, respectively.
One of the multinomial likelihood maximizing distributions  $\hat q_\PR$ is displayed in Table~\ref{tab:MDI}.
In the active letter $\hat q^a_\PR = 0.1$ is unique, in the passive letters $\hat q^p_\PR \in C^p(q^a_\PR)$.
The table exhibits the $\hat q^p_\PR$ which can also be obtained by the PP algorithm
with the uniform activation~(\ref{EQ:PP-nu(delta)-linear}).
Note that $C^a(0^p) = \emptyset$, so that the MDI-minimizing distribution does not exist.

\begin{table}[h!]
\centerline{
\begin{tabular}{|ccc|}
  \hline
  \multicolumn{3}{|c|}{$\hat q_\PR$}\\
  \hline\rule{0pt}{1.25em}%
  0.100 & 0.000 & 0.000\\
  0.094 & 0.080 & 0.026\\
  0.306 & 0.320 & 0.074\\
  \hline
\end{tabular}
}
\caption{A solution of \refPR{}.}
\label{tab:MDI}
\end{table}
\end{example}

It is worth stressing that the PP algorithm makes
the multinomial likelihood primal problem~\refPR{} computationally feasible.

\medskip

2) Multinomial likelihood maximization has the same solution regardless of wheth\-er the proportions
$\nu$ or the counts $(n_i)_{i \in \A}$ are used.
Note that the vector $\nu$ of normalized frequencies is an instance of the compositional data.
In the analysis of compositional data, it is assumed that
the compositional data $(x_1, \dots, x_m)$ belong to
$\{(x_1,\dots,x_m): x_1>0, \dots, x_m>0, \sum x_i = 1\}$;
cf.~Aitchison~\cite[Sect.~2.2]{AitchisonCompositional}.
This assumption transforms $\nu\in\sx{\A}$ into $\nu^a\in\sx{\A^a}$.
Consequently, the multinomial likelihood problem \refPR{} is replaced by the empirical likelihood problem \refEL{}.
However, this replacement is not without consequences, as the solution of the empirical likelihood problem \refEL{}
(if it exists) may differ from the solution of \refPR{}; cf.~Section~\ref{St:EL}.

\medskip

3) Lindsay~\cite[Sect.~7.2]{Lindsay} discusses  multinomial mixtures under linear constraints
on the mixture components, and assumes that it is sufficient to consider the distributions supported
in the data (i.e., in the active alphabet). Though the objective function $\ell(\cdot)$ in
\refPR{} is a `single-component' multinomial likelihood, the present results for the H-set, Z-set,
and BD-gap suggest that it would be more appropriate to work with the complete alphabet;
see also Anaya-Izquierdo et al.~\cite[Sect.~5.1]{Anaya}.

\medskip

4) Bootstrap in the presence of auxiliary information~(cf.~Zhang~\cite{ZhangBootstrap},~Hall
and Presnell~\cite{HallPresnell}) in the form of a convex closed set, resamples from the EL-maximizing distribution $\hat q_\EL$.
Hence, this method intentionally discards information about the alphabet.
Resampling from $\hat q_\PR$ seems to be a better option.

\medskip

5) The Lagrange multiplier (score) test (cf.~Silvey~\cite{Silvey}) of the linear restrictions on $q$
(cf.~$C$ given by~(\ref{EQ:C-linear})) fails if $C$ is an H-set or a Z-set with respect to $\nu$,
because the Lagrangean first-order conditions do not lead to a finite solution of \refPR{}. However, the multinomial likelihood ratio exists.

\section{Proofs}\label{St:proofs}

\subsection{Notation and preliminaries}\label{S:preliminaries}
In this section we introduce notation and recall notions and results which will be used
later; it is based mainly on Bertsekas \cite{bertsekas} and Rockafellar \cite{rockafellar,rockafellar1974}.
We will not repeat the definitions introduced in the previous part of the paper.

We assume that the extended real line $\RRRex=[-\infty,\infty]$ is equipped
with the order topology; so it is a compact metrizable space homeomorphic to the unit interval.
The arithmetic operations on $\RRRex$ are defined in a usual way; further we put
$0\cdot (\pm\infty)\triangleq 0$. For $\alpha\le 0$ we define $\log(\alpha)\triangleq -\infty$;
then $\log:\RRR\to\RRRex$ is continuous.

For $m\ge 0$ put $\RRR^m_+\triangleq\{x\in\RRR^m: x\ge 0\}$ and
$\RRR^m_-\triangleq\{x\in\RRR^m: x\le 0\}$ (recall that, for $m=0$, $\RRR^m=\{0\}$).
In the matrix operations, the members of $\RRR^m$ are considered to be column matrices.
If no confusion can arise, a vector with constant values is denoted by a scalar.

Let $C$ be a nonempty subset of $\RRR^m$. The \emph{convex hull} of $C$ is denoted by
$\conv(C)$.
The \emph{polar cone} of $C$ is the set
$C^\ast\triangleq\{y\in\RRR^m:\ \sp{y}{q}\le 0 \text{ for every } q\in C\}$. This is a
nonempty closed convex cone \cite[p.~166]{bertsekas}.
Assume that $C$ is convex. The \emph{relative interior} $\ri(C)$ of $C$ is
the interior of $C$ relative to the affine hull $\aff(C)$ of $C$
\cite[p.~40]{bertsekas}; it is nonempty and convex \cite[Prop.~1.4.1]{bertsekas}.

The \emph{recession cone} of a convex set $C$ is the convex cone
\begin{equation}\label{EQ:recession-cone-def}
    R_C\triangleq\{z\in \RRR^m:\ x+\alpha z\in C \text{ for every } x\in C, \alpha>0\}
\end{equation}
\cite[p.~50]{bertsekas}.
Every $z\in R_C$ is called a \emph{direction of recession} of $C$.
Clearly, $0\in R_C$; if $R_C=\{0\}$ it is said to be \emph{trivial}.
The \emph{lineality space $L_C$} of $C$ is defined by $L_C\triangleq R_C\cap (-R_C)$
\cite[p.~54]{bertsekas}; it is a linear subspace of $\RRR^m$.
Note that if $C$ is a cone then $R_C=C$ and $L_C=C\cap (-C)$.

\medskip

Let $X$ be a subset of $\RRR^m$ and $f:X\to\RRRex$ be a function.
By $f'(x;y)$ we denote the \emph{directional derivative} of $f$ at $x$
in the direction $y$ \cite[p.~17]{bertsekas}.
By $\nabla f(x)$ and $\nabla^2 f(x)$ we denote
the \emph{gradient} and the \emph{Hessian} of $f$ at $x$.
For a nonempty set $C\subseteq X$,
$\argmin_C {f}$ and $\argmax_C {f}$ denote the sets of all minimizing
and maximizing
points of $f$ over $C$, respectively; that is,
\begin{equation*}
\begin{split}
  &\argmin_C {f} \,= \left\{\bar{x}\in C:\ f(\bar{x})=\inf_{x\in C} f(x)\right\},\qquad
\\
  &\argmax_C {f} = \left\{\bar{x}\in C:\ f(\bar{x})=\sup_{x\in C} f(x)\right\}.
\end{split}
\end{equation*}
The \emph{(effective) domain} and the \emph{epigraph}
of $f$ are the sets \cite[p.~25]{bertsekas}
\begin{equation*}
    \dom(f) \triangleq \{x\in X:\ f(x)<\infty\}
    \quad\text{and}\quad
    \epi(f) \triangleq \{(x,w)\in X\times\RRR: f(x)\le w\}.
\end{equation*}
A function $f:X\to \RRRex$ is
\begin{itemize}
  \item \emph{proper} if $f>-\infty$ and there is $x\in\RRR^m$ with $f(x)<\infty$
    \cite[p.~25]{bertsekas}
   (this should not be confused with the properness associated with compactness of point preimages);
  \item \emph{closed} if $\epi(f)$ is closed in $\RRR^{m+1}$ \cite[p.~28]{bertsekas};
  \item \emph{lower semicontinuous (lsc)} if
    $f(x)\le\liminf_k f(x_k)$ for every $x\in X$ and every
    sequence $(x_k)_k$ in $X$ converging to $x$ \cite[p.~27]{bertsekas};
    analogously for the \emph{upper semicontinuity (usc)};
  \item \emph{convex} if both $X$ and $\epi(f)$ are convex \cite[Def.~1.2.4]{bertsekas};
  \item \emph{concave} if $(-f)$ is convex.
\end{itemize}
When dealing with closedness of $f$, we will often use the following simple lemma
\cite[Prop.~1.2.2 and p.~28]{bertsekas}.

\begin{lemma}\label{L:closed-f}
  Let $f:X\to\RRRex$ be a map defined on a set $X\subseteq \RRR^m$.
  Define
  \begin{equation*}
    \tilde{f}:\RRR^m\to\RRRex,
    \qquad
    \tilde{f}(x) \triangleq \begin{cases}
      f(x)    &\text{if }x\in X;
    \\
      \infty  &\text{if }x\not\in X.
    \end{cases}
  \end{equation*}
  Then the following are equivalent:
  \begin{enumerate}
    \item[(a)] $f$ is closed;
    \item[(b)] $\tilde{f}$ is closed;
    \item[(c)] $\tilde{f}$ is lower semicontinuous;
    \item[(d)] the level sets $V_\gamma\triangleq\{x\in\RRR^m:  \tilde{f}(x)\le\gamma\}
        = \{x\in X: f(x)\le\gamma\}$
      are closed for every $\gamma\in\RRR$.
  \end{enumerate}
\end{lemma}

The \emph{recession cone $R_f$} of a  proper convex closed function $f:\RRR^m\to\RRRex$
is the recession cone of any of its nonempty level sets $V_\gamma$ \cite[p.~93]{bertsekas}.
The lineality space $L_f$ of $R_f$ is, due to the convexity of $f$, the set of directions $y$
in which $f$ is constant (that is,~$f(x+\alpha y)=f(x)$
for every $x\in \dom(f)$ and $\alpha\in\RRR$);
thus $L_f$ is also called the \emph{constancy space} of $f$  \cite[p.~97]{bertsekas}.
If $g:\RRR^m\to\RRRex$ is \emph{concave}, the corresponding notions for $g$
are defined via the convex function $(-g)$.

The fundamental results underlying the importance of recession cones, are the following theorems
(\cite[Props.~2.3.2 and 2.3.4]{bertsekas} or \cite[Thm.~27.3]{rockafellar}).

\begin{theorem}\label{T:bertsekas}
Let $C$ be a nonempty convex closed subset of $\RRR^m$ and $f:\RRR^m\to \RRRex$ be a proper 
convex closed function such that $C\cap \dom(f)\ne\emptyset$. Then the following are equivalent:
\begin{enumerate}
  \item[(a)] the set $\argmin_C {f}$ of minimizing points of $f$ over $C$
    is nonempty and compact;
  \item[(b)] $C$ and $f$ have no common nonzero direction of recession, that is,
    \begin{equation*}
        R_C\cap R_f = \{0\}.
    \end{equation*}
\end{enumerate}
Both conditions are satisfied, in particular, if $C\cap \dom(f)$ is bounded.
\end{theorem}

\begin{theorem}\label{T:bertsekas2}
Let $C$ be a nonempty convex closed subset of $\RRR^m$ and $f:\RRR^m\to \RRRex$ be a 
convex closed function such that $C\cap \dom(f)\ne\emptyset$.
If
\begin{equation}\label{EQ:no-direction-recession-and-constancy}
    R_C\cap R_f = L_C\cap L_f,
\end{equation}
or if
\begin{equation*}
    C \text{ is polyhedral}
    \qquad\text{and} \qquad
    R_C\cap R_f \subseteq L_f,
\end{equation*}
then the set $\argmin_C {f}$ of minimizing points of $f$ over $C$ is nonempty.
Under condition (\ref{EQ:no-direction-recession-and-constancy}),
$\argmin_C {f}$ can be written
as $\tilde{C}+(L_C\cap L_f)$, where $\tilde{C}$ is compact.
\end{theorem}

\medskip

\begin{assumption}
If not stated otherwise, in the sequel it is  assumed that
a nonempty convex closed subset $C$ of $\sx{\A}$ having support $\A$, and
a type $\nu\in\sx{\A}$
are given.
\end{assumption}
Since no confusion can arise, by $\ell$ we denote also
an extension of the original function $\ell$ (defined
in (\ref{EQ:ell(q)})) to $\RRR^m$ :
\begin{equation}\label{EQ:f-def}
  \ell:\RRR^m\to\RRRex, \qquad
  \ell(x) \triangleq
      - \sp{\nu}{\log(x)}
    = - \sum_{i\in \Aa} \nu_i^a \log(x_i^a);
\end{equation}
the conventions $0\cdot(-\infty)=0$ and $\log\beta=-\infty$
for every $\beta\le 0$ apply.

\subsection{Proof of Theorem~\ref{T:P} (Primal problem)}\label{SS:primal}

In the next lemma we prove that
$\ell$ is a proper convex closed function. Since $C$ is compact,
$C$ has no nonzero direction of recession. From this and Theorem~\ref{T:bertsekas},
Theorem~\ref{T:P} will follow.

\begin{lemma}\label{L:f-props}
If $\nu\in\sx{\A}$ and $\ell$ is given by (\ref{EQ:f-def}), then
\begin{enumerate}
  \item[(a)] $\ell(x)>-\infty$ for every $x\in \RRR^m$, and
             $\ell(x)<\infty$ if and only if $x^a>0$; so
      \begin{equation*}
        \dom(\ell)=\{x\in\RRR^m:\ x^a>0\};
      \end{equation*}
  \item[(b)] $\ell$ is a proper continuous (hence closed) convex function;
  \item[(c)] the restriction of $\ell$ to its domain ${\dom(\ell)}$ is strictly convex
      if and only if $\nu>0$;
  \item[(d)] $\ell$ is differentiable on $\dom(\ell)$ with the gradient given by
    \begin{equation*}
      \nabla \ell(x) = - (\nu^a/x^a, 0^p);
    \end{equation*}
    the one-sided directional derivative of $\ell$ at $x\in\dom(\ell)$ in the direction $y$ is
    \begin{equation*}
      \ell'(x;y) = \sp{\nabla \ell(x)}{y} = - \sum_{i\in \Aa} \frac{\nu_i^a y_i^a}{x_i^a} \,;
    \end{equation*}
  \item[(e)]
      the recession cone $R_\ell$ and the constancy space $L_\ell$ of $\ell$ are
      \begin{equation*}
        R_{\ell} = \{
          z\in\RRR^m:\ z^a\ge 0
        \}
        \quad\text{and}\quad
        L_\ell = \{
          z\in\RRR^m:\ z^a= 0
        \}.
      \end{equation*}
\end{enumerate}
\end{lemma}

\begin{proof}
The properties (a) and (d) are trivial and the property (b)
follows from the continuity and concavity
of the logarithm (extended to the whole real line);
closedness of $\ell$ follows from continuity by Lemma~\ref{L:closed-f}.
For the property (c),
use that the Hessian $\nabla^2 \ell(x)=(\partial^2 \ell(x)/\partial x_i\partial x_j)_{ij}$
is a diagonal matrix $\diag(\nu^a/(x^a)^2, 0^p)$.
Hence it is positive definite and $\ell$ is strictly convex
if and only if $\nu>0$ \cite[Prop.~1.2.6]{bertsekas}.

It remains to prove (e). Fix any $\gamma\in\RRR$ such that the level set
$V=\{x:\ \ell(x)\le\gamma\}$ is nonempty.
If $z$ is such that $z^a\not\ge 0$ then there is an active letter $i$ with $z_i<0$.
In such a case,
for any $x\in V$ there is $\alpha>0$ with $y_i+\alpha z_i\le 0$ and so
$y+\alpha z\not\in \dom(\ell)\supseteq V$. Hence, by (\ref{EQ:recession-cone-def}),
$z\not\in R_{V}=R_{\ell}$.

Now take any $z$ with $z^a\ge 0$. Then, for every $x\in V$ and $\alpha>0$,
$\ell(x+\alpha z)\le \ell(x)$ by the monotonicity of the logarithm. That is,
$x+\alpha z\in V$ and so $z\in R_{V}=R_{\ell}$. The property (e) is proved.
\end{proof}

\begin{proposition}\label{P:primal}
If the support of $C$ is $\A$, then
the primal \refPR{} is finite
and its solution set $\SOL_\PR$ is nonempty and compact.
Moreover, if $\nu>0$ then $\SOL_\PR$ is a singleton.
\end{proposition}
\begin{proof}
Since $C$ is compact, its recession cone is trivial.
Thus the first assertion follows from Theorem~\ref{T:bertsekas}.
The fact that $\SOL_\PR$ is a singleton provided $\nu>0$, follows from the
strict convexness of $f$.
\end{proof}

\begin{proposition}\label{P:primal-active-projection}
Let $\nu\in\sx{\A}$ and $C$ be a convex closed set having support $\A$.
Then the $\pi^a$-projection of $\SOL_\PR$ onto active letters is always a singleton
$\{\hat{q}_\PR^a\}$ and
\begin{equation*}
    \SOL_\PR
    = \{q\in C:\ q^a=\hat{q}^a_\PR\}
    = \{\hat{q}^a_\PR\} \times C^p(\hat{q}^a_\PR).
\end{equation*}
Consequently, $\SOL_\PR$ is a singleton if and only if
$C^p(\hat{q}^a_\PR)$ is a singleton.
\end{proposition}
\begin{proof}
Note that $C^a=\pi^a(C)$ is a nonempty convex closed subset of $\RRR^{m_a}$.
Define $\ell^a:\RRR^{m_a}\to\RRRex$ by $\ell^a(x^a)\triangleq -\sp{\nu^a}{\log x^a}$
for $x^a\in C^a$ and $\ell^a(x^a)\triangleq \infty$ otherwise.
Since
\begin{equation}\label{EQ:P:primal-active-projection}
    \ell(q)=\ell^a(q^a)
    \qquad \text{for every}\quad
    q\in C,
\end{equation}
it holds that
\begin{equation*}
    \inf_{x^a\in C^a} \ell^a(x^a) = \hat\ell_\PR.
\end{equation*}
The map $\ell^a$ is proper, convex and closed (use Lemma~\ref{L:closed-f}
and the fact that the restriction of $\ell^a$
to the closed set ${C^a}$ is continuous, hence closed).
Since $\dom(\ell^a)\subseteq C^a$ is bounded,
Theorem~\ref{T:bertsekas}
gives that $\argmin \ell^a$ is a nonempty compact set.
This set is a subset of $\dom(\ell^a)$ and the restriction
of $\ell^a$ to $\dom(\ell^a)$ is strictly convex (Lemma~\ref{L:f-props}(c)),
so $\argmin \ell^a$ is a singleton.
Hence there is a unique point $\hat q^a_\PR\in C^a$ such that $\ell^a(\hat q^a_\PR)=\hat\ell_\PR$.
Now, (\ref{EQ:P:primal-active-projection}) gives that $q\in\SOL_\PR$ if and only if
$q^a=\hat q^a_\PR$; so $\SOL_\PR=\{\hat q^a_\PR\} \times C^p(\hat q^a_\PR)$.
\end{proof}

Theorem~\ref{T:P} immediately follows from Propositions~\ref{P:primal}
and \ref{P:primal-active-projection}.

\subsection{Proof of Theorem~\ref{T:cc} (Convex conjugate by Lagrange duality)}\label{SS:fenchel}

In this section we prove Theorem~\ref{T:cc}
on the convex conjugate $\ell^\ast$, defined by
the {convex conjugate primal problem} ({cc-primal}, for short)
\begin{equation*}
 \ell^\ast(z)=\sup_{q\in\sx{\A}} (\sp{q}{z}-\ell(q)).
\end{equation*}
The proof is based on the following reformulation of the cc-primal
\begin{equation*}
    \ell^\ast(z)
    =
    \adjustlimits \sup_{x\ge 0} \inf_{\mu\in\RRR} K_z(x,\mu)
    =
    \adjustlimits \inf_{\mu\in\RRR} \sup_{x\ge 0} K_z(x,\mu),
\end{equation*}
where
\begin{equation*}
    K_z(x,\mu) = \sp{x}{z} + \sp{\nu}{\log x} - \mu \left(\sum x -1\right)
\end{equation*}
is the Lagrangian function; cf.~Lemma~\ref{L:F-cc-via-Lagrange-dual}.
Then we will show that the map
$\mu \mapsto \sup_{x\ge 0} K_z(x,\mu)$
is minimized at $\hat\mu(z)=\max\{\hatmu(z^a),\max(z^p)\}$;
cf.~Section~\ref{SSS:cc-proof}.
Structure of the solution set $\SOL_\cc(z)$ of the cc-primal is described in Section~\ref{SSS:cc-sol}.
Additional properties of the convex conjugate, which will be utilized in the proof
of Theorem~\ref{T:F}, are stated in Section~\ref{SSS:cc-additional}.

For every $z^a\in\RRR^{m_a}$ and $\mu>\max(z^a)$ put
\begin{equation}\label{EQ:xi-def}
    \xi(\mu) = \xi_{z^a}(\mu) \triangleq \sum \frac{\nu^a}{\mu-z^a}
\end{equation}
and recall that $\hat\mu(z)\triangleq \max\{\hatmu(z^a),\max(z^p)\}$,
where $\hatmu(z^a)>\max(z^a)$
solves $\xi(\mu)=1$.
Since
\begin{equation}\label{EQ:xi-is-monotone}
  \xi \text{ is strictly decreasing},
  \quad
  \lim_{\mu\searrow \max(z^a)} \xi(\mu)=\infty
  \quad\text{and}\quad
  \lim_{\mu\to\infty} \xi(\mu)=0,
\end{equation}
$\hatmu(z^a)$ is well-defined. We start with a simple lemma.

\begin{lemma}\label{L:cc-ell*(z+c)}
For every $z\in\RRR^m$ and $c\in\RRR$,
\begin{equation*}
    \hatmu(z^a+c)=\hatmu(z^a)+c,
    \quad
    \hat\mu(z+c) = \hat\mu(z)+c,
    \quad\text{and}\quad
    \ell^\ast(z+c) = \ell^\ast(z)+c.
\end{equation*}
\end{lemma}
\begin{proof}
The first two equalities follow from the facts
that $\xi_{z^a+c}(\cdot+c)=\xi_{z^a}(\cdot)$ and that $\max(z^p+c) = \max(z^p)+c$.
The final one is a trivial consequence of the definition of $\ell^\ast$; indeed,
since $\sp{q}{c}=c$ for every $q\in\sx{\A}$,
$\ell^\ast(z+c) = \sup_{q} (\sp{q}{(z+c)}-\ell(q)) =\ell^\ast(z)+c$.
\end{proof}

\subsubsection{Lagrange duality for the convex conjugate}\label{S:cc-lagrange-duality}
Assume that $\nu\in\sx{\A}$ and $z\in\RRR^m$ are given.
For $x\in\RRR^m_+$ put
\begin{equation*}
    h_z(x) \triangleq \sp{x}{z} - \ell(x) =  \sp{x}{z} + \sp{\nu^a}{\log x^a}
\end{equation*}
and define extended-real-valued functions
\begin{eqnarray*}
    &K_z:\RRR^m\times \RRR \to\RRRex,
    \quad
    K_z(x,\mu) &\triangleq
      \begin{cases}
        h_z(x) - \mu \left(\sum x -1\right)  &\text{if } x\in\RRR^m_+, \mu\in\RRR,
      \\
        \infty   &\text{otherwise};
      \end{cases}
\\
   &k_z:\RRR\to\RRRex,
   \quad\qquad\quad
   k_z(\mu) &\triangleq \sup_{x\in\RRR^m_+} K_z(x,\mu).
\end{eqnarray*}

\begin{lemma}[Lagrange duality for the convex conjugate] \label{L:F-cc-via-Lagrange-dual}
For every $\nu\in\sx{\A}$ and $z\in\RRR^m$,
\begin{equation*}
    \ell^\ast(z) = \inf_{\mu\in\RRR} k_z(\mu)
              = \adjustlimits \inf_{\mu\in\RRR} \sup_{x\ge 0} K_z(x,\mu).
\end{equation*}
\end{lemma}

\begin{proof}
We follow \cite[Sect.~4]{rockafellar1974}.
Denote by $\RRR^m_\oplus$ the subset
$\{x\in\RRR^m: x^a>0,x^p\ge0\}$ of $\RRR^m_+$.
Define $F_z:\RRR^m\times\RRR^2\to \RRRex$ and $f_z:\RRR^m\to\RRRex$ by
\begin{eqnarray*}
   F_z(x,u) &\triangleq&
   \begin{cases}
     -h_z(x)  &\text{if } x\in\RRR^m_\oplus, 1-u_2\le\sum x\le 1+u_1;
   \\
     \infty   &\text{otherwise};
   \end{cases}
\\
   f_z(x) &\triangleq& F_z(x,0).
\end{eqnarray*}
Note that, in the definition of $F_z$, $x\in\RRR^m_\oplus$ can be replaced by $x\in\RRR^m_+$;
indeed, if $x\in \RRR^m_+\setminus\RRR^m_\oplus$ then $x_i^a=0$ for some $i\in \Aa$ and hence $-h_z(x)=\infty$.
The set $D\triangleq\{(x,u)\in\RRR^m\times\RRR^2: x\ge 0, 1-u_2\le\sum x\le 1+u_1\}$
is closed convex (in fact, polyhedral) and the map $\tilde F_z:D\to\RRRex$, $(x,u)\mapsto -h_z(x)$
is convex and continuous, hence closed.
Since the epigraphs of $\tilde F_z$ and $F_z$ coincide,
\begin{equation}\label{EQ:F-cc-via-Lagrange-dual-proof1}
    F_z \quad\text{is convex and closed jointly in } x \text{ and } u.
\end{equation}

The corresponding optimal value function $\varphi_z:\RRR^2\to\RRRex$
is defined by (cf.~\cite[(4.7)]{rockafellar1974})
\begin{equation*}
    \varphi_z(u) \triangleq  \inf_{x\in\RRR^m} F_z(x,u).
\end{equation*}
We are going to show that
\begin{equation}\label{EQ:F-cc-via-Lagrange-dual-proof2}
    \varphi_z(0) = \liminf_{u\to 0} \varphi_z(u).
\end{equation}
To this end, take any $\eps\in(0,1)$ and $u\in\RRR^2$ with $\abs{u}<1$.
Assume that $u_1+u_2\ge 0$. Then
\begin{equation*}
    \varphi_z(u)
    = \inf_{\substack{x\ge 0\\\sum x\in[1-u_2,1+u_1]}} (-h_z(x))
    = \inf_{\theta\in[1-u_2,1+u_1]}  \ \inf_{x\in\sx{\A}} (-h_z(\theta x)).
\end{equation*}
For any $\theta>0$,
\begin{equation*}
\begin{split}
    \sup_{x\in\sx{\A}} \abs{h_z(\theta x)-h_z(x)}
    &\le
    \sup_{x\in\sx{\A}}\left(  \abs{\log\theta}  + \abs{1-\theta}\cdot \abs{\sp{x}{z}}   \right)
\\
    &=
    \abs{\log\theta}  + \abs{1-\theta} \cdot\max\abs{z}
    \triangleq \psi_z(\theta).
\end{split}
\end{equation*}
Since $\psi_z$ is continuous at $\theta=1$ and $\psi_z(1)=0$,
there is $\delta>0$ such that $\psi_z(\theta)<\eps$ for every $\theta\in[1-\delta,1+\delta]$.
Thus
$\abs{\varphi_z(u)-\varphi_z(0)} <\eps$
whenever $\abs{u}<\delta$ and $u_1+u_2\ge 0$. This gives
\begin{equation*}
    \varphi_z(0) = \lim_{\substack{u\to 0\\u_1+u_2\ge 0}} \varphi_z(u).
\end{equation*}
Since $\varphi_z(u)=\infty$ if $u_1+u_2<0$ (indeed, for such $u$, $F_z(x,u)=\infty$ for every $x$),
(\ref{EQ:F-cc-via-Lagrange-dual-proof2}) is proved.

The Lagrangian function $L_z:\RRR^m\times\RRR^2\to \RRRex$ associated with $F_z$ is defined by
(cf.~\cite[(4.2)]{rockafellar1974})
\begin{equation*}
    L_z(x,y) \triangleq \inf_{u\in\RRR^2}  \left(  F_z(x,u) + \sp{u}{y}   \right).
\end{equation*}
A simple computation yields (cf.~\cite[(4.4)]{rockafellar1974} with $f_0(x)=h_z(x)$, $f_1(x)=\sum x-1$,
and $f_2(x)=1-\sum x$, all restricted to $C=\RRR^m_\oplus$)
\begin{equation*}
    L_z(x,y) =
    \begin{cases}
        -h_z(x) + (y_1-y_2)\left(\sum x-1 \right)   &\text{if } x\in\RRR^m_\oplus, y\in\RRR^2_+;
    \\
        -\infty &\text{if } x\in\RRR^m_\oplus, y\not\in\RRR^2_+;
    \\
        \infty &\text{if } x\not\in\RRR^m_\oplus.
    \end{cases}
\end{equation*}
(Indeed, fix any $x\in\RRR^m$ and $y\in\RRR^2$.
If $x\not\in\RRR^m_\oplus$ then $F_z(x,u)=\infty$ for every $u$, hence $L_z(x,y)=\infty$.
If $x\in\RRR^m_\oplus$ and $y\not\ge 0$, then $F_z(x,u)=-h_z(x)$ whenever both $u_1,u_2>0$ are sufficiently large;
for such $u$, $\sp{u}{y}$ is not bounded from below (use that $y_1<0$ or $y_2<0$), hence $L_z(x,y)=-\infty$.
Finally, assume that $x\in\RRR^m_\oplus$ and $y\ge 0$. If $u_1+u_2<0$ then $F_z(x,u)=\infty$ by the definition of $F_z$.
If $u_1+u_2\ge 0$ then $\sp{u}{y}\ge y_1(\sum x-1) + y_2(1-\sum x)$, with the equality if $u_1=\sum x-1$, $u_2=1-\sum x$.
Thus $L_z(x,y) = -h_z(x) + (y_1-y_2)\left(\sum x-1\right)$.)

If we define (cf.~\cite[(4.6)]{rockafellar1974})
\begin{equation*}
    g_z:\RRR^2\to\RRRex,
    \qquad
    g_z(y) \triangleq \inf_{x\in\RRR^m} L_z(x,y),
\end{equation*}
then \cite[Thm.~7]{rockafellar1974}, (\ref{EQ:F-cc-via-Lagrange-dual-proof1}), and (\ref{EQ:F-cc-via-Lagrange-dual-proof2})
imply
\begin{equation*}
    \inf_{x\in\RRR^m} f_z(x)
    = \varphi_z(0)
    = \liminf_{u\to 0} \varphi_z(u)
    = \sup_{y\in\RRR^2} g_z(y).
\end{equation*}
By the definition of $f_z$, $\inf_x f_z(x)=-\ell^\ast(z)$; thus to finish the proof of the lemma
it suffices to  show that
\begin{equation}\label{EQ:F-cc-via-Lagrange-dual-proof3}
    \inf_{\mu\in\RRR} k_z(\mu)
    =
    - \sup_{y\in\RRR^2} g_z(y).
\end{equation}
If $y\not\in\RRR^2_+$ then $g_z(y)=-\infty$;
to see this, take arbitrary $x\in\RRR^m_\oplus$ and realize that $g_z(y)\le L_z(x,y)=-\infty$.
Fix any $y\in\RRR^2_+$. Then, for every $x\in\RRR^m_\oplus$, $K_z(x,y_1-y_2) = -L_z(x,y)$; hence
\begin{equation*}
    k_z(y_1-y_2) = -g_z(y).
\end{equation*}
Since $\{y_1-y_2:y\in\RRR^2_+\} = \RRR$, we have
\begin{equation*}
    \inf_{\mu\in\RRR} k_z(\mu)
    =
    \inf_{y\in\RRR^2_+} (-g_z(y))
    =
    -\sup_{y\in\RRR^2} g_z(y).
\end{equation*}
Thus (\ref{EQ:F-cc-via-Lagrange-dual-proof3}) is established and the proof of the lemma is finished.
\end{proof}

\subsubsection{Proof of Theorem~\ref{T:cc}}\label{SSS:cc-proof}

With Lemma~\ref{L:F-cc-via-Lagrange-dual}, Theorem~\ref{T:cc} can be proved. Recall that the theorem states that, for every $z\in\RRR^m$,
\begin{equation*}
    \ell^\ast(z) = -1 + \hat\mu(z) + I_{\hat\mu(z)}(\nu^a\,\|\,-z^a),
    \qquad\text{where}\quad
    \hat\mu(z)\triangleq \max\{ \hatmu(z^a), \max(z^p) \}.
\end{equation*}

\begin{proof}[Proof of Theorem~\ref{T:cc}]
Keep the notation from Section~\ref{S:cc-lagrange-duality}.
By Lemma~\ref{L:F-cc-via-Lagrange-dual}, $\ell^\ast(z) = \inf_{\mu\in\RRR} k_z(\mu)$,
and, using partial maximization,
\begin{equation*}
    k_z(\mu) = \sup_{x^a\ge 0} \tilde{k}_z(x^a,\mu),
    \quad\text{where}\quad
    \tilde{k}_z(x^a,\mu) \triangleq \sup_{x^p\ge 0} {K}_z((x^a,x^p),\mu).
\end{equation*}
Note that $K_z(x,\mu)=c+\sum_{i\in \Ap} x_i^p(z_i^p-\mu)$, where $c$ does not depend on $x^p$.
Thus
\begin{equation}\label{EQ:T-cc-L(x^a,mu)}
  \tilde{k}_z(x^a,\mu) =
  \begin{cases}
    \sp{x^a}{z^a} +\sp{\nu^a}{\log x^a} - \mu\left(  \sum x^a-1 \right)
    &\text{if } \mu\ge\max(z^p);
  \\
    \infty &\text{otherwise}.
  \end{cases}
\end{equation}

The second case immediately gives
\begin{equation}\label{EQ:T-cc-L(mu)-1}
    k_z(\mu)=\infty
    \qquad\text{if}\quad
    \mu<\max(z^p).
\end{equation}
If $\mu\ge\max(z^p)$ and $z_i^a-\mu\ge 0$ for some $i\in\Aa$, then $k_z(\mu)=\infty$
(use (\ref{EQ:T-cc-L(x^a,mu)}) and the fact
that, for any $a\ge 0$ and $b>0$, the map
$f(x)\triangleq ax+b\log x$
is strictly increasing and $\lim_{x\to\infty} f(x)=\infty$). That is
\begin{equation}\label{EQ:T-cc-L(mu)-2}
    k_z(\mu)=\infty
    \qquad\text{if}\quad
    \mu\le\max(z^a).
\end{equation}
Assume now that $\mu\ge\max(z^p)$ and $\mu>\max(z^a)$.
Note that
\begin{equation*}
    k_z(\mu)=
    \max_{x^a> 0} \tilde{k}_z(x^a,\mu).
\end{equation*}
By (\ref{EQ:T-cc-L(x^a,mu)}),
$\tilde{k}_z(\cdot,\mu)$ is differentiable and strictly concave
on the open set $\{x^a\in\RRR^{m_a}: x^a>0\}$
(use that the Hessian, which is equal to $\diag(-\nu_i^a/(x_i^a)^2)_{i=1}^{m_a}$,
is negative definite).
Thus the basic necessary and sufficient condition for unconstrained optimization
\cite[p.~258]{bertsekas} gives that
\begin{equation*}
  k_z(\mu) = \tilde{k}_z(\hat x^a,\mu),
\end{equation*}
where $\hat x^a= \hat x^a_\mu>0$ is the unique solution
of $\nabla_{x^a} \tilde{k}_z(\hat{x}^a,\mu)=0$; that is,
\begin{equation*}
    z_i^a + \frac{\nu_i^a}{\hat x_i^a} - \mu = 0
    \qquad\text{for every} \quad
    i\in\Aa.
\end{equation*}
The above equation immediately yields
\begin{equation*}
    \hat x^a  = \frac{\nu^a}{\mu-z^a}
\end{equation*}
and
\begin{equation}\label{EQ:T-cc-L(mu)-3}
    k_z(\mu)=-1 + \mu + I_\mu(\nu^a\,\|\,-z^a)
    \qquad\text{if}\quad
    \mu\ge\max(z^p),\mu>\max(z^a).
\end{equation}

Put $J\triangleq \{\mu\in\RRR: \mu>\max(z^a), \mu\ge\max(z^p)\}$ and
$\underline\mu\triangleq \max(z)$; then
either $J=(\underline\mu,\infty)$ or $J=[\underline\mu,\infty)$.
Equations (\ref{EQ:T-cc-L(mu)-1}), (\ref{EQ:T-cc-L(mu)-2}), and (\ref{EQ:T-cc-L(mu)-3})
yield
\begin{equation*}
    k_z(\mu)=
    \begin{cases}
        -1 + \mu + I_\mu(\nu^a\,\|\,-z^a)
        &\text{if }\mu\in J;
    \\
        \infty
        &\text{otherwise}.
    \end{cases}
\end{equation*}
For $\mu>\underline{\mu}$ we have
\begin{equation*}
    k_z'(\mu) = 1 -\xi(\mu).
\end{equation*}
Since $\xi$ is decreasing on $J$ and $\xi(\hatmu(z^a))=1$,
$k_z(\mu)$ is decreasing on $[\underline{\mu},\hatmu(z^a)]$
and increasing on $[\hatmu(z^a), \infty)$, provided $\underline{\mu}\le \hatmu(z^a)$;
otherwise $k_z(\mu)$ is increasing on $[\underline{\mu},\infty]$.
Thus,
\begin{equation*}
    \ell^\ast(z)
    =\inf_{\mu\in\RRR} k_z(\mu)
    =
    \begin{cases}
        k_z(\hatmu(z^a))  &\text{if }\hatmu(z^a) \ge \underline{\mu};
    \\
        k_z(\underline\mu) &\text{otherwise}.
    \end{cases}
\end{equation*}
So $\argmin k_z=\max\{\hatmu(z^a),\max(z^p)\} = \hat\mu(z)$ and
Theorem~\ref{T:cc} is proved.
\end{proof}

\begin{corollary}\label{C:BD-Lemma21}
If $\nu>0$ then, for every $z\in\RRR^m$,
\begin{equation*}
\begin{split}
    \ell^\ast(z)
    &= -1 + \hatmu(z) + I_{\hatmu(z)} (\nu\,\|\,-z)
\\
    &= -1 + \hatmu(z) + \sp{\nu}{\log\nu}  - \sp{\nu}{\log(\hatmu(z)-z)}.
\end{split}
\end{equation*}
\end{corollary}

Corollary~\ref{C:BD-Lemma21} was first proved by El~Barmi and Dykstra~\cite[Lemma~2.1]{elbarmi-dykstra}.

\subsubsection{The structure of the solution set $\SOL_\cc(z)$}\label{SSS:cc-sol}
The structure of the cc-primal solution set $\SOL_\cc(z)$, defined by (\ref{EQ:F-cc-finite-sol}), is described.
First, recall the definition (\ref{EQ:xi-def}) of $\xi$.

\begin{proposition}\label{P:cc-SOL}
For every $z\in\RRR^m$, $\SOL_\cc(z)$ is a nonempty compact set and
\begin{equation*}
\begin{split}
    \SOL_\cc(z) = \{\hat{q}^a_\cc(z)\}
    \times
    \{
      &q^p\ge 0: \sum q^p = 1-\xi(\hat\mu(z)),
\\
      &q^p_i=0 \text{ whenever } z^p_i\ne\hat\mu(z)
    \},
\end{split}
\end{equation*}
where
\begin{equation*}
    \hat{q}^a_\cc(z) \triangleq \frac{\nu^a}{\hat\mu(z)-z^a}.
\end{equation*}
In particular, if $\hatmu(z^a)\ge \max(z^p)$ then
$\SOL_\cc(z)=\{(\hat{q}^a_\cc(z),0^p)\}$
is a singleton.
\end{proposition}

\begin{proof}
Fix $z\in\RRR^m$ and put $\mu\triangleq \hat\mu(z)$, $\hat q^a_\cc\triangleq \hat{q}^a_\cc(z)$;
then, by Theorem~\ref{T:cc},
$\ell^\ast(z) = -1+\mu+I_\mu(\nu^a\,\|\,-z^a)$.
The fact that $\SOL_\cc(z)$ is nonempty and compact follows from Theorem~\ref{T:bertsekas}.

Take any $\bar q\in\SOL_\cc(z)$.
Then $\bar q^a>0$ and, for every $q\in\sx{\A}$,
\begin{equation*}
  0 \ge f_z'(\bar q;q-\bar q)
  = \sp{z}{(q-\bar q)} + \sum\frac{\nu^a q^a}{\bar q^a} - 1.
\end{equation*}
If also $q\in\SOL_\cc(z)$ then, analogously,
\begin{equation*}
  0 \ge \sp{z}{(\bar q-q)} + \sum\frac{\nu^a \bar q^a}{q^a} - 1.
\end{equation*}
Thus, by combining these two inequalities,
\begin{equation*}
  \sum\nu^a\left(\frac{q^a}{\bar q^a}+\frac{\bar q^a}{q^a}\right) \le 2.
\end{equation*}
For positive $x$ and $y$, $(x/y+y/x)\ge 2$, with equality if and only if $x=y$.
Thus, $q^a=\bar q^a$ for every $\bar q,q\in\SOL_\cc(z)$; hence,
\begin{equation}\label{EQ:cc-SOLa-singleton}
    \pi^a(\SOL_\cc^a(z))
    \quad\text{is a singleton}.
\end{equation}

Distinguish two cases.
First assume that $\mu=\hatmu(z^a)\ge\max(z^p)$
and put $\bar q\triangleq (\hat q^a_\cc,0^p)$.
Then $\bar q\in\sx{\A}$ (use that $\sum\bar q=\xi(\mu)=1$) and
\begin{eqnarray*}
  f_z(\bar q)
  &=& \sp{\hat q^a_\cc}{z^a} + \sp{\nu^a}{\log(\hat q^a_\cc)}
  = \sum\frac{(z^a-\mu+\mu)\nu^a}{\mu-z^a} + I_\mu(\nu^a\,\|\,-z^a)
\\
  &=& -1+\mu + I_\mu(\nu^a\,\|\,-z^a)
  =\ell^\ast(z).
\end{eqnarray*}
Hence $\bar q\in \SOL_\cc(z)$ and, since $\sum\bar q^a=1$,
$\SOL_\cc(z)=\{\bar q\}$ by (\ref{EQ:cc-SOLa-singleton}).

Assume now that $\mu=\max(z^p)>\hatmu(z^a)$;
then $\gamma\triangleq\xi(\mu)<\xi(\hatmu(z^a))=1$. Take any $\bar q\in\sx{\A}$
with $\bar q^a=\hat q^a_\cc$ and note that $\sum \bar q^p = 1-\gamma$.
An argument analogous to the first case gives
\begin{equation*}
    f_z(\bar q)
    =
    \sp{\bar q^p}{z^p} + \sp{\hat q^a}{z^a} + \sp{\nu^a}{\log(\hat q^a_\cc)}
  = \sp{\bar q^p}{z^p} -1+\mu\gamma + I_\mu(\nu^a\,\|\,-z^a).
\end{equation*}
Hence $\bar q\in\SOL_\cc(z)$ if and only if
$\sp{\bar q^p}{z^p}=\mu(1-\gamma)$. Since $\sum \bar q^p = 1-\gamma$, the last condition
is equivalent to the fact that $\bar q^p_i=0$ for every $i\in\Ap$ with $z^p_i<\mu$.
In view of (\ref{EQ:cc-SOLa-singleton}) this proves the proposition.
\end{proof}

\subsubsection{Additional properties of the convex conjugate}\label{SSS:cc-additional}
Some additional properties of the convex conjugate $\ell^\ast$, concerning its
monotonicity and differentiability, are summarized in the next lemmas.

\begin{lemma}\label{L:cc-nondecreasing}
The following is true for every $z,\tilde z\in\RRR^m$:
\begin{enumerate}
  \item[(a)]
    If $z^a\ge\tilde z^a$ then $\hatmu(z^a)\ge \hatmu(\tilde z^a)$,
    with the equality if and only if $z^a=\tilde z^a$.
  \item[(b)]
    If $z\ge\tilde z$ then $\ell^\ast(z)\ge \ell^\ast(\tilde z)$ (that is, $\ell^\ast$ is nondecreasing),
    with the equality if and only if $z^a=\tilde z^a$ and $\hat\mu(z)=\hat\mu(\tilde z)$.
\end{enumerate}
\end{lemma}
\begin{proof}
(a) If $z^a\ge\tilde z^a$ then $\xi_{z^a}(\mu) \ge \xi_{\tilde z^a}(\mu)$ for every
$\mu>\max(z^a)$; moreover, $\xi_{z^a}(\mu) = \xi_{\tilde z^a}(\mu)$ for some $\mu>\max(z^a)$
if and only if $z^a=\tilde z^a$. Using (\ref{EQ:xi-is-monotone}), (a) follows.

(b) Assume that $z\ge\tilde z$. Then $\sp{z}{q}\ge \sp{\tilde z}{q}$ for every $q\in\sx{\A}$,
hence $\ell^\ast(z)\ge \ell^\ast(\tilde z)$.
If $z^a=\tilde z^a$ and $\hat\mu(z)=\hat\mu(\tilde z)$
then $\ell^\ast(z)= \ell^\ast(\tilde z)$ by Theorem~\ref{T:cc}.
It suffices to prove that if $z^a\ne \tilde z^a$ or $\hat\mu(z)\ne \hat\mu(\tilde z)$ then
$\ell^\ast(z) > \ell^\ast(\tilde z)$.

To this end, fix some $\tilde q\in\SOL_\cc(\tilde z)$.
If $z^a\ne \tilde z^a$ then
$\ell^\ast(z)\ge \sp{\tilde{q}}{z} -\ell(\tilde q)\ge\sp{\tilde q^a}{(z^a-\tilde z^a)}
+ \ell^\ast(\tilde z)$.
Since $\tilde q^a>0$ and $(z^a-\tilde z^a)\ge 0$ is not zero, we have $\ell^\ast(z) > \ell^\ast(\tilde z)$.
Finally, assume that $z^a= \tilde z^a$ and $\hat\mu(z)>\hat\mu(\tilde z)$.
The map
\begin{equation*}
  g:[\hatmu({z}^a),\infty)\to\RRR,
  \qquad
  g(\mu)\triangleq \mu+I_\mu(\nu^a\,\|\,-z^a)
\end{equation*}
is strictly increasing (indeed, $g'(\mu)=1-\xi(\mu)>0$ for $\mu>\hatmu({z}^a)$ by
(\ref{EQ:xi-is-monotone})).
Thus, by Theorem~\ref{T:cc} and the fact that
$\hat\mu(\tilde z)\ge\hatmu(\tilde z^a)=\hatmu(z^a)$,
it follows that
  $\ell^\ast(z)-\ell^\ast(\tilde z)
  = g(\hat\mu(z))- g(\hat\mu(\tilde z))>0$.
\end{proof}

Since the convex conjugate $\ell^\ast$ is finite-valued and convex, by
\cite[Thm.~10.4]{rockafellar} it is locally Lipschitz. The following lemma claims
that $\ell^\ast$ is even globally Lipschitz with the Lipschitz constant equal to $1$.
(Here we assume that $\RRR^m$ is equipped with the sup-norm
$\norm{x}_\infty=\max\abs{x_i}$.)
\begin{lemma}\label{L:cc-lipschitzness}
The convex conjugate $\ell^\ast:\RRR^m\to\RRR$ is a (finite-valued) convex function
which is Lipschitz with $\Lip(\ell^\ast) = 1$.
\end{lemma}
\begin{proof}
The fact that $\ell^\ast$ is always finite is obvious.
To prove that $\ell^\ast$ is Lipschitz-$1$, fix any $z,z'\in\RRR^m$.
Then, for any $q\in\SOL_\cc(z)$,
\begin{eqnarray*}
  \ell^\ast(z)-\ell^\ast(z')
  &\le&
  (\sp{z}{q}-\ell(q))  -   (\sp{z'}{q}-\ell(q))
  =
  \sp{(z-z')}{q}
\\
  &\le&
  \norm{z-z'}_\infty.
\end{eqnarray*}
Analogously $\ell^\ast(z')-\ell^\ast(z)\le\norm{z-z'}_\infty$. Thus $\Lip(\ell^\ast) \le 1$.
Since $\Lip(\ell^\ast) \ge 1$ by Lemma~\ref{L:cc-ell*(z+c)}
(use that $\norm{(c,c,\dots,c)}_\infty=\abs{c}$ for $c\in \RRR$),
we have $\Lip(\ell^\ast) = 1$.
\end{proof}

\begin{lemma}\label{L:cc-mu-continuous}
The map $\hatmu:\RRR^{m_a}\to\RRR$ is differentiable (even $C^\infty$) and
\begin{equation*}
    \nabla \hatmu (z^a) = \frac{1}{\sum_i\alpha_i} (\alpha_1,\dots,\alpha_{m_a}),\qquad
    \alpha_k \triangleq \frac{\nu_k}{(\hatmu(z^a)-z^a_k)^2}
    \quad (1\le k\le m_a).
\end{equation*}
The map $\hat\mu:\RRR^{m}\to\RRR$ is continuous.
\end{lemma}
\begin{proof} Since $\hat\mu(z)=\max(\hatmu(z^a),\max(z^p))$, the continuity of
$\hat\mu$ follows from the continuity of $\hatmu$;
thus it suffices to prove the first part of the lemma.

To this end, we may  assume
that $\Ap=\emptyset$, that is, $\nu>0$ and $m_a=m$.
Put $\Omega\triangleq\{(z,\mu): z\in\RRR^m, \mu>\max(z)\}$ and define $F:\Omega\to(0,\infty)$ by
\begin{equation*}
    F(z,\mu) \triangleq \sum_{i=1}^m \frac{\nu_i}{\mu-z_i} \,.
\end{equation*}
Note that, for every $z$, $\hatmu(z)$ is the unique solution of $F(z,\mu)=1$ such that
$(z,\mu)\in\Omega$. The set $\Omega$ is open and connected and $F$ is $C^\infty$ on $\Omega$.
Moreover,
\begin{equation*}
    \nabla_z F(z,\mu) = \frac{\nu}{(\mu-z)^2}
    \qquad\text{and}\qquad
    \nabla_\mu F(z,\mu) = -\sum_{i=1}^m \frac{\nu_i}{(\mu-z_i)^2} \ne 0.
\end{equation*}
By the local implicit function theorem \cite[Prop.~1.1.14]{bertsekas},
$\hatmu$ is $C^\infty$ on $\RRR^m$ and $\nabla\hatmu(z)$ is given by
$-\nabla_z F(z,\hatmu(z)) \cdot [\nabla_\mu F(z,\hatmu(z))]^{-1}$. From this the lemma follows.
\end{proof}

\begin{lemma}\label{L:cc-differentiability-Danskin}
For every $z,v\in\RRR^m$,
the subgradient and the directional derivative of $\ell^\ast$ are given by
\begin{equation*}
    \partial\ell^\ast(z)=\SOL_\cc(z)
    \qquad\text{and}\qquad
    (\ell^\ast)'(z;v) = \max_{q\in\SOL_\cc(z)} \sp{q}{v}.
\end{equation*}
In particular, if $\hatmu(z^a)\ge \max(z^p)$ then
$\ell^\ast$ is differentiable at $z$
and
\begin{equation*}
    \nabla \ell^\ast(z) = (\hat{q}^a_\cc(z),0^p).
\end{equation*}
\end{lemma}
\begin{proof}
This is a consequence of Danskin's theorem \cite[Prop.~4.5.1]{bertsekas}.
To see this, fix $\bar z\in\RRR^m$ and take $\eps>0$ such that $\hat q^a_\cc(\bar z)>\eps$.
Put $\sx{\A}^\eps\triangleq\{q\in\sx{\A}: q^a\ge\eps\}$ and define
\begin{eqnarray*}
  \ &\varphi:\RRR^m\times\sx{\A}^\eps\to\RRR,\qquad
  &\varphi(z,q)\triangleq \sp{z}{q}-\ell(q),
 \\
  \ &f:\RRR^m\to \RRR,\qquad\qquad\quad
  &f(z)\quad\triangleq\max_{q\in\sx{\A}^\eps} \varphi(z,q).
\end{eqnarray*}
The map $\varphi$ is continuous and the partial functions $\varphi(\cdot,q)$ are (trivially)
convex and differentiable for every $q\in\sx{\A}^\eps$ with
(continuous) $\nabla_z \varphi(z,q)=q$.
Thus, by Danskin's theorem,
\begin{equation*}
    f'(\bar z;v) = \max_{q\in\SOL_\cc(\bar z)} \sp{q}{v}
    \qquad (v\in\RRR^m)
\end{equation*}
and
\begin{equation*}
    \partial f(\bar z)
    = \conv\{\nabla_z \varphi(\bar z,q): q\in\SOL_\cc(\bar z)\}
    =\SOL_\cc(\bar z).
\end{equation*}

Note that $\hat\mu(\cdot)$ is continuous by Lemma~\ref{L:cc-mu-continuous}
and hence also $\hat q^a_\cc(\cdot)$ is continuous.
Thus $\hat q^a_\cc(z)>\eps$ on a neighborhood $U$ of $\bar z$, and so
$\ell^\ast=f$ on $U$. This proves the first assertion.

If $\hatmu(\bar z^a)\ge \max(\bar z^p)$ then $\SOL_\cc(\bar z)$
is a singleton $\{\bar q\}$, where $\bar q\triangleq (\hat q^a_\cc(\bar z),0^p)$.
In such a case Danskin's theorem gives that $\ell^\ast$ is differentiable at $\bar z$
with $\nabla \ell^\ast(z)=\nabla_z \varphi(z,\bar q)=\bar q$.
So also the second assertion of the lemma is proved.
\end{proof}

The convex conjugate $\ell^\ast$ is not strictly convex, even if $\nu>0$.
This is so because $\ell^\ast(z+c)=\ell^\ast(z)+c$ for every constant $c$,
cf.~Lemma~\ref{L:cc-ell*(z+c)}.
However, the following holds.

\begin{lemma}\label{L:F-cc-derivative}
Let $\nu>0$. Then $\ell^\ast$ is $C^\infty$ and, for every $z\in\RRR^m$,
\begin{equation*}
  \nabla \ell^\ast(z)=\frac{\nu}{\hatmu(z)-z},
  \qquad
  \nabla^2 \ell^\ast(z)=\diag(\alpha) - \frac{1}{\sum\alpha} \alpha \alpha',
\end{equation*}
where $\alpha\triangleq (\alpha_1,\dots,\alpha_m)'$,
$\alpha_i \triangleq  \nu_i/(\hatmu(z)-z_i)^2$.
Consequently, for every $x\in\RRR^m$, $x'(\nabla^2 \ell^\ast(z)) x\ge 0$
and
\begin{equation*}
    x'(\nabla^2 \ell^\ast(z)) x
    = 0
    \qquad\text{if and only if}\qquad
    x=(c,\dots,c)' \text{ for some } c\in\RRR.
\end{equation*}
\end{lemma}
\begin{proof}
We use the notation from the proof of Lemma~\ref{L:cc-mu-continuous}.
By Corollary~\ref{C:BD-Lemma21} and Lemma~\ref{L:cc-mu-continuous},
\begin{equation*}
    \frac{\partial \ell^\ast(z)}{\partial z_k}
    =
    \frac{\partial \hatmu(z)}{\partial z_k}
    \left(
      1-\sum_i \frac{\nu_i}{\hatmu(z)-z_i}
    \right)
    + \frac{\nu_k}{\hatmu(z)-z_k}
    =
    \frac{\nu_k}{\hatmu(z)-z_k}
\end{equation*}
since $\sum_i \nu_i/(\hatmu(z)-z_i) = F(z,\hatmu(z))=1$.
Further,
\begin{equation*}
    \frac{\partial^2 \ell^\ast(z)}{\partial z_k \partial z_l}
    =
    -\frac{\nu_k}{(\hatmu(z)-z_k)^2}   \frac{\partial \hatmu(z)}{\partial z_l}
    =
    -\frac{\alpha_k\alpha_l}{\sum \alpha}
\end{equation*}
for $k\ne l$, and
\begin{equation*}
    \frac{\partial^2 \ell^\ast(z)}{\partial z_k^2}
    =
    -\frac{\nu_k}{(\hatmu(z)-z_k)^2}
     \left( \frac{\partial \hatmu(z)}{\partial z_k} - 1\right)
    =
    \alpha_k -\frac{\alpha_k^2}{\sum \alpha}\,.
\end{equation*}
Thus the first assertion of the lemma is proved.

To prove the second part of the lemma, fix any $z,x\in\RRR^m$ and put
$a\triangleq \sum \alpha$,
$A\triangleq a \nabla^2 \ell^\ast(z)$.
Then
\begin{equation*}
    x'Ax
    =
    a\sum _i \alpha_i x_i^2 - \left(  \sum_i \alpha_i x_i \right)^2
    =
    a^2 \left [\sum_i w_i x_i^2 - \left(\sum_i w_i x_i\right)^2  \right]
\end{equation*}
($w_i\triangleq \alpha_i/a$). Since $a\ne 0$, Jensen's inequality gives that
$x'Ax\ge 0$, and $x'Ax=0$ if and only if $x_1=\dots=x_m$. The lemma is proved.
\end{proof}

\subsection{Proof of Theorem~\ref{T:F} (Relation between \refF{} and \refPR)}\label{SS:Fenchel}

The proof of Theorem~\ref{T:F} goes through several lemmas.
First, in Lemma~\ref{L:F-primal=fenchel},
the extremality relation between \refPR{} and \refF{} is established
using the Primal Fenchel duality theorem.
Then the minimax equality for $L$ (cf.~(\ref{EQ:F-L-def})) is proved, see Corollary~\ref{C:F-L-minimax}.
Proposition~\ref{P:F-Scc-SP} gives a relation between
the solution set $\SOL_\PR$ of the primal
and the solution set $\SOL_\cc(-\hat y)$ of the convex conjugate $\ell^\ast$
for $\hat y\in\SOL_\F$.
Lemma~\ref{L:F-primal->dual} provides a key for establishing
the second part of Theorem~\ref{T:F}.
The structure of the solution set $\SOL_\F$ is described in Lemma~\ref{L:F-SOL}.

\subsubsection{Extremality relation}\label{SSS:F-extremality-rel}
In the following denote by $\sx{\A}^+$ the set $\{q\in\sx{\A}: q^a>0\}$;
note that neither the primal nor the convex conjugate are affected
by restricting to $\sx{\A}^+$.

\begin{lemma}\label{L:C*} Let $C\subseteq \sx{\A}$.
If $y\in C^\ast$ then $y+\RRR^m_{-}\subseteq C^\ast$.
\end{lemma}
\begin{proof}
Take any $y\in C^\ast$ and $z\le 0$. Then, for every $q\in C$,
\begin{equation*}
  \sp{q}{(y+z)}
  =
  \sp{q}{y} + \sp{q}{z}
  \le \sp{q}{y} \le 0.
\end{equation*}
Thus $y+z\in C^\ast$.
\end{proof}

\begin{lemma}\label{L:F-primal=fenchel}
    For every convex closed set $C\subseteq \sx{\A}$ having support $\A$
    and for every $\nu\in\sx{\A}$ it holds that
    \begin{equation*}
        \hat\ell_\F = -\hat\ell_\PR
        \qquad\text{and}\qquad
        \SOL_\F\ne\emptyset.
    \end{equation*}
\end{lemma}
\begin{proof}
Write the primal problem \refPR{} in the form
\begin{equation*}
    \hat\ell_\PR = \inf_{x\in\sx{\A}^+\cap K_C} \ell(x),
    \qquad\text{where}\quad
    K_C\triangleq \{\alpha q: q\in C, \alpha\ge 0\}
\end{equation*}
is the convex cone generated by $C$. Since $C$ is convex, there is $q\in\ri(C)$
\cite[Prop.~1.4.1(b)]{bertsekas}. Further, $\A$ is the support of $C$,
so there is $q'\in C$ satisfying $q'>0$.
By \cite[Prop.~1.4.1(a)]{bertsekas},
$\tilde{q}\triangleq \frac 12(q+q') >0$ belongs to $\ri(C)$.
Further, by \cite[p.~50]{rockafellar},
\begin{equation*}
    \ri(K_C) = \{\alpha q: q\in\ri(C), \alpha>0\};
\end{equation*}
thus $\tilde{q}\in \ri(K_C)$.
Since trivially $\ri(\sx{\A}^+) = \{q\in\sx{\A}: q>0\}$, it holds that
\begin{equation*}
    \ri(\sx{\A}^+) \cap \ri(K_C)\ne\emptyset.
\end{equation*}
Now the Primal Fenchel duality theorem \cite[Prop.~7.2.1, pp.~439--440]{bertsekas},
applied to the convex set $\sx{\A}^+$,
the convex cone $K_C$ and the real-valued
convex function $\ell|_{\sx{\A}^+}$,
gives that
\begin{equation*}
    \inf_{x\in\sx{\A}^+\cap K_C} \ell(x)
    =
    \sup_{y\in K_C^\ast} (-\ell^\ast(-y))
    =
    -\inf_{y\in C^\ast} \ell^\ast(-y)
\end{equation*}
and that the supremum in the right-hand side is attained. That is,
$\hat\ell_\PR = -\hat\ell_\F$ and $\SOL_\F\ne\emptyset$.
\end{proof}

\subsubsection{Minimax equality for $L$}
Let $L:C^\ast \times \sx{\A}^+\to\RRR$ be given by
\begin{equation}\label{EQ:F-L-def}
    L(y,q) \triangleq -\sp{y}{q} - \ell(q) =  -\sp{y}{q} + \sp{\nu^a}{\log q^a}.
\end{equation}

\begin{lemma}\label{L:F-L-minimax1}
For every $y\in C^\ast$ and $q\in \sx{\A}^+$,
\begin{equation*}
    \sup_{q\in\sx{\A}^+} L(y,q)=\ell^\ast(-y)
    \qquad\text{and}\qquad
    \inf_{y\in C^\ast} L(y,q) =
    \begin{cases}
      -\ell(q) &\text{if } q\in C;
    \\
      -\infty &\text{if } q\not\in C.
    \end{cases}
\end{equation*}
\end{lemma}
\begin{proof}
The first equality is immediate since if $q\in\sx{\A}\setminus\sx{\A}^+$ then
$\ell(q)=\infty$. To show the second one, realize that
$\inf_{y\in C^\ast} L(y,q)=-\ell(q) -\sup_{y\in C^\ast} \sp{y}{q}$.
If $q\in C$ then $\sp{y}{q}\le 0$ for every $y\in C^\ast$, and $\sp{y}{q}=0$ for $y=0\in C^\ast$;
hence $\inf_{y\in C^\ast} L(y,q)= -\ell(q)$.
If $q\in \sx{\A}^+\setminus C$, there is $\bar y\in C^\ast$ with $\sp{\bar y}{q}>0$
(use that $C$ is convex closed, thus $C^{\ast\ast}=K_C^{\ast\ast}=K_C$ \cite[Prop.~3.1.1]{bertsekas}
and $K_C\cap\sx{\A}=C$). Since $\lambda \bar y\in C^\ast$ for every $\lambda\ge 0$,
\begin{equation*}
    \sup_{y\in C^\ast} \sp{y}{q}
    \ge \sup_{\lambda\ge 0} \lambda \sp{\bar y}{q}
    = \infty.
\end{equation*}
Thus $\inf_{y\in C^\ast} L(y,q)=-\infty$ provided $q\in \sx{\A}^+\setminus C$.
\end{proof}

Lemmas~\ref{L:F-primal=fenchel} and \ref{L:F-L-minimax1} immediately yield the minimax equality for $L$.

\begin{corollary}\label{C:F-L-minimax}
We have
\begin{equation*}
    \hat\ell_\F
    = \adjustlimits\inf_{y\in C^\ast} \sup_{q\in\sx{\A}^+} L(y,q)
    = \adjustlimits\sup_{q\in\sx{\A}^+}\inf_{y\in C^\ast} L(y,q)=-\hat\ell_\PR.
\end{equation*}
\end{corollary}

\begin{lemma}\label{L:F-L-saddle-points}
For every $\bar y\in\RRR^m$ and $\bar q\in\sx{\A}^+$,
\begin{equation*}
    (\bar y,\bar q) \text{ is a saddle point of } L
    \quad\iff\quad
    \bar y\in\SOL_\F
    \text{ and }
    \bar q \in \SOL_\PR.
\end{equation*}
\end{lemma}

Recall that $(\bar y,\bar q)$ is a \emph{saddle point} of $L$
\cite[Def.~2.6.1]{bertsekas} if, for every $y\in C^\ast$ and $q\in\sx{\A}^+$,
\begin{equation}\label{EQ:F-saddle-point-of-L}
    L(\bar y, q)\le L(\bar y,\bar q) \le L(y,\bar q).
\end{equation}

\begin{proof}
This result is an immediate consequence of Corollary~\ref{C:F-L-minimax},
Lemma~\ref{L:F-L-minimax1},
and \cite[Prop.~2.6.1]{bertsekas}.
\end{proof}

\subsubsection{Relation between $\SOL_\PR$ and $\SOL_\cc(-\hat y)$}
\begin{proposition}\label{P:F-Scc-SP}
Let $\hat y\in\SOL_\F$. Then
\begin{equation*}
  \SOL_\PR\subseteq \SOL_\cc(-\hat y).
\end{equation*}
\end{proposition}
\begin{proof}
Take any $\hat y\in\SOL_\F$ and any $\hat q\in\SOL_\PR$.
By Lemma~\ref{L:F-L-saddle-points}, $(\hat y,\hat q)$ is a saddle point of~$L$.
Hence, by (\ref{EQ:F-saddle-point-of-L}),
\begin{equation*}
    -\sp{\hat y}{\hat q} - \ell(\hat q)
    = L(\hat y,\hat q)
    \ge L(\hat y,q)
    = -\sp{\hat y}{q} - \ell(q)
\end{equation*}
for any $q\in \sx{\A}^+$.
Thus $\ell^\ast(-\hat y) = L(\hat y,\hat q)$ and so $\hat q\in \SOL_\cc(-\hat y)$.
\end{proof}

\subsubsection{From $\hat q_\PR$ to $\hat y_\F$}\label{SSS:F-from-qP-to-yF}
The proof of the following lemma is inspired by that of \cite[Thm.~2.1]{elbarmi-dykstra}.

\begin{lemma}\label{L:F-primal->dual}
Let $\nu\in\sx{\A}$ and $\hat q^a_\PR$ be as in Theorem~\ref{T:P}. Then
\begin{equation*}
    \hat y_\F \triangleq (\hat y^a_\F, -1^p) \in\SOL_\F,
    \qquad\text{where}\quad
    \hat y^a_\F\triangleq \frac{\nu^a}{\hat q^a_\PR}-1^a \,.
\end{equation*}
Further, $\hat\mu(-\hat y_\F)=1$, $\hat y_\F\bot \SOL_\PR$,  and
\begin{enumerate}
  \item[(a)]
    if $\sum \hat q^a_\PR=1$ then $\hatmu(-\hat y^a_\F)=1$;
  \item[(b)]
    if $\sum \hat q^a_\PR<1$ then $\hatmu(-\hat y^a_\F)<1$.
\end{enumerate}
\end{lemma}

\begin{proof}
Take any $\hat q\in\SOL_\PR$. Then $\hat q\in\dom(\ell)$ and,
by Lemma~\ref{L:f-props}(d),
\begin{equation*}
\begin{split}
    0
    &\le     \ell'(\hat q; q-\hat q)
    = -\sum_{i\in\Aa} \frac{\nu_i^a(q_i^a-\hat q_i^a)}{\hat q_i^a}
    = 1-\sum_{i\in\Aa} \frac{\nu_i^a q_i^a}{\hat q_i^a}
\\
    &= \sum_{i\in\Aa} q_i^a \left(  1-\frac{\nu_i^a}{\hat q_i^a} \right)
       +\sum_{i\in\Ap} q_i^p
    = -\sp{q}{\hat y_\F}
\end{split}
\end{equation*}
for every $q\in C$. Hence $\hat y_\F\in C^\ast$.
Further,
\begin{equation*}
    \sp{\hat q}{\hat y_\F}
    =
    \sum_{i\in \Aa} \hat q_i^a \left(  \frac{\nu_i^a}{\hat q_i^a} - 1 \right)
    + \sum_{i\in \Ap} \hat q_i^p \cdot (-1)
    = \sum \nu^a - \sum \hat q
    = 0
\end{equation*}
and so $\hat y_\F \bot \hat q$.

We claim that
\begin{equation}\label{EQ:L:F-primal->dual}
    \hatmu(-\hat y^a_\F)   \
    \begin{cases}
        = 1  &\text{if } \sum \hat q^a=1;
    \\
        < 1  &\text{if } \sum \hat q^a<1.
    \end{cases}
\end{equation}
Since $\hat y^a_\F=(\nu^a/\hat q^a_\PR)-1^a>-1^a$, it holds that $1>\max(-\hat y^a_\F)$. Moreover,
\begin{equation*}
    \xi_{-y^a_\F}(1)
    = \sum \frac{\nu^a}{1+y^a_\F}
    =\sum\hat q^a_\PR.
\end{equation*}
Thus (\ref{EQ:L:F-primal->dual}) follows from (\ref{EQ:xi-is-monotone}).

Assume that $\sum\hat q^a=1$. Then $\hatmu(-\hat y^a_\F)=1$. Since $-\hat y^p_\F=1^p$,
$\hat\mu(-\hat y_\F)=1$.
Theorem~\ref{T:cc} and Lemma~\ref{L:F-primal=fenchel} yield
\begin{equation*}
    \ell^\ast(-\hat y_\F)
    = I_1(\nu^a\,\|\,\hat y^a_\F)
    = \sp{\nu^a}{\log\hat q^a}
    = -\ell(\hat q) = -\hat \ell_\PR = \hat\ell_\F.
\end{equation*}
Since $\hat y_\F\in C^\ast$, we have $\hat y_\F\in\SOL_\F$.

If $\sum\hat q^a<1$ then $\hatmu(-\hat y^a_\F)<1$ by (\ref{EQ:L:F-primal->dual}).
Since $-\hat y^p_\F=1^p$, again $\hat\mu(-\hat y_\F)=1$.
As before we obtain
$\ell^\ast(-\hat y_\F) = -\ell(\hat q)=\hat\ell_\F$ and $\hat y_\F\in\SOL_\F$.
The lemma is proved.
\end{proof}

\subsubsection{Structure of $\SOL_\F$}\label{SSS:F-sol}

\begin{lemma}\label{L:F-SOLa-is-almost-singleton}
 Let $\hat q_\PR^a$ be as in Theorem~\ref{T:P}. Then
 \begin{equation*}
    \SOL_\F\subseteq \left\{
      \hat y\in C^\ast: \hat y^a=\frac{\nu^a}{\hat q^a_\PR}-\hat\mu(-\hat y)
    \right\}.
 \end{equation*}
\end{lemma}
\begin{proof}
 Take any $\hat y\in\SOL_\F$.
 By Propositions~\ref{P:F-Scc-SP} and \ref{P:cc-SOL},
 \begin{equation*}
    \hat q^a_\PR
    = \hat q^a_\cc(-\hat y)
    = \frac{\nu^a}{\hat\mu(-\hat y) + \hat y^a}\,.
 \end{equation*}
From this the result follows immediately.
\end{proof}

\begin{lemma}\label{L:F-SOL}
Let $\hat y^a_\F$ be as in Lemma~\ref{L:F-primal->dual}.
Then
\begin{equation*}
    \SOL_\F
    = \{
      \hat y\in C^\ast: \  \hat y^a = \hat y^a_\F,\  \hat\mu(-\hat y)=1
    \}.
\end{equation*}
\end{lemma}
\begin{proof}
One inclusion follows from Lemmas~\ref{L:F-primal->dual} and \ref{L:cc-nondecreasing}(b).
To prove the other one, take any $\hat y\in \SOL_\F$ and put $\mu=\hat\mu(-\hat y)$.
Let $\hat y_\F$ be as in Lemma~\ref{L:F-primal->dual}.
Then $\hat y^a = \hat y^a_\F - (\mu-1)$ by Lemma~\ref{L:F-SOLa-is-almost-singleton}.
Theorem~\ref{T:cc} yields
\begin{equation*}
    0
    = \ell^\ast(\hat y) - \ell^\ast(\hat y_\F)
    =  (\mu + I_\mu(\nu^a\,\|\,\hat y^a))
     - (1 + I_1(\nu^a\,\|\,\hat y^a_\F))
    =\mu-1.
\end{equation*}
Thus $\mu=1$ and $\hat y^a = \hat y^a_\F$.
\end{proof}

\subsubsection{Proof of Theorem~\ref{T:F}}\label{SSS:F-proof}
Now we are ready to prove Theorem~\ref{T:F}.

\begin{proof}[Proof of Theorem~\ref{T:F}]
The facts that $\hat\ell_\F=-\hat\ell_\PR$ and $\SOL_\F\ne\emptyset$
were proved in Lemma~\ref{L:F-primal=fenchel}.
Further, the set $\SOL_\F$ is convex and closed
due to the fact that the conjugate $\ell^\ast$ is convex and closed
(even continuous, see Lemma~\ref{L:cc-lipschitzness}).
To show compactness of $\SOL_\F$ it suffices to prove that it is bounded.

We already know from Lemmas~\ref{L:F-primal->dual} and \ref{L:F-SOL}
that $\hat y^a > -1$ and $\hat y^p\ge -\hat\mu(-\hat y)=-1$
for every $\hat y\in\SOL_\F$; thus $\SOL_\F$ is bounded from below by $-1$.

Since $\supp C=\A$, there is $q\in C$ with $q>0$; put $\beta\triangleq\min q_i>0$.
Take any $\hat y\in\SOL_\F$ and put $I\triangleq\{i\in\A: \hat y_i>0\}$.
Then, since $\hat y\in C^\ast$ and $\hat y\ge -1$,
\begin{equation*}
  0
  \ge
  \sp{\hat y}{q}
  \ge
  \beta \sum_{i\in I} \hat y_i - \sum_{j\in\A\setminus I} q_j
  \ge \beta \hat y_i - 1
\end{equation*}
for every $i\in I$. Thus $\hat y\le (1/\beta)$ and so $\SOL_\F$ is bounded, hence compact.

The rest of Theorem~\ref{T:F} follows from Lemmas~\ref{L:F-primal->dual} and \ref{L:F-SOL}.
\end{proof}

\subsection{Proof of Theorem~\ref{T:BD}
  (Relation between \refBD{} and \refPR{})}\label{S:proofs-BD}

First, basic properties of $\ell^\ast_\BD$ are proven.
Then Lemma~\ref{L:dual-finite-C*} provides a preparation for the recession cone
considerations of \refBD{}. This leads to Proposition~\ref{P:dual-finite}
giving the conditions of finiteness of $\hat\ell_\BD$.
There \refBD{} is seen as a primal problem and Theorem~\ref{T:bertsekas2} is applied to it.
The solution set $\SOL_\BD$ is described in Lemma~\ref{L:SOL-BD}.
Lemma~\ref{L:BD-duality:dual-to-primal} provides properties of
$\hat q_\BD\in C$ defined via $\hat y_\BD$,
noting that $\hat q_\BD$ need not belong to $\SOL_\PR$.
Conditions equivalent to $\hat q_\BD \in \SOL_\PR$ are stated in
Lemma~\ref{L:BD-duality:dual-to-primal3}. Its proof utilizes also
Lemmas~\ref{L:BD-duality:dual-to-primal2}~and~\ref{L:BD-duality:dual-to-primal2b}.

Let $\nu\in \sx{\A}$ be a type. Recall from Section~\ref{St:BD} that the
map $\bd{}:\RRR^m\to\RRRex$ is for $y^a>-1$ defined by
\begin{equation*}
  \bd{}(y) = I_1(\nu^a\,\|\,y^a) = \sp{\nu^a}{\log\frac{\nu^a}{1+y^a}}\,.
\end{equation*}
Since
\begin{equation*}
    \bd{}(y)=\ell(1+y)  + \sp{\nu^a}{\log \nu^a}
\end{equation*}
for every $y\in\RRR^m$ (where $\ell$ is defined on $\RRR^m$ by (\ref{EQ:f-def})),
Lemma~\ref{L:f-props} yields the following result.

\begin{lemma}\label{L:g-props}
If $\nu\in\sx{\A}$, then
\begin{enumerate}
  \item[(a)] $\bd{}(y)>-\infty$ for every $y\in \RRR^m$, and
             $\bd{}(y)<\infty$ if and only if $y^a>-1$; so
      \begin{equation*}
        \dom(\bd{})= \{
          y\in \RRR^m:\ y^a>-1
        \};
      \end{equation*}
  \item[(b)] $\bd{}$ is a proper continuous (hence closed) convex function;
  \item[(c)] the restriction of $\bd{}$ to its domain ${\dom(\bd{})}$ is strictly convex
    if and only if $\nu>0$;
  \item[(d)] $\bd{}$ is differentiable on $\dom(\bd{})$ with the gradient given by
    \begin{equation*}
      \nabla \bd{}(y) =  -(\nu^a/(1+y^a), 0^p);
    \end{equation*}
  \item[(e)] the recession cone $R_\bd{}$ and the constancy space $L_\bd{}$ of $\bd{}$ are
      \begin{equation*}
        R_{\bd{}} = \{
          z\in\RRR^m:\ z^a\ge 0
        \},
        \qquad
        L_\bd{} = \{
          z\in\RRR^m:\ z^a= 0
        \}.
      \end{equation*}
\end{enumerate}
\end{lemma}

\subsubsection{Finiteness of the BD-dual \refBD{}}

\begin{lemma}\label{L:dual-finite-C*}
Let $\nu\in\sx{\A}$ and $C$ be a convex closed subset
of $\sx{\A}$ having support $\A$.
Assume that $C$ is either an $H$-set or a $Z$-set.
Then there is an active letter $i$ with the following property:
For any $\gamma>0$ and $v>0$
\begin{equation}\label{EQ:dual-finite-C*-y}
    \text{there is}\quad
    y\in C^\ast  \quad\text{such that}\quad
    y^a\ge -\gamma
    \quad\text{and}\quad
    y^a_i=v.
\end{equation}
\end{lemma}
\begin{proof}
Assume first that $C$ is a $Z$-set, that is, $C^a(0^p)\ne\emptyset$
and there is an active letter $i$ such that $q^a_i=0$ whenever $q\in C$ satisfies $q^p=0$.
Fix any $\gamma>0,v>0$, and choose $\eps\in (0,1)$ such that
$\eps<\gamma/(v+2\gamma)$.
By compactness of $C$
we can find $0<\delta<\eps$ such that, for every $q\in C$,
$\sum{q^p}\le\delta$ implies $q^a_i<\eps$.
(For if not, there is a sequence $(q^{(n)})_n$ in $C$ with $\sum{(q^{(n)})^p}\le 1/n$ and
$(q^{(n)})^a_i\ge\eps$ for every $n$; by compactness,
there is a limit point $q\in C$ of this sequence
and any such $q$ satisfies $\sum{q^p}=0$ and $q^a_i\ge \eps$, a contradiction.)
Finally, take any $w\ge v/\delta$.

Define $y\in\RRR^m$ by
\begin{equation}\label{EQ:dual-finite-C*-ydef}
    y_i^a=v,\quad
    y_j^a=-\gamma \text{ for } j\in\Aa\setminus\{i\},
    \quad\text{and}\quad
    y_j^p=-w \text{ for } j\in\Ap.
\end{equation}
Take arbitrary $q\in C$; we are going to show
that $\sp{y}{q}\le 0$. If $\sum{q^p}\le\delta$ then $q_i^a\le \eps$ and
$\sum{q^a}\ge 1-\delta > 1-\eps$, so
\begin{equation*}
    \sp{y}{q}
    = (v+\gamma) q^a_i - \gamma \sum{q^a} -w \sum{q^p}
    < (v+\gamma)\eps - \gamma (1-\eps) < 0
\end{equation*}
by the choice of $\eps$. On the other hand, if $\sum{q^p}>\delta$ then,
using $q_i^a\le \min\{1,\sum{q^a}\}$,
\begin{equation*}
  \sp{y}{q}
  = v q^a_i - \gamma \left(\sum{q^a}-q^a_i \right) - w\sum{q^p}
  < v  - w\delta \le 0
\end{equation*}
by the choice of $w$. Thus $y\in C^\ast$.

Now assume that $C$ is an $H$-set; that is, $C^a(0^p)=\emptyset$.
By compactness, there exists $\delta>0$ such that $\sum{q^p}>\delta$
for every $q\in C$. Continuing as above we obtain
that, for any $\gamma>0,v>0$, and $w\ge v/\delta$, the vector $y$ given by (\ref{EQ:dual-finite-C*-ydef}) belongs to~$C^\ast$.
\end{proof}

\begin{proposition}\label{P:dual-finite}
Let $\nu\in\sx{\A}$ and $C$ be a convex closed subset of $\sx{\A}$ having support~$\A$.
Then the following are equivalent:
\begin{enumerate}
  \item[(a)] the dual \refBD{} is finite;
  \item[(b)] the set $C$ is neither an $H$-set nor a $Z$-set (with respect to $\nu$).
\end{enumerate}
\end{proposition}

\begin{proof}
Assume that $C$ is either an $H$-set or a $Z$-set.
Fix any $\gamma\in(0,1)$. Then, by Lemma~\ref{L:dual-finite-C*}, for arbitrary $v>0$
there is $y\in C^\ast$ satisfying (\ref{EQ:dual-finite-C*-y}).
For such $y$,
\begin{equation*}
\begin{split}
    -\bd{}(y) + \sp{\nu^a}{\log\nu^a}
    &\ge \sum_{j\in\Aa\setminus\{i\}} \nu_j^a \log(1-\gamma)   +  \nu_i^a \log(1+ v)
\\
    &= (1-\nu_i^a)\log(1-\gamma)   +  \nu_i^a \log(1+ v).
\end{split}
\end{equation*}
Since $\gamma$ is fixed and $v>0$ is arbitrary, $\hat\ell_\BD=-\infty$.
This proves that (a) implies (b).

Now we show that (b) implies (a). Assume that $C$ is neither an $H$-set nor a $Z$-set;
that is, there is $q\in C$ with $q^a>0$, $q^p=0$. Put $D=\{q\}^\ast$.
Since $R_D=D=\{y:\sp{y^a}{q^a}\le 0\}$,
$L_D=D\cap (-D)=\{y: \sp{y^a}{q^a}=0\}$
and, by Lemma~\ref{L:g-props}, $R_\bd{}=\{y:y^a\ge 0\}$ and $L_\bd{}=\{y: {y^a}=0\}$,
it follows that
\begin{equation*}
    R_\bd{}\cap R_D = L_\bd{}\cap L_D = \{y: y^a=0\}.
\end{equation*}
Now Theorem~\ref{T:bertsekas2} (applied to $\bd{}$ and $D$)
implies that $\argmin_{D} \bd{}$ is nonempty,
hence $\bd{}$ is bounded from below on $D$.
Since $C^\ast\subseteq D$,  (a) is established.
\end{proof}

\subsubsection{The solution set $\SOL_\BD$ of the BD-dual}

\begin{lemma}\label{L:SOL-BD}
Let $\nu\in\sx{\A}$.
Let $C$ be a convex closed subset of $\sx{\A}$
having support~$\A$, which is neither an
$H$-set nor a $Z$-set (with respect to $\nu$).
Then the solution set $\SOL_\BD$ is nonempty, and there is
$\hat y_\BD^a\in C^{\ast a}$ such that $\hat y_\BD^a>-1$ and
\begin{equation*}
  \SOL_\BD
  = \{\hat y\in C^\ast: \hat y^a = \hat y_\BD^a\}
  = \{\hat y_\BD^a\} \times C^{\ast p}(\hat y_\BD^a).
\end{equation*}
Moreover, $\SOL_\BD$ is a singleton if and only if $\nu>0$.
\end{lemma}
\begin{proof}
The proof is analogous to that of Proposition~\ref{P:primal-active-projection}.
Define $\bda:C^{\ast a}\to\RRRex$ by $\bda(y^a)\triangleq \sp{\nu^a}{\log(\nu^a/(1+y^a))}$
if $y^a>-1$, $\bda(y^a)\triangleq \infty$ otherwise.
Then $\bd(y)=\bda(y^a)$ for any $y\in C^\ast$, hence
\begin{equation}\label{EQ:L:SOL-BD}
    \hat\ell_\BD=\inf_{y^a\in C^{\ast a}} \bda(y^a)
    \quad\text{and}\quad
    \SOL_\BD = \left\{y\in C^\ast: y^a\in\argmin_{C^{\ast a}} \bda{}\right\}.
\end{equation}
By Lemma~\ref{L:g-props}, $\bda$ is strictly convex, so $\argmin_{C^{\ast a}} \bda{}$
contains at most one point.
The set $C^\ast$ is neither an $H$-set nor a $Z$-set,
so there is $q\in C$ with $q^a>0$ and $q^p=0$.
Thus, as in the proof of Proposition~\ref{P:dual-finite},
$R_\bda{}\cap R_{C^{\ast a}}\subseteq R_\bda{}\cap R_{\{q^a\}^\ast}=\{0^a\}$.
Since $C^{\ast a}\cap \dom(\bda)\ne\emptyset$, Theorem~\ref{T:bertsekas} gives that
$\argmin_{C^{\ast a}} \bda{}$ is nonempty, hence a singleton;
denote its unique point by  $\hat y_\BD^a$.
Since trivially $\hat y_\BD^a>-1$,
the first assertion of the lemma follows from (\ref{EQ:L:SOL-BD}).
The second assertion follows from the first one and Lemma~\ref{L:C*}.
\end{proof}

\subsubsection{No BD-duality gap; proof of Theorem~\ref{T:BD}}

The proof of the following lemma is inspired by that of \cite[Thm.~2.1]{elbarmi-dykstra}.

\begin{lemma}[Relation between $\hat y_\BD$ and $\hat q_\BD$]\label{L:BD-duality:dual-to-primal}
Let $\nu\in\sx{\A}$.
Let $C$ be a convex closed subset of $\sx{\A}$ having support~$\A$,
which is neither an $H$-set nor a $Z$-set (with respect to $\nu$).
Let $\hat y^a_\BD$ be as in Lemma~\ref{L:SOL-BD} and put
\begin{equation}\label{EQ:BD-duality-proof-qbar}
    \hat{q}_\BD \triangleq  \left(
      \frac{\nu^a}{1+\hat y^a_\BD}, 0^p
    \right).
\end{equation}
Then
\begin{equation*}
  \hatmu(-\hat y^a_\BD)=1,
  \qquad
  \hat{q}_\BD\in C,
  \qquad
  \hat{q}_\BD\bot\SOL_\BD,
  \qquad\text{and}\qquad
  \ell(\hat q_\BD)=-\hat\ell_\BD.
\end{equation*}
\end{lemma}

Note that $\hat q_\BD$ need not belong to $\SOL_\PR$; for
conditions equivalent to $\hat q_\BD\in\SOL_\PR$, see Lemma~\ref{L:BD-duality:dual-to-primal3}.

\begin{proof}
Take any $\hat{y}=(\hat y^a_\BD,\hat y^p)\in\SOL_\BD$.
Then
\begin{equation}\label{EQ:BD-duality-proof-1}
    0
    \le (\bd{})'(\hat y; y-\hat y)
    =-\sum \frac{\nu^a (y^a-\hat{y}^a_\BD{})}{1+\hat y^a_\BD{}}
    \qquad\text{for every } y\in C^\ast.
\end{equation}
Applying (\ref{EQ:BD-duality-proof-1}) to $y=2\hat y$ and then to $y=(1/2)\hat y$
(both belonging to $C^\ast$ since $C^\ast$ is a cone) gives
\begin{equation}\label{EQ:BD-duality-proof-2}
    \sp{\hat q_\BD}{\hat y}
    = \sum \frac{\nu^a \hat{y}^a_\BD{}}{1+\hat y^a_\BD{}}
    =0.
\end{equation}
Now (\ref{EQ:BD-duality-proof-1}) and (\ref{EQ:BD-duality-proof-2}) yield
$\sp{\hat q_\BD}{y}\le 0$
for every $y\in C^\ast$, that is,
$\hat{q}_\BD\in C^{\ast\ast}=K_C$ (for the last equality use that $C$
is convex closed).
Further (recall the definition (\ref{EQ:xi-def}) of $\xi$),
\begin{equation*}
    1
    = \sum  \frac{\nu^a(1+\hat y^a_\BD{})}{1+\hat y^a_\BD{}}
    = \sum \frac{\nu^a}{1+\hat y^a_\BD{}}
    = \xi_{-\hat y^a_\BD}(1)
\end{equation*}
by (\ref{EQ:BD-duality-proof-2}) and (\ref{EQ:BD-duality-proof-qbar}).
Since $\max(-\hat y^a_\BD)<1$, $\hatmu(-\hat y^a_\BD)=1$.
Moreover, $\hat q_\BD\in \sx{\A}$ (use that $\hat q_\BD{}\ge 0$),
and so $\hat q_\BD\in(K_C\cap\sx{\A})=C$.
Finally,
\begin{equation*}
    \ell(\hat q_\BD)
    =  - \sp{\nu^a}{\log \hat q^a_\BD}
    =  -I_1(\nu^a\,\|\,\hat y^a_\BD{})
    = -\bd{}(\hat y)
    = -\hat\ell_\BD.
\end{equation*}
The lemma is proved.
\end{proof}

\begin{lemma}\label{L:BD-duality:dual-to-primal2}
Let $\nu\in\sx{\A}$ and $C$ be a convex closed subset of $\sx{\A}$ having support~$\A$. Then
\begin{equation*}
    \hat\ell_\BD\le\hat\ell_\F.
\end{equation*}
Moreover, if $\nu>0$ then $\hat\ell_\BD=\hat\ell_\F$.
\end{lemma}
\begin{proof}
It follows from Theorems~\ref{T:F} and \ref{T:cc} that
\begin{equation*}
  \hat\ell_\F
  = \inf_{\substack{y\in C^\ast\\ \hat\mu(-y)=1}} \ell^\ast(-y)
  = \inf_{\substack{y\in C^\ast\\ \hat\mu(-y)=1}} I_1(\nu^a\,\|\,y^a)
  \ge \inf_{y\in C^\ast} I_1(\nu^a\,\|\,y^a)
  =\hat\ell_\BD.
\end{equation*}
Hence the inequality is proved. If $\nu>0$ then $\hat\mu(-y)=\hatmu(-y^a)$ for every $y$;
so, by Lemma~\ref{L:BD-duality:dual-to-primal},
\begin{equation*}
  \hat\ell_\BD
  = \inf_{\substack{y\in C^\ast\\ \hatmu(-y^a)=1}} I_1(\nu^a\,\|\,y^a)
  = \hat\ell_\F.
\end{equation*}
\end{proof}

\begin{lemma}\label{L:BD-duality:dual-to-primal2b}
Let $\nu\in\sx{\A}$, $C$ be a convex closed subset of $\sx{\A}$ having support~$\A$,
and $\hat q^a_\PR$ be as in Theorem~\ref{T:P}.
Assume that $C$ is neither an $H$-set nor a $Z$-set (with respect to $\nu$),
and that $\sum\hat q^a_\PR=1$.
Then $\hat\ell_\BD=\hat\ell_\F$ and $\hat y\triangleq (\nu^a/\hat q^a_\PR-1^a,-1^p)$ belongs to~$\SOL_\BD\cap \SOL_\F$.
\end{lemma}
\begin{proof}
By Theorems~\ref{T:F}(a) and \ref{T:cc},
$\hat y\in\SOL_\F$ and $\hat\ell_\F=\ell^\ast(-\hat y)=I_1(\nu^a\,\|\,\hat y) = \bd(\hat y)$. Moreover,
$\xi_{-\hat y^a}(1)=\sum \hat q^a_\PR = 1$,
so $\hatmu(-\hat y^a)=1$. Thus
\begin{equation}\label{EQ:L:BD-duality:dual-to-primal2b}
    \bd(\hat y) = \hat\ell_\F
    \qquad\text{and}\qquad
    \hat\mu(-\hat y) = \hatmu(-\hat y^a) = 1.
\end{equation}

Now we prove that $\hat\ell_\BD=\hat\ell_\F$. For $\nu>0$ this follows
from Lemma~\ref{L:BD-duality:dual-to-primal2}. In the other case put
$C'\triangleq C^a(0^p)=C^a\cap\sx{\Aa}$. This is a nonempty
convex closed set with $\supp(C')=\A^a$ (use that $C$ is neither an $H$-set nor
a $Z$-set).
Put $\nu'\triangleq \nu^a>0$ and
\begin{equation*}
    \hat\ell'_\BD \triangleq \inf_{y'\in C'^\ast} I_1(\nu'\,\|\,y'),
    \qquad
    \hat\ell'_\F \triangleq \inf_{y'\in C'^\ast} \ell'^\ast(-y'),
    \qquad
    \hat\ell'_\PR \triangleq \inf_{q'\in C'} \ell'(q'),
\end{equation*}
where $\ell'(q')\triangleq -\sp{\nu'}{\log q'}$ is defined on $\sx{\Aa}$.
Since $\hat q_\PR^a\in C'$ we trivially have $\hat\ell'_\PR=\hat\ell_\PR$.
Further, $C'^\ast\supseteq C^{\ast a}$. (To see this, take any $y^a\in C^{\ast a}$; then there
is $y^p$ such that $\sp{y^a}{q^a}+\sp{y^p}{q^p}\le 0$ for every $q\in C$.
Since $(q',0^p)\in C$ for any $q'\in C'$, for every such $q'$ we have
$\sp{y^a}{q'}\le 0$; that is, $y^a\in C'^\ast$.)
Hence $\hat\ell'_\BD\le \inf_{y^a\in C^{\ast a}} I_1(\nu^a\,\|\,y^a)=\hat\ell_\BD$
(for the  equality see the definition (\ref{EQ:intro-BD}) of $\hat\ell_\BD$).
By Lemma~\ref{L:BD-duality:dual-to-primal2} and Theorem~\ref{T:F},  $\hat\ell'_\BD=\hat\ell'_\F=- \hat\ell'_\PR$, thus,
\begin{equation*}
  \hat\ell_\BD
  \ge \hat\ell'_\BD
  = - \hat\ell'_\PR=-\hat\ell_\PR = \hat\ell_\F.
\end{equation*}
Now $\hat\ell_\BD=\hat\ell_\F$ by the inequality from Lemma~\ref{L:BD-duality:dual-to-primal2}.
To finish the proof it suffices to realize that this fact
and (\ref{EQ:L:BD-duality:dual-to-primal2b}) yield $\bd(\hat y)=\hat\ell_\BD$,
and so $\hat y\in \SOL_\BD$.
\end{proof}

\begin{lemma}\label{L:BD-duality:dual-to-primal3}
Keep the assumptions and notation from Lemma~\ref{L:BD-duality:dual-to-primal}.
Then there is no BD-duality gap, that is,
\begin{equation*}
  \hat\ell_\BD = -\hat\ell_\PR,
\end{equation*}
if and only if any of the following (equivalent) conditions hold:

\begin{enumerate}
   \item[(a)] $\hat{q}^a_\BD=\hat{q}^a_\PR$;
   \item[(b)] $\hat q_\BD\in\SOL_\PR$;
   \item[(c)] $\SOL_\PR=\{\hat q_\BD\}$;
   \item[(d)] $\sum \hat q^a_\PR=1$;
   \item[(e)] $\hat{y}^a_\BD=\hat{y}^a_\F$;
   \item[(f)] $\SOL_\F\cap \SOL_\BD\ne\emptyset$;
   \item[(g)] $\SOL_\F = \SOL_\BD \cap\{\hat y\in C^\ast: \hat\mu(-\hat y)=1\}$;
   \item[(h)] $\hat\mu(-\hat y_\BD)=1$ for some $\hat y_\BD\in\SOL_\BD$;
   \item[(i)] $\hat y^p_\BD\ge -1$ for some $\hat y_\BD\in\SOL_\BD$;
   \item[(j)] $\ell(\hat q_{\BD}) + \ell^\ast(-\hat{y}_\BD)=0$  for some $\hat y_\BD\in\SOL_\BD$
      (\emph{extremality relation});
   \item[(k)] $(\hat{y}^a_\BD, -1^p)\in C^\ast$.
\end{enumerate}
\end{lemma}

\begin{proof}
We first prove that the conditions (a)--(k) are equivalent. First, Theorem~\ref{T:P}
and Lemma~\ref{L:BD-duality:dual-to-primal}
yield that the conditions (a)--(c) are equivalent and that any of them implies (d).
By the definitions of $\hat y^a_\F$ and $\hat q^a_\BD$
from Theorem~\ref{T:F} and Lemma~\ref{L:BD-duality:dual-to-primal},
(a) is equivalent to (e).
Theorem~\ref{T:F} and Lemma~\ref{L:SOL-BD}
yield that the conditions (e)--(g) are equivalent.
Since (d) implies (f) by Lemma~\ref{L:BD-duality:dual-to-primal2b},
it follows that (a)--(g) are equivalent.

Since $\hatmu(-\hat y^a_\BD)=1$ for any $\hat y_\BD\in\SOL_\BD$ by
Lemmas~\ref{L:SOL-BD} and \ref{L:BD-duality:dual-to-primal},
(h) is equivalent to (i).
Since $\SOL_\F$ is nonempty, (g) implies (h).
If (h) is true then, for some $\hat y^\BD{}\in\SOL_\BD{}$,
\begin{equation*}
  \hat\ell_\BD{}
  = \bd(\hat y_\BD{})
  = I_1(\nu^a\,\|\, \hat y^a_\BD{})
  = \ell^\ast(-\hat y_\BD{})
  \ge \hat\ell_\F
\end{equation*}
by Theorem~\ref{T:cc};
now Lemma~\ref{L:BD-duality:dual-to-primal2} gives that $\hat y_\BD\in\SOL_\F$.
Hence (h) implies (f). Further, (b), (e), and (h) imply (j)
by Theorem~\ref{T:cc} and Lemma~\ref{L:BD-duality:dual-to-primal}.
On the other hand, the extremality relation (j) implies that $\hat q_\BD\in\SOL_\PR$ and
that $\hat y_\BD\in\SOL_\F$; that is, (j) implies both (b) and (f).
So (a)--(j) are equivalent

By Theorem~\ref{T:F}, (e) implies (k).
By Lemmas~\ref{L:BD-duality:dual-to-primal} and \ref{L:SOL-BD},
(k) implies that $\hat y\triangleq (\hat y^a_\BD,-1^p)$ belongs to $\SOL_\BD$ and
that $\hat\mu(-\hat y)=1$;
thus, (k) implies (h).

We have proved that the conditions (a)--(k) are equivalent.
To finish we show that no BD-duality gap is equivalent to the condition (b).
If (b) is true (that is, $\hat q_\BD\in\SOL_\PR$)
then $\hat\ell_\PR=\ell(\hat q_\BD)=-\hat\ell_\BD$ by
Lemma~\ref{L:BD-duality:dual-to-primal}; hence (b) implies no BD-duality gap.
Assume now that $\hat\ell_\BD=-\hat\ell_\PR$. Then
Lemma~\ref{L:BD-duality:dual-to-primal} yields $\ell(\hat q_\BD)=-\hat\ell_\BD=\hat\ell_\PR$.
So (b) is satisfied and the proof is finished.
\end{proof}

\begin{proof}[Proof of Theorem~\ref{T:BD}]
The theorem immediately follows from
Proposition~\ref{P:dual-finite} and
Lemmas~\ref{L:SOL-BD}, \ref{L:BD-duality:dual-to-primal},
\ref{L:BD-duality:dual-to-primal2} and \ref{L:BD-duality:dual-to-primal3}.
\end{proof}

\subsection{Proof of Theorem~\ref{T:AP} (Active-passive dualization \refAP{})}\label{SS:AP}

\begin{proof}[Proof of Theorem~\ref{T:AP}]
Fix $q^p\in C^p$ with $\supp(C^a(q^p))=\Aa$. Then $\sum q^p<1$.
Recall that $\kappa = \kappa(q^p)=1/(1-\sum q^p)$
and note that $C^a(q^p)$ is a convex closed subset of
$\{q^a\in\RRR^{m_a}_+: \sum q^a=1/\kappa\}$.
Thus
\begin{equation*}
    \tilde C \triangleq
    \kappa \cdot C^a(q^p)
    = \{\kappa q^a: q^a\in C^a(q^p)\}
\end{equation*}
is a convex closed subset of $\sx{\Aa}$ with $\supp(\tilde C)=\Aa$.
Put $\tilde\nu\triangleq \nu^a$, and denote by \refPRtilde{} and $\SOL_{\tilde\PR}$
the primal problem and its solution set for
minimizing $\tilde\ell(\tilde q)\triangleq -\sp{\tilde\nu}{\log \tilde q}
=\ell(\tilde q/\kappa, q^p)-\log\kappa$
over $\tilde C$.
Then $\SOL_{\tilde\PR}$ is a singleton $\{\hat q_{\tilde\PR}\}$ by Theorem~\ref{T:P}, and
$\SOL_{\tilde\PR}=\kappa\SOL_\PR(q^p)$; hence
\begin{equation}\label{EQ:AP-qhat}
    \hat q^a_\PR(q^p) = \kappa^{-1} \hat q_{\tilde\PR}.
\end{equation}

Consider now the BD-dual \refBDtilde{} to \refPRtilde{}.
First, $\bdtil(\tilde y) = I_1(\tilde\nu \,\|\, \tilde y)$ and,
by Corollary~\ref{C:P+F+B:nu-positive},

\begin{equation}\label{EQ:AP-yhat-tilde}
  \SOL_{\tilde\BD} = \{\hat y_{\tilde\BD} \},
  \qquad
  \hat y_{\tilde\BD} = \frac{\tilde\nu}{\hat q_{\tilde\BD}}  - 1.
\end{equation}
Further, $\tilde C = [C^a(q^p)]^\ast$ by the definition of the polar cone.
The fact that $I_\kappa(\nu^a \,\|\, y^a) = I_1(\nu^a \,\|\, \kappa^{-1}y^a) - \sp{\nu^a}{\log\kappa}
=\bdtil(\kappa^{-1}y^a) - \log\kappa$ now implies
that $\SOL_{\BD}(q^p) = \kappa\SOL_{\tilde\BD}$.
That is,
\begin{equation}\label{EQ:AP-yhat}
    \SOL_\BD(q^p)=\{\hat y^a_\BD\},
    \qquad\text{where}\quad
    \hat y^a_{\BD}(q^p) = \kappa\hat y_{\tilde\BD}.
\end{equation}
Finally,
(\ref{EQ:AP-qhat}), (\ref{EQ:AP-yhat-tilde}), and (\ref{EQ:AP-yhat}) yield (\ref{EQ:BD-duality-active}).
The rest follows from
(\ref{EQ:BD-duality-active}) and Corollary~\ref{C:P+F+B:nu-positive}.
\end{proof}

\subsection{Proof of Theorem~\ref{T:PP}
 (Perturbed primal \refPP{} --- the general case, epi-convergence)}\label{SS:PP-epi}

We start with some notation, which will be used also in Section~\ref{SS:PP-linear}.
Then we embark on proving the epi-convergence of perturbed primal problems.

\subsubsection{Perturbed primal \refPP{} --- notation}\label{SS:PP-notation}

Fix any $\nu\in\sx{\A}$.
Recall that $m_a,m_p$ denote the cardinalities of $\Aa,\Ap$.
For every $\delta>0$ take $\nu(\delta)\in\sx{\A}$ such that
(\ref{EQ:PP-nu(delta)}) is true; that is,
\begin{equation*}
    \nu(\delta)>0
    \qquad\text{and}\qquad
    \lim_{\delta\searrow 0} \nu(\delta) =\nu.
\end{equation*}
Recall also the definitions of $\ell_\delta, \hat\ell_\PR(\delta), \SOL_\PR(\delta)$
from (\refPP{}) and the fact that (cf.~(\ref{EQ:PP-solutions}))
\begin{equation*}
    \SOL_\PR(\delta) = \{ \hat q_\PR(\delta) \}
    \qquad\text{for }\delta>0.
\end{equation*}
The maps $\ell^C,\ell^C_\delta:\RRR^m\to\RRRex$ are defined by
(see Section~\ref{St:PP})
\begin{equation*}
    \ell^C(x) \triangleq \begin{cases}
      \ell(x)  &\text{if }x\in C,
    \\
      \infty  &\text{if }x\not\in C,
    \end{cases}
  \qquad \text{and}\qquad
    \ell^C_\delta(x) \triangleq \begin{cases}
      \ell_\delta(x)  &\text{if }x\in C,
    \\
      \infty  &\text{if }x\not\in C.
    \end{cases}
\end{equation*}

\subsubsection{Epi--convergence}\label{SS:PP-epi-conv}
For the definition of epi-convergence,
see \cite[Chap.~7]{rockafellar2009variational}. We will use
\cite[Ex.~7.3, p.~242]{rockafellar2009variational} stating that,
for every sequence $(g_n)_n$ of maps $g_n:\RRR^m\to\RRRex$ and for every $x\in\RRR^m$,
\begin{equation}\label{EQ:PP-eliminf}
\begin{split}
    (\eliminf_n g_n)(x) &= \lim_{\eps\to 0+} \liminf_{n\to\infty}  \inf_{x'\in B(x,\eps)} g_n(x'),
\\
    (\elimsup_n g_n)(x) &= \lim_{\eps\to 0+} \limsup_{n\to\infty}  \inf_{x'\in B(x,\eps)} g_n(x'),
\end{split}
\end{equation}
where $B(x,\eps)$ is the closed ball with the center $x$ and radius $\eps$.
If $\eliminf_n g_n=\elimsup_n g_n\triangleq g$, we say that the sequence $(g_n)_n$ \emph{epi-converges}
to $g$ and we write $\elim_n g_n=g$.
For a system $(g_\delta)_{\delta>0}$ of maps indexed by real numbers
$\delta>0$, we say that $(g_\delta)_\delta$ \emph{epi-converges}
to $g$ for $\delta\searrow 0$ and we write $\elim_\delta g_\delta=g$, provided
$\elim_n g_{\delta_n}=g$ for every sequence $(\delta_n)_n$ decreasing to zero.

Before proving the epi-convergence of $\ell_{\delta}^C$ to $\ell^C$, two simple lemmas are presented.

\begin{lemma}\label{L:PP-qhat-nu-ineq}
 Let $C$ be a convex closed subset of $\sx{\A}$ with $\supp(C)=\A$.
 Then there exists a positive real $\beta$ such that
\begin{equation*}
    \hat q_\PR(\nu) \ge \beta\nu
    \qquad\text{for every}\quad
    \nu\in\sx{\A}, \nu>0,
\end{equation*}
where, for $\nu>0$, $\hat q_\PR(\nu)$ denotes the unique point of $\SOL_\PR(\nu)$.
\end{lemma}
\begin{proof}
Fix some $q>0$ from $C$ and put $\beta=\min_i q_i$.
For any $\nu>0$ and $\hat{q}=\hat{q}_\PR(\nu)$ we have $\hat{q}>0$;
so, by Lemma~\ref{L:f-props}(d), $\ell_\nu$ is differentiable at $\hat q$ and
the directional derivative
$\ell_\nu'(\hat q;q-\hat q)=-\sum_{i\in\A} [\nu_i (q_i-\hat q_i)/\hat q_i]$ is nonnegative.
Thus, by the choice of $\beta$,
\begin{equation*}
    0
    \ge
    \sum_{i\in\A} \frac{\nu_i q_i}{\hat q_i} - 1
    \ge
    \frac{\nu_j \beta}{\hat q_j} - 1
\end{equation*}
for every $j$.
From this the lemma follows.
\end{proof}

\begin{lemma}\label{L:PP-continuity-of-F(nu,q)}
Fix $\beta>0$ and take the map
\begin{equation*}
    F:D\to\RRRex,\qquad
    F(\nu,q) \triangleq \ell_\nu(q) = -\sp{\nu}{\log q}
\end{equation*}
defined on
\begin{equation*}
    D \triangleq \{(\nu,q)\in \sx{\A}\times C: q\ge\beta \nu\}.
\end{equation*}
Then $F$ is continuous on $D$.
\end{lemma}
\begin{proof}
Take any sequence $(\nu^n,q^n)_n$ from $D$ such that $\nu^n\to \nu$, $q^n\to q$.
Then $q\ge \beta\nu$.
The sum $\sum_{i:\nu_i>0} \nu^n_i \log q^n_i$ converges to $\sum_{i:\nu_i>0} \nu_i \log q_i$. Further,
for every $i$ such that $\nu_i=0$,
\begin{equation*}
    0
    \ge \nu^n_i\log q^n_i
    \ge \nu^n_i \log (\beta \nu^n_i)
\end{equation*}
(which is true also in the case when $\nu_i^n=0$)
and the right-hand side converges to zero.
Hence $\lim_n F(\nu^n,q^n)=F(\nu,q)$ and the lemma is proved.
\end{proof}

Note that $F$ need not be continuous on $\sx{\A}\times C$. For example,
take $\A=\{0,1\}$, $C=\{(0,1)\}$, $\nu^n=(1/n, 1-(1/n))$ converging to $\nu=(0,1)$,
and $q^n=q=(0,1)$. Then $(\nu^n,q^n)\to (\nu,q)$, but $F(\nu^n,q^n)=\infty$
does not converge to $F(\nu,q)=0$.

\begin{lemma}\label{L:PP-epi-conv}
Let $C$ be a convex closed subset of $\sx{\A}$ with $\supp(C)=\A$.
Then the maps $\ell_\delta^C$ epi-converge to $\ell^C$, that is,
\begin{equation*}
  \elim\limits_{\delta\searrow 0} \ell_{\delta}^C = \ell^C.
\end{equation*}
\end{lemma}
\begin{proof}
Take any sequence $(\delta_n)_n$ decreasing to $0$.
To prove that $\elim_{n} \ell_{\delta_n}^C = \ell^C$
we use (\ref{EQ:PP-eliminf}).
Take $x\in\RRR^m$, $\eps>0$, and put
\begin{equation*}
    \psi_n(x,\eps)\triangleq
    \inf_{y\in B(x,\eps)} \ell_{\delta_n}^C(y)
    =
    \inf_{y\in C\cap B(x,\eps)} \ell_{\delta_n}(y).
\end{equation*}
If $x\not\in C$ then, since $C$ is closed, $\psi_n(x,\eps)=\infty$ provided $\eps$ is small enough;
hence $\lim_{\eps} \lim_{n} \psi_n(x,\eps)=\infty=\ell^{C}(x)$.

Assume now that $x\in C$. Then, by the monotonicity of the logarithm,
\begin{equation*}
    \psi_n(x,\eps)\in [\ell_{\delta_n}(x+\eps), \ell_{\delta_n}(x)].
\end{equation*}
If $x>0$ and $0<\eps<\min_i x_i$, then $\lim_n \ell_{\delta_n}(x)=\ell^C(x)$ and
$\lim_n \ell_{\delta_n}(x+\eps)=\ell^C(x+\eps)$, so continuity of $\ell$ at $x$ easily gives
that $\lim_{\eps} \lim_{n} \psi_n(x,\eps)=\ell^C(x)$.
If there is $i\in\Aa$ with $x_i^a=0$, then $\ell^C(x)=\infty$
and $\lim_\eps\lim_n (-\nu(\delta_n)_i^a \log(x_i^a+\eps))
= -\nu_i^a \,\lim_\eps  \log \eps=\infty$, so
again $\lim_{\eps} \lim_{n} \psi_n(x,\eps)=\ell^C(x)$.

Finally, assume that $x\in C$ is such that $x^a>0$ and there is $i\in\Ap$ with $x_i^p=0$.
By Theorem~\ref{T:P} (applied to $\nu(\delta_n)>0$ and to the
nonempty convex closed subset $C\cap B(x,\eps)$ of $\sx{\A}$),
for every $n$ there is unique $\hat x_{n,\eps}>0$ from $C\cap B(x,\eps)$ such that
\begin{equation*}
    \psi_n(x,\eps)=\ell_{\delta_n}(\hat x_{n,\eps}).
\end{equation*}
Moreover, by Lemma~\ref{L:PP-qhat-nu-ineq} there is $\beta>0$ such that
\begin{equation}\label{EQ:PP-epi-beta}
    \hat x_{n,\eps} \ge\beta\nu(\delta_n)
    \qquad\text{for every } n.
\end{equation}
Take any convergent subsequence of $(\hat x_{n,\eps})_n$, denoted again by
$(\hat x_{n,\eps})_n$, and let $\bar x_\eps\in C\cap B(x,\eps)$ denote its limit.
Then, by (\ref{EQ:PP-epi-beta}) and Lemma~\ref{L:PP-continuity-of-F(nu,q)},
$\ell_{\delta_n}(\hat x_{n,\eps})$ converges to $\ell(\bar x_\eps)$.
Hence all cluster points of $(\psi_n(x,\eps))_n$ belong to $\ell(C\cap B(x,\eps))$.
Since $\ell$ is continuous, $\lim_\eps \ell(C\cap B(x,\eps))=\{\ell(x)\}$.
This implies that $\elim_n \ell_{\delta_n}^C(x)=\ell^C(x)$ and the lemma is proved.
\end{proof}

\begin{lemma}\label{L:PP-epi-limsup-argmin}
Let $\supp(C)=\A$. Then
\begin{equation*}
 \lim_{\delta\searrow 0} \hat\ell_\PR(\delta)  = \hat\ell_\PR
 \qquad\text{and}\qquad
 \lim_{\delta\searrow 0} \inf_{\hat q_\PR\in\SOL_\PR} d(\hat q_\PR(\delta), \hat q_\PR) = 0.
\end{equation*}
\end{lemma}
\begin{proof}
Take any $(\delta_n)_n$ decreasing to $0$.
Since $(\ell_{\delta_n}^C)_n$
epi-converges to $\ell^C$ and $\hat\ell_\PR$ is finite, the result follows
from \cite[Thm.~7.1, p.~264]{rockafellar2009variational}
(in part (a) take $B\triangleq C$;
use that $\hat\ell_\PR(\delta_n)=\min \ell_{\delta_n}^C$ and that
$\hat\ell_\PR=\min \ell^C$).
\end{proof}

The lemma states that every cluster point of a sequence $(\hat q_\PR(\delta_n))_n$ of solutions
of the perturbed primal problems $\mathcal{P}_{\delta_n}$ ($\delta_n\searrow 0$)
is a solution of the (unperturbed) primal \refPR{}.
Of course, not every solution of \refPR{} can be obtained as a cluster point of a sequence of perturbed solutions.
For example, it is shown in the following section that if $(\nu(\delta))_\delta$
satisfies the regularity condition (\ref{EQ:PP-nu-bounded-differ}),
then the perturbed solutions converge to a unique solution of \refPR{}, regardless of
whether $\SOL_\PR$ is a singleton.

\begin{proof}[Proof of Theorem~\ref{T:PP}]
The first part of Theorem~\ref{T:PP} was shown in Lemmas~\ref{L:PP-epi-conv} and~\ref{L:PP-epi-limsup-argmin}.
The second part on the convergence of active coordinates
then follows from Theorem~\ref{T:P}.
\end{proof}

\subsection{Proof of Theorem~\ref{T:PP-linear}
(Perturbed primal \refPP{} ---
the linear case, pointwise convergence)}\label{SS:PP-linear}

Let a model $C$ be given by finitely many linear constraints
(\ref{EQ:C-linear}), that is,
\begin{equation*}
    C = \{q\in\sx{\A}: \ \sp{q}{u_h}=0 \text{ for } h=1,\dots,r\},
\end{equation*}
where $u_h$ ($h=1,\dots,r$) are fixed vectors from $\RRR^m$.
Let a type $\bar\nu\in\sx{\A}$ be given and let $\Aa$ and $\Ap$ be the
sets of active and passive coordinates with respect to $\bar\nu$.

Assume that perturbed types $\nu(\delta)$ ($\delta\in(0,1)$) are such that
the conditions (\ref{EQ:PP-nu(delta)}) and (\ref{EQ:PP-nu-bounded-differ})
are true (with $\nu$ replaced by $\bar\nu$); that is,
$\nu:(0,1)\to \sx{\A}$ is continuously differentiable,
\begin{equation*}
    \nu(\delta)>0,
    \qquad
    \lim_{\delta\searrow 0} \nu(\delta)=\bar \nu,
\end{equation*}
and there is a constant $c>0$
such that, for every $i\in\Ap$,
\begin{equation*}
  \abs{\w_i'(\delta)}\le c\,\w_i(\delta),
  \qquad\text{where}\quad
  \w_i(\delta) \triangleq \frac{\nu_i(\delta)}{\sum_{j\in\Ap} \nu_j(\delta)}\,.
\end{equation*}
The aim of this section is to prove that, under these conditions,
solutions $\hat q_\PR(\delta)$ of the perturbed primal problems \refPP{}
\emph{converge} to a solution of the unperturbed primal~\refPR{}; that is,
\begin{equation*}
    \lim_{\delta\searrow 0} \hat{q}_\PR(\delta)
    \quad\text{exists and belongs to }
    \SOL_\PR.
\end{equation*}

\subsubsection{Outline of the proof}
The proof
is based on the following `passive-active' reformulation of the perturbed primal problem
\begin{equation*}
    \min_{q\in C} \ell_\delta(q)
    =
    \min_{q^a\in C^a}  \min_{q^p\in C^p(q^a)}  \ell_\delta(q^a,q^p);
\end{equation*}
hence, the passive projection $\hat q^p_\PR(\delta)$ of the optimal solution $\hat q_\PR(\delta)$
is
\begin{equation*}
    \hat q^p_\PR(\delta)
    = \argmin_{q^p\in C^p(\hat q^a_\PR(\delta))} {\ell_{\delta}(\hat q^a_\PR(\delta), q^p)}
\end{equation*}
(recall that
$C^a=\pi^a(C)$ is the projection of $C$ onto the active coordinates and,
for $q^a\in C^a$, that $C^p(q^a)=\{q^p\ge 0: (q^a,q^p)\in C\}$ is the $q^a$-slice of~$C$).
Employing the implicit function theorem, it is then shown that
the passive projections $\hat q^p_\PR(\delta)$ can be obtained from the active
ones $\hat q^a_\PR(\delta)$ via a uniformly continuous map $\varphi$. Since
$\hat q^a_\PR(\delta)$ converges by Theorem~\ref{T:PP},
the uniform continuity of $\varphi$ ensures that also $\hat q^p_\PR(\delta)$ converges.

The above argument is implemented in the following steps:
\begin{enumerate}
  \item Assume that $\SOL_\PR$ is not a singleton, the other case being trivial.
    Then the linear constraints (\ref{EQ:C-linear}) defining $C$ can be rewritten into a parametric form
    \begin{equation*}
        q^p = A \lambda + B q^a + c  \qquad (q^a\in C^a,\lambda\in \Lambda(q^a))
    \end{equation*}
    (see Lemma~\ref{L:PP-lin-Cp(qa)}). There, $A$ is an ${m_p\times s}$ matrix of rank $s$,
    $B$ is an ${m_p\times m_a}$ matrix, $c\in\RRR^{m_p}$,
    none of $A,B,c$  depends on $q^a$,
    and the (polyhedral) closed subset $\Lambda(q^a)$ of $\RRR^{s}$ is defined by
    \begin{equation*}
        \Lambda(q^a) \triangleq \{\lambda\in\RRR^s:\ A \lambda + B q^a + c\ge 0\}.
    \end{equation*}

  \item Define an open bounded polyhedral set $G\subset \RRR^{1+m_a+s}$ by
    \begin{equation}\label{EQ:PP-lin-G-def}
      G
      \triangleq
      \{
       (\delta,x,y)\in (0,1)\times \RRR^{m_a}\times \RRR^s:
       x>0, Ay+Bx+c>0
      \};
    \end{equation}
    there  $x$ stands for $q^a$ and $y$ stands for $\lambda$, for brevity.
    Let $Z$ denote the projection of $G$ onto the first $(1+m_a)$ coordinates.
    For $z=(\delta,x)\in Z$ put
    \begin{equation}\label{EQ:PP-lin-Gz-def}
    \begin{split}
      G(z) &\triangleq \{ y\in\RRR^s: Ay+Bx+c > 0 \},
      \\
      \bar G(z) &\triangleq \{ y\in\RRR^s: Ay+Bx+c\ge 0 \}.
    \end{split}
    \end{equation}
    In Lemma~\ref{L:PP-lin-psiz-argmin}, it is proven that, for every $z=(\delta,x)\in Z$, the map
    \begin{equation*}
        \psi_z:\bar G(z)\to\RRRex,\qquad
        \psi_z(y)\triangleq -\sum_{i\in\Ap} \w_i(\delta) \log(\sp{\alpha_i}{y} + \sp{\beta_i}{x} + c_i)
    \end{equation*}
    (where $\alpha_i$ and $\beta_i$ denote the $i$-th rows of $A$ and $B$, respectively)
    has a unique minimizer $\varphi(z)\in G(z)$.

  \item Since $\varphi(z)$ is also a local minimum of $\psi_z$,
    it satisfies $F(z,\varphi(z))=0$, where $F:G\to \RRR^s$ is given by
    \begin{equation}\label{EQ:PP-lin-F-def}
      F(z,y)
      \triangleq
      -\nabla \psi_z(y).
    \end{equation}
    Using the implicit function theorem and an algebraic result (Proposition~\ref{P:PP-bounded-ADAinvAD}),
    it is shown that $\varphi:Z\to\RRR^s$ is uniformly continuous (see Lemma~\ref{L:PP-lin-phi-Lipschitz}).

  \item Finally, in Lemma~\ref{L:PP-lin-psiz-hatqa}, it is demonstrated that, for every $\delta>0$,
    \begin{equation*}
        (\delta,\hat q^a_\PR(\delta))\in Z
        \qquad\text{and}\qquad
        \hat q^p_\PR(\delta) = A\varphi(\delta,\hat q^a_\PR(\delta)) + B \hat q^a_\PR(\delta) + c.
    \end{equation*}
    This fact, together with the convergence of $\hat q^a_\PR(\delta)$ and the uniform continuity of $\varphi$,
    implies the convergence of $\hat q^p_\PR(\delta)$, and thus proves Theorem~\ref{T:PP-linear}.
\end{enumerate}

\subsubsection{Parametric expression for $q^p\in C^p(q^a)$}\label{SS:PP-lin-equations}
Let $C$ be given by (\ref{EQ:C-linear}).
Denote by $U^p$ the $(r+1)\times m_p$ matrix whose first $r$ rows are equal to $u_h^{p}$
(the passive projections of $u_h$, $h=1,\dots,r$),
and the last row is a vector of $1$'s; analogously define $U^a$ using the active projections $u_h^{a}$ of $u_h$.
Let $b\in\RRR^{r+1}$ be such that $b_i=0$ for $i\le r$ and $b_{r+1}=1$.
Then
\begin{equation}\label{EQ:C-linear-matrix}
  C = \{
    (q^a,q^p)\ge 0:\ U^p q^p = - U^a q^a + b
  \}.
\end{equation}
Let $V^p$ (of dimension $m_p\times(r+1)$) be the Moore-Penrose inverse of $U^p$.
By \cite[p.~5-12]{hogben2006handbook},
(\ref{EQ:C-linear-matrix}) can be written in the following form, where
$I$ denotes the $m_p\times m_p$ identity matrix:
\begin{equation*}
  C = \{
    (q^a,q^p)\ge 0:\ q^p = (I-V^p U^p) \gamma + B q^a + c
    \ \text{ for some }\  \gamma\in\RRR^{m_p}
  \},
\end{equation*}
where
\begin{equation*}
    B\triangleq-V^p U^a
    \qquad\text{and}\qquad
    c\triangleq V^p b.
\end{equation*}
Put $s\triangleq\rank(I-V^p U^p)$.

If $s=0$ then $(I-V^p U^p)=0$ and $q^p$ uniquely depends on $q^a$.
That is, for every $q^a\in C^a$, the set $C^p(q^a)$ is a singleton. By
Theorems~\ref{T:P} and \ref{T:PP},
also $\SOL_\PR$ is a singleton and $\hat q_\PR(\delta)$ converges
to the unique member of $\SOL_\PR$;
so in this case Theorem~\ref{T:PP-linear} is proved.

Hereafter, we assume that $s\ge 1$ and, without loss of generality,
that the first $s$ columns of $(I-V^p U^p)$ are linearly independent. Put
$$
  A\triangleq(I-V^p U^p)[\{1,\dots,m_p\}, \{1,\dots,s\}]
$$
(the submatrix of entries that lie in the first $m_p$ rows and the first $s$ columns).
Since $\{(I-V^p U^p) \gamma: \gamma\in\RRR^{m_p}\}$
equals $\{A\lambda: \lambda\in\RRR^s\}$,
the next lemma follows.

\begin{lemma}\label{L:PP-lin-Cp(qa)}
    Let $C$ be given by (\ref{EQ:C-linear})
    and $U^p,V^p$ be as above.
    Assume that $V^p U^p\ne I$.
    Then there are $s\ge 1$, an $m_p\times s$ matrix $A$ of rank $s$,
    an $m_p\times m_a$ matrix $B$, and a vector $c\in\RRR^{m_p}$ such that
    \begin{equation*}
      C = \{
        (q^a,L(q^a,\lambda)):\
        q^a\in C^a, \lambda\in\Lambda(q^a)
      \}.
    \end{equation*}
    There,
    \begin{equation*}
      L(q^a,\lambda) \triangleq A \lambda + B q^a + c
    \end{equation*}
    and
    \begin{equation*}
      \Lambda(q^a)\triangleq \{\lambda\in\RRR^s: L(q^a,\lambda)\ge 0\}
    \end{equation*}
    is a nonempty closed polyhedral subset of $\RRR^s$.
\end{lemma}

Note that, for $\lambda\ne\lambda'$, $L(q^a,\lambda)\ne L(q^a,\lambda')$
since $A$ has full column rank.

\medskip
Keep the notation from Lemma~\ref{L:PP-lin-Cp(qa)}.
Write $x$ for $q^a$, $y$ for $\lambda$, and define $G\subseteq\RRR^{1+m_a+s}$
by (\ref{EQ:PP-lin-G-def}).
Let $\pi_1:G\to (0,1)\times \RRR^{m_a}$ be the natural
projection mapping $(\delta,x,y)$ onto $(\delta,x)$.
Put
\begin{equation*}
    Z
    \triangleq \pi_1(G)
\end{equation*}
and, for $z=(\delta,x)\in Z$, define $G(z),\bar G(z)$ by (\ref{EQ:PP-lin-Gz-def}).
Since $G(z)$ is the $z$-slice of $G$,
\begin{equation*}
  G = \{(z,y): z\in Z, y\in G(z)\}.
\end{equation*}
Note also that, by Lemma~\ref{L:PP-lin-Cp(qa)},
\begin{equation}\label{EQ:PP-lin-Cp-and_Gz}
\begin{split}
  \Lambda(q^a) &= \bar G(\delta,q^a),
\\
  C^p(q^a) &= A \bar G(\delta,q^a) + B q^a + c,
\\
  \{q^p\in C^p(q^a): q^p>0\}
  &= A G(\delta,q^a) + B q^a + c,
\end{split}
\end{equation}
for every $q^a\in C^a$, $q^a>0$, and every $\delta\in(0,1)$.
Moreover,
\begin{equation}\label{EQ:PP-lin-G-and-C}
    \{q\in C: q>0\}
    =
    \{
      (x,Ay+Bx+c): (\delta,x,y)\in G \text{ for } \delta\in(0,1)
    \}.
\end{equation}

\begin{lemma}\label{L:PP-lin-G-Z-Gz}
The following are true:
  \begin{enumerate}
    \item[(a)] $G$ is a nonempty open bounded polyhedral set;
    \item[(b)] $Z$ is a nonempty open bounded set;
    \item[(c)] $G(z)$ ($\bar G(z)$) is a nonempty open (closed) bounded polyhedral set for every $z\in Z$.
  \end{enumerate}
\end{lemma}

\begin{proof}
(a) The fact that $G$ is open and polyhedral is immediate from (\ref{EQ:PP-lin-G-def}). It
is nonempty since $(\delta,q)\in G$ for every $\delta\in (0,1)$ and every $q\in C$ such that $q>0$
(such~$q$ exists since $C$ has support $\A$). To finish the proof of (a) it remains to show that
$G$ is bounded. Since $A$ has full column rank, (\ref{EQ:PP-lin-Cp-and_Gz})
and \cite[p.~5-12]{hogben2006handbook} yield
\begin{equation*}
\begin{split}
    G(\delta,q^a)
    &=
      \{A^+(q^p-Bq^a-c): q^p\in C^p(q^a), q^p>0\}
\\
    &\subseteq
      A^+ [0,1]^{m_p} - A^+B [0,1]^{m_a} - A^+c,
\end{split}
\end{equation*}
where $A^+$ is the Moore-Penrose inverse of $A$. Thus $G$ is bounded.

(b) Since $Z$ is the natural projection of $G$, it is open, bounded and nonempty.

(c) Boundedness follows from (a); the rest is trivial.
\end{proof}

\subsubsection{Global minima of the maps $\psi_z$ ($z\in Z$)}

Fix $z=(\delta,x)\in Z$ and define a map $\psi_z:\bar G(z)\to\RRRex$ by
\begin{equation}\label{L:PP-lin-psiz}
    \psi_z(y) \triangleq -\sum_{i\in\Ap} \w_i(\delta)\log d_i(x,y),
    \qquad
    d_i(x,y)\triangleq \sp{\alpha_i}{y} + \sp{\beta_i}{x} + c_i
\end{equation}
(the convention $\log 0=-\infty$ applies),
where $\alpha_i$ and $\beta_i$ denote the $i$-th rows of $A$ and~$B$, respectively.
Since $\psi_z$ is differentiable on $G(z)$, we can
define a map $F:G\to \RRR^s$ by (\ref{EQ:PP-lin-F-def}); that is,
\begin{equation*}
  F_h(\delta,x,y)
  \triangleq
    \sum_{i\in\Ap}  \frac{a_{ih} \w_i(\delta) }{d_i(x,y)}
  \qquad
  \text{for }
  h=1,\dots,s
  \text{ and }
  (\delta,x,y)\in G.
\end{equation*}
Easy computation gives (recall that $\nu(\cdot)$, hence also $\w(\cdot)$, is differentiable)
\begin{equation}\label{EQ:PP-lin-Fder}
    \frac{\partial F}{\partial \delta}
    =A'E1^p,
    \qquad
    \frac{\partial F}{\partial x}
    =-A'DB,
    \qquad
    \frac{\partial F}{\partial y}
    =-A'DA,
\end{equation}
where $A'$ is the transpose of $A$,
and $D=D(\delta,x,y), E=E(\delta,x,y)$ are given by
\begin{equation}\label{EQ:PP-lin-DE}
    D \triangleq \diag\left(  \frac{\w_i(\delta)}{d_i^2(x,y)}  \right)_{i=1}^{m_p}
    \qquad\text{and}\qquad
    E \triangleq \diag\left(  \frac{\w_i'(\delta)}{d_i(x,y)}   \right)_{i=1}^{m_p} \,.
\end{equation}

\begin{lemma}\label{L:PP-lin-psiz-argmin}
For every $z\in Z$ there is a unique minimizer $\varphi(z)\in G(z)$ of $\psi_z$:
    \begin{equation*}
        \argmin_{\bar G(z)} {\psi_z} = \{\varphi(z)\}.
    \end{equation*}
\end{lemma}
\begin{proof}
Let $z=(\delta,x)\in Z$.
Since $\bar G(z)$ is compact (Lemma~\ref{L:PP-lin-G-Z-Gz}) and $\psi_z:\bar G(z)\to\RRRex$
is continuous, $\argmin \psi_z$ is nonempty.
If $y\in \bar G(z)\setminus G(z)$ then there is $1\le i\le m_p$ such that $d_i(x,y)=0$ and so
$\psi_z(y)=\infty$. Hence $\argmin \psi_z\subseteq G(z)$.
The map $\psi_z$ is strictly convex on $G(z)$ (use \cite[Prop.~1.2.6(b)]{bertsekas}
and the fact that, by (\ref{EQ:PP-lin-Fder}),
$\nabla^2 \psi_z=A'DA$ is positive definite). Hence the minimum is unique.
\end{proof}

Since $G(z)$ is open for every $z$, the necessary condition for optimality gives the following result.

\begin{lemma}\label{L:PP-lin-phi-equations}
Let $\varphi:Z\to\RRR^s$ be as in Lemma~\ref{L:PP-lin-psiz-argmin}.
Then, for every $(z,y)\in G$,
\begin{equation*}
    F(z,y)=0
    \qquad\text{if and only if}\qquad
    y=\varphi(z).
\end{equation*}
\end{lemma}

Since $F$ is continuously differentiable and
$\partial F/\partial y = -A'DA$ is always regular,
the local implicit function theorem \cite[Prop.~1.1.14]{bertsekas}
and (\ref{EQ:PP-lin-Fder}) immediately imply the next lemma.

\begin{lemma}\label{L:PP-lin-phi-der}
The map $\varphi:Z\to\RRR^s$ is continuously differentiable and,
for every $z=(\delta,x)\in Z$,
\begin{equation*}
    \frac{\partial \varphi}{\partial \delta}
    = (A'DA) ^{-1} A'E1^p
  \qquad\text{and}\qquad
    \frac{\partial \varphi}{\partial x}
    = -(A'DA) ^{-1} A'DB,
\end{equation*}
where $D=D(z,\varphi(z)),E=E(z,\varphi(z))$ are given by (\ref{EQ:PP-lin-DE}).
\end{lemma}

In the next section  an algebraic result (Proposition~\ref{P:PP-bounded-ADAinvAD})
implying that $\varphi$ is Lipschitz on $Z$ (Lemma~\ref{L:PP-lin-phi-Lipschitz}) is proven.

\subsubsection{Boundedness of $(A'DA)^{-1}A'D$}

Let $\norm{\cdot}$ denote the spectral matrix norm \cite[p.~37-4]{hogben2006handbook}, that is, the matrix norm induced by
the Euclidean vector norm,
which is also denoted by $\norm{\cdot}$.
If $A$ is a matrix, $A'$ denotes its transpose.
The following result must be known, but the authors are not able to give a reference for it.

\begin{proposition}\label{P:PP-bounded-ADAinvAD}
Let $1\le s\le m$ and $A$ be an $m\times s$ matrix with full column rank $\rank(A)=s$. Then
there is $\cC>0$ such that
\begin{equation*}
    \norm{(A'DA)^{-1}A'D} \le \cC
\end{equation*}
for every $m\times m$-diagonal matrix $D$ with
positive entries on the diagonal.
\end{proposition}

Before proving this proposition we  give a formula for the inverse of $A'DA$, which
is a simple consequence of the Cauchy-Binet formula; cf.~\cite[p.~4-4]{hogben2006handbook}.
To this end some notation is needed.
If $C$ is an $m\times k$ matrix and $H\subseteq \{1,\dots,m\}$, $K\subseteq \{1,\dots,k\}$ are nonempty,
by $C[H,K]$ we denote the submatrix of $C$ of entries that lie in the rows of $C$ indexed by $H$
and the columns indexed by $K$.
$C_H$, $C^{(j)}$, $C^{(ij)}$, and $C_H^{(j)}$
are shorthands for $C[H,\{1,\dots,k\}]$,
$C[\{1,\dots,m\},\{j\}^c]$, $C[\{i\}^c,\{j\}^c]$,
and $C[H,\{j\}^c]$, respectively (there, $K^c$ means the complement of $K$).
For $l\le m$, $\HHh_{m,l}$ denotes the system of all subsets of $\{1,\dots,m\}$ of cardinality $l$.
Finally, if $(x_i)_{i\in\IIi}$ is an indexed system of numbers and $I\subseteq \IIi$ is finite, put
\begin{equation*}
  x_I\triangleq\prod_{i\in \IIi} x_i.
\end{equation*}

\begin{lemma}\label{L:PP-ADAinv}
Let $1\le s\le m$, $A$ be an $m\times s$ matrix with full rank $\rank(A)=s$,
and $D=\diag(d_1,\dots,d_m)$ be a diagonal matrix with every $d_i>0$.
Put
\begin{equation*}
    \B\triangleq A'DA.
\end{equation*}
Then $\B$ is regular and the following are true:
\begin{enumerate}
  \item[(a)] $\det(\B)=\sum_{H\in\HHh_{m,s}} d_H \aleph_H^2>0$,
  \item[(b)] $\det(\B^{(hk)})=\sum_{I\in\HHh_{m,s-1}} d_I \aleph_{I,h}\aleph_{I,k}$
    for every $1\le h,k\le s$,
  \item[(c)] $\B^{-1}=(c_{hk})_{hk=1}^s$,
\end{enumerate}
where $\aleph_H\triangleq \det(A_H)$,
$\aleph_{I,h}\triangleq \det\left(A_I^{(h)}\right)$,
and $c_{hk}\triangleq (-1)^{h+k} \det(\B^{(hk)})/\det(\B)$
for every $H\in\HHh_{m,s}$, $I\in\HHh_{m,s-1}$, and $h,k\in\{1,\dots,s\}$.
\end{lemma}
\begin{proof}
Put $\bar A\triangleq DA = (d_i a_{ij})_{ij}$. By the Cauchy-Binet formula,
\begin{equation*}
  \det(A'\bar A)
  = \sum_{H\in \HHh_{m,s}} \det(A_H)\det(\bar A_H)
  = \sum_{H\in \HHh_{m,s}} d_H\aleph_H^2,
\end{equation*}
thus (a) and the regularity of $\B$ are proved. Also the property (b) immediately follows from the Cauchy-Binet formula,
since $\B^{(hk)}=(A^{(h)})' \bar{A}^{k}$. Finally, (c) follows from the formula $\B^{-1}=(1/\det(\B)) \adj \B$,
where $\adj \B$ is the adjugate of $\B$, and the fact that $\B$ is symmetric.
\end{proof}

\begin{proof}[Proof of Proposition~\ref{P:PP-bounded-ADAinvAD}]
The proof is by induction on $m$. If $m=s$ then $A$ is a regular square matrix and
$(A'DA)^{-1}A'D=A^{-1}$, so it suffices to put $\cC\triangleq \norm{A^{-1}}$.
Assume that $m\ge s$ and that the assertion is true for every matrix $A$ of type $m\times s$.
Take any $(m+1)\times s$ matrix $\tilde A$ of rank $s$ and any diagonal matrix
$\tilde D=\diag(d_1,\dots,d_m,\delta)$ with $d_i>0, \delta>0$.
Since $\norm{(A'(\theta D)A)^{-1}A'(\theta D)}=\norm{(A'DA)^{-1}A'D}$
for any $\theta>0$, we may assume that $\delta>1$ and $d_i>1$ for every $i$.
Put $D\triangleq \diag(d_1,\dots,d_m)$ and write $\tilde{A}$ in the form
\begin{equation*}
    \tilde{A}=
    \left(
       \begin{array}{c}
         A \\
         \alpha' \\
       \end{array}
     \right)
\end{equation*}
with $A$ being an $m\times s$ matrix and $\alpha\in\RRR^s$. Without loss of generality assume
that $\rank(A)=s$. By the induction hypothesis, there is a constant $\cC>0$ not depending on $D$ such that
$\norm{(A'DA)^{-1}A'D}\le \cC$.

An easy computation gives
\begin{equation*}
    \tilde{A}'\tilde D \tilde A = \B + \delta\alpha\alpha',
    \qquad\text{where } \B\triangleq A'DA.
\end{equation*}
By the Sherman-Morrison formula \cite[p.~14-15]{hogben2006handbook},
for any $u,v\in\RRR^s$, it holds that
\begin{equation*}
  (\B+uv')^{-1}
  = \B^{-1} - \frac{1}{1+v'\B^{-1}u} \B^{-1}uv' \B^{-1}.
\end{equation*}
Hence
\begin{equation*}
    (\tilde{A}'\tilde D \tilde A)^{-1}
    =
    \B^{-1} - \varsigma \B^{-1}\alpha\alpha' \B^{-1}
    \qquad
    \left(\varsigma \triangleq \frac{\delta}{1+\delta \alpha'\B^{-1}\alpha} \right)
\end{equation*}
and
\begin{equation*}
    (\tilde{A}'\tilde D \tilde A)^{-1}  \tilde A'\tilde D
    =
    (\B^{-1}A'D - \varsigma \B^{-1}\alpha\alpha' \B^{-1}A'D,\ \varsigma \B^{-1}\alpha).
\end{equation*}
To finish the proof it suffices to show that $\norm{\varsigma \B^{-1}\alpha}$ is bounded by a constant~$\cC'$ not depending on $D,\delta$.
(Indeed, if this is true then
$\norm{(\tilde{A}'\tilde D \tilde A)^{-1}  \tilde A'\tilde D}
\le (\cC+\cC'\norm{\alpha}\cC) + \cC'$, which follows from the matrix norm triangle inequality
applied to $(E, F) = (E, 0_{s\times 1}) + (0_{s\times m}, F)$,
and from the matrix norm consistency property; cf.~\cite[p.~37-4]{hogben2006handbook}.)

If $\alpha=0$, then $\varsigma \B^{-1}\alpha=0$, so one can take $\cC'=0$.
Assume now that $\alpha=(\alpha_h)_h\ne 0$.
Using the notation from Lemma~\ref{L:PP-ADAinv},
\begin{eqnarray*}
    \alpha'\B^{-1}\alpha
    &=&
    \sum_{h,k=1}^s c_{hk}\alpha_h\alpha_k
\\
    &=&
    \frac{1}{\det(\B)}\sum_{h,k=1}^s  \sum_{I\in\HHh_{m,s-1}}
       (-1)^{h+k}\alpha_h\alpha_k  d_I \aleph_{I,h}\aleph_{I,k}
\\
    &=&
    \frac{1}{\det(\B)}\sum_{I\in\HHh_{m,s-1}} d_I  \eps_{I}^2,
    \qquad\text{where}\quad
       \eps_{I} \triangleq \sum_{k=1}^s (-1)^{k} \alpha_k\aleph_{I,k} .
\end{eqnarray*}
On the other hand,
\begin{eqnarray*}
    (\B^{-1}\alpha)_h
    &=&
    \sum_{k=1}^s c_{hk} \alpha_k
\\
    &=&
    \frac{1}{\det(\B)}\sum_{k=1}^s \sum_{I\in\HHh_{m,s-1}}
       (-1)^{h+k} \alpha_k  d_I \aleph_{I,h}\aleph_{I,k}
\\
    &=&
    \frac{(-1)^h }{\det(\B)}
    \sum_{I\in\HHh_{m,s-1}}  d_I \eps_I\aleph_{I,h}.
\end{eqnarray*}
Thus, the $h$-th coordinate of $v\triangleq\varsigma \B^{-1}\alpha$ satisfies
\begin{eqnarray*}
  \abs{v_h}
  &=&
  \left|
  \sum_{I\in\HHh_{m,s-1}}
  \frac
    {d_I {\eps_I\aleph_{I,h}}}
    {\frac{1}{\delta}\det(\B) + \sum_{J\in\HHh_{m,s-1}} d_J  \eps_{J}^2}
  \right|
\\
  &\le&
  \sum_{I\in\HHh_{m,s-1}}
  \frac
    {d_I \abs{\eps_I\aleph_{I,h}}}
    {\frac{1}{\delta}\det(\B) + \sum_{J\in\HHh_{m,s-1}} d_J  \eps_{J}^2}
\\
  &\le&
  \sum_{I\in\HHh_{m,s-1}}
  \frac
    {d_I \abs{\eps_I\aleph_{I,h}}}
    {\frac{1}{\delta}\det(\B) + d_I  \eps_{I}^2}
\\
  &\le&
  \sum_{I\in\HHh_{m,s-1},\, \eps_I\ne 0}
  \frac
    {\abs{\aleph_{I,h}}}
    {\abs{\eps_{I}}}
    \,.
\end{eqnarray*}
This proves that the absolute value of
every coordinate of $v=\varsigma \B^{-1}\alpha$ is bounded from above
by a constant not depending on $D,\delta$. It finishes the proof of Proposition~\ref{P:PP-bounded-ADAinvAD}.
\end{proof}

\subsubsection{Lipschitz property of $\varphi$}

The result of the previous section and Lemma~\ref{L:PP-lin-phi-der} yield that
$\varphi$ is Lipschitz, hence uniformly continuous on $Z$.

\begin{lemma}\label{L:PP-lin-phi-Lipschitz}
The map $\varphi:Z\to \RRR^s$  is Lipschitz on $Z$.
Consequently, there exists a continuous extension $\bar\varphi:\bar Z\to \RRR^s$ of $\varphi$
to the closure $\bar Z$ of $Z$.
\end{lemma}

\begin{proof}
By (\ref{EQ:PP-nu-bounded-differ}), the norm of
$D^{-1}E=\diag(d_i(x,y) \w_i'(\delta)/\w_i(\delta))_{i=1}^{m_p}$ is bounded from above
by a constant not depending on $z=(\delta,x)$ and $y$ (use that $0<d_i\le 1$ by (\ref{EQ:PP-lin-G-and-C})).
Thus, by Proposition~\ref{P:PP-bounded-ADAinvAD} and Lemma~\ref{L:PP-lin-phi-der},
$\varphi$ has bounded derivative on $Z$.
Now the Lipschitz property of $\varphi$ follows from the mean value theorem.
\end{proof}

\subsubsection{Proof of Theorem~\ref{T:PP-linear}}
The following lemma demonstrates that, for $\delta>0$,
the passive projection $\hat q^p_\PR(\delta)$
of the solution of \refPP{}
can be expressed via $\varphi$ and the active projection~$\hat q^a_\PR(\delta)$.

\begin{lemma}\label{L:PP-lin-psiz-hatqa}
Let $C$ be given by (\ref{EQ:C-linear}) and $A,B,c$ be given by Lemma~\ref{L:PP-lin-Cp(qa)}.
Let $(\nu(\delta))_{\delta\in(0,1)}$ satisfy (\ref{EQ:PP-nu(delta)})
and (\ref{EQ:PP-nu-bounded-differ}).
Then, for every $\delta>0$,
\begin{equation*}
    (\delta,\hat q^a_\PR(\delta))\in Z
    \qquad\text{and}\qquad
    \hat q^p_\PR(\delta) = A\varphi(\delta,\hat q^a_\PR(\delta)) + B \hat q^a_\PR(\delta) + c.
\end{equation*}
\end{lemma}

\begin{proof}
Fix any $\delta>0$ and put $x\triangleq \hat q^a_\PR(\delta)$, $z\triangleq (\delta,x)$.
Let $\hat y \in\Lambda(x)$ be such that
$\hat q^p_\PR(\delta)=A\hat y + Bx+c$
(note that $\hat y$ is unique since $A$ has
full column rank; see Lemma~\ref{L:PP-lin-Cp(qa)}).
Then $z\in Z$ since $(\delta,x,\hat y)$ belongs to $G$ (use that $\hat q^p_\PR(\delta)>0$).
Further, $G(z)$ is a subset of $\Lambda(x)$ by (\ref{EQ:PP-lin-Cp-and_Gz}), and
$\ell_{\delta}(x,Ay+Bx+c)=\infty$ for every $y\in\Lambda(x)\setminus G(z)$
(indeed, for such $y$ some coordinate of $Ay+Bx+c$ is zero).
Thus $\hat y\in G(z)$ and,
by (\ref{L:PP-lin-psiz}) and (\ref{EQ:PP-nu-bounded-differ}),
\begin{equation*}
\begin{split}
    \hat y
    &= \argmin_{y\in G(z)} {\ell_{\delta}(x,Ay+Bx+c)}
\\
    &= \argmin_{y\in G(z)} {\left( -\sum_{i\in \Ap} \nu_i(\delta)
    \log d_i(x,y)\right)}
\\
    &= \argmin_{y\in G(z)} {\psi_z(y)}.
\end{split}
\end{equation*}
Hence, by Lemma~\ref{L:PP-lin-psiz-argmin}, $\hat y=\varphi(z)$ and so
$\hat q^p_\PR(\delta) = A\varphi(z) + B x + c$.
\end{proof}

\begin{proof}[Proof of Theorem~\ref{T:PP-linear}]
Let $\bar\varphi:\bar Z\to\RRR^s$ denote the continuous extension of $\varphi$
to the closure $\bar Z$ of $Z$
(see Lemma~\ref{L:PP-lin-phi-Lipschitz}).
By Theorem~\ref{T:PP}, $\lim_{\delta\searrow 0} \hat q^a_\PR(\delta)$
exists and equals the unique member
$\hat q^a_\PR$ of $\SOL_\PR$.
Since $(\delta,\hat q^a_\PR(\delta))\in Z$ for every $\delta>0$
(see Lemma~\ref{L:PP-lin-psiz-hatqa}),
$(0,\hat q^a_\PR)$ belongs to $\bar Z$. Thus, by Lemma~\ref{L:PP-lin-psiz-hatqa},
\begin{equation*}
\begin{split}
  \lim_{\delta\searrow 0} \hat q^p_\PR(\delta)
  &= \lim_{\delta\searrow 0}  \left(A\varphi(\delta,\hat q^a_\PR(\delta)) + B \hat q^a_\PR(\delta) + c\right)
\\
  &= A\bar\varphi(0,\hat q^a_\PR) + B \hat q^a_\PR + c.
\end{split}
\end{equation*}
This proves convergence of $\hat q^p_\PR(\delta)$. Now, by Theorem~\ref{T:PP}, Theorem~\ref{T:PP-linear} follows.
\end{proof}

\section*{Acknowledgements}
Substantive feedback from Jana Majerov\'a is gratefully acknowledged.
This research is an outgrowth of the project ``SPAMIA'', M\v S SR-3709/2010-11,
supported by the Ministry of Education, Science, Research and Sport of the Slovak Republic, under the heading of the state budget support for research and development.
The first author also acknowledges support
from VEGA~2/0038/12, APVV-0096-10, and CZ.1.07/2.3.00/20.0170 grants.

\bibliography{refs-21}

\begin{thebibliography}{10}

\bibitem{Agresti}
A.~Agresti.
\newblock {\em Categorical data analysis}.
\newblock Wiley-Interscience [John Wiley \& Sons], New York, second edition,
  2002.

\bibitem{Agresti2}
A.~Agresti and B.~A. Coull.
\newblock The analysis of contingency tables under inequality constraints.
\newblock {\em J. Statist. Plann. Inference}, 107(1-2):45--73, 2002.

\bibitem{AitchisonCompositional}
J.~Aitchison.
\newblock {\em The statistical analysis of compositional data}.
\newblock Chapman \& Hall, London, 1986.

\bibitem{Aitchison}
J.~Aitchison and S.~D. Silvey.
\newblock Maximum-likelihood estimation of parameters subject to restraints.
\newblock {\em Ann. Math. Statist.}, 29:813--828, 1958.

\bibitem{Anaya}
K.~Anaya-Izquierdo, F.~Critchley, P.~Marriott, and P.~W. Vos.
\newblock Computational information geometry: theory and practice.
\newblock {\em arXiv preprint arXiv:1209.1988}, 2012.

\bibitem{Baker}
R.~Baker, M.~Clarke, and P.~Lane.
\newblock Zero entries in contingency tables.
\newblock {\em Comput. Statist. Data Anal.}, 3:33--45, 1985.

\bibitem{Bergsma}
W.~Bergsma, M.~Croon, and L.~A. van~der Ark.
\newblock The empty set and zero likelihood problems in maximum empirical
  likelihood estimation.
\newblock {\em Electron. J. Stat.}, 6:2356--2361, 2012.

\bibitem{BergsmaBook}
W.~Bergsma, M.~A. Croon, and J.~A. Hagenaars.
\newblock {\em Marginal models}.
\newblock Springer, 2009.

\bibitem{bertsekas}
D.~P. Bertsekas.
\newblock {\em Convex analysis and optimization}.
\newblock Athena Scientific, Belmont, MA, 2003.
\newblock With Angelia Nedi{\'c} and Asuman E. Ozdaglar.

\bibitem{elbarmi-dykstra}
H.~El~Barmi and R.~L. Dykstra.
\newblock Restricted multinomial maximum likelihood estimation based upon
  {F}enchel duality.
\newblock {\em Statist. Probab. Lett.}, 21(2):121--130, 1994.

\bibitem{BD2}
H.~El~Barmi and R.~L. Dykstra.
\newblock Maximum likelihood estimates via duality for log-convex models when
  cell probabilities are subject to convex constraints.
\newblock {\em Ann. Statist.}, 26(5):1878--1893, 1998.

\bibitem{Rinaldo}
S.~E. Fienberg and A.~Rinaldo.
\newblock Maximum likelihood estimation in log-linear models.
\newblock {\em Ann. Statist.}, 40(2):996--1023, 2012.

\bibitem{Fisher}
R.~A. Fisher.
\newblock Theory of statistical estimation.
\newblock {\em Math. Proc. Cambridge Philos. Soc.}, 22(05):700--725, 1925.

\bibitem{geyer09}
C.~J. Geyer.
\newblock Likelihood inference in exponential families and directions of
  recession.
\newblock {\em Electron. J. Stat.}, 3:259--289, 2009.

\bibitem{Rsolnp}
A.~Ghalanos and S.~Theussl.
\newblock {\em Rsolnp: general non-linear optimization using augmented Lagrange
  multiplier method}.
\newblock 2012.
\newblock R package version 1.14.

\bibitem{Gokhale}
D.~V. Gokhale.
\newblock Iterative maximum likelihood estimation for discrete distributions.
\newblock {\em Sankhy\=a Ser. B}, 35(3):293--298, 1973.

\bibitem{ESP}
M.~Grend{\'a}r and G.~Judge.
\newblock Empty set problem of maximum empirical likelihood methods.
\newblock {\em Electron. J. Stat.}, 3:1542--1555, 2009.

\bibitem{Rcode}
M.~Grend{\'a}r and V.~\v{S}pitalsk{\'y}.
\newblock Source code for ``{M}ultinomial and empirical likelihood under convex
  constraints: directions of recession, {F}enchel~duality, perturbations''.
\newblock
  \href{http://www.savbb.sk/~grendar/likelihood/Rcode-arxiv-v1.zip}{http://www.savbb.sk/$\sim$grendar/likelihood/Rcode-arxiv-v1.zip}.

\bibitem{Haber}
M.~Haber.
\newblock Maximum likelihood methods for linear and log-linear models in
  categorical data.
\newblock {\em Comput. Statist. Data Anal.}, 3(1):1--10, 1985.

\bibitem{HallPresnell}
P.~Hall and B.~Presnell.
\newblock Intentionally biased bootstrap methods.
\newblock {\em J. R. Stat. Soc. Ser. B Stat. Methodol.}, 61(1):143--158, 1999.

\bibitem{hogben2006handbook}
L.~Hogben, editor.
\newblock {\em Handbook of linear algebra}.
\newblock Discrete Mathematics and its Applications (Boca Raton). Chapman \&
  Hall/CRC, Boca Raton, FL, 2007.

\bibitem{IrelandKuKullback}
C.~T. Ireland, H.~H. Ku, and S.~Kullback.
\newblock Symmetry and marginal homogeneity of an {$r\times r$} contingency
  table.
\newblock {\em J. Amer. Statist. Assoc.}, 64:1323--1341, 1969.

\bibitem{IrelandKullback}
C.~T. Ireland and S.~Kullback.
\newblock Contingency tables with given marginals.
\newblock {\em Biometrika}, 55:179--188, 1968.

\bibitem{Kall}
P.~Kall.
\newblock Approximation to optimization problems: an elementary review.
\newblock {\em Math. Oper. Res.}, 11(1):9--18, 1986.

\bibitem{Kerridge}
D.~F. Kerridge.
\newblock Inaccuracy and inference.
\newblock {\em J. Roy. Statist. Soc. Ser. B}, 23:184--194, 1961.

\bibitem{Klotz}
J.~H. Klotz.
\newblock Testing a linear constraint for multinomial cell frequencies and
  disease screening.
\newblock {\em Ann. Statist.}, 6(4):904--909, 1978.

\bibitem{Lang}
J.~B. Lang.
\newblock Multinomial-{P}oisson homogeneous models for contingency tables.
\newblock {\em Ann. Statist.}, 32(1):340--383, 2004.

\bibitem{Lindsay}
B.~G. Lindsay.
\newblock Mixture models: theory, geometry and applications.
\newblock In {\em NSF-CBMS regional conference series in probability and
  statistics}, pages i--163, 1995.

\bibitem{Lindsey}
J.~K. Lindsey.
\newblock {\em Parametric statistical inference}.
\newblock The Clarendon Press, Oxford University Press, New York, 1996.

\bibitem{LittleWu}
R.~J. Little and M.-M. Wu.
\newblock Models for contingency tables with known margins when target and
  sampled populations differ.
\newblock {\em J. Amer. Statist. Assoc.}, 86(413):87--95, 1991.

\bibitem{Owen1988}
A.~B. Owen.
\newblock Empirical likelihood ratio confidence intervals for a single
  functional.
\newblock {\em Biometrika}, 75(2):237--249, 1988.

\bibitem{Owen}
A.~B. Owen.
\newblock {\em Empirical likelihood}.
\newblock Chapman \& Hall/CRC, 2001.

\bibitem{Pelz}
W.~Pelz and I.~Good.
\newblock Estimating probabilities from contingency tables when the marginal
  probabilities are known, by using additive objective functions.
\newblock {\em Statistician}, 35(1):45--50, 1986.

\bibitem{Pitman}
E.~J.~G. Pitman.
\newblock {\em Some basic theory for statistical inference}.
\newblock Chapman and Hall, London; A Halsted Press Book, John Wiley \& Sons,
  New York, 1979.

\bibitem{QL}
J.~Qin and J.~Lawless.
\newblock Empirical likelihood and general estimating equations.
\newblock {\em Ann. Statist.}, 22(1):300--325, 1994.

\bibitem{R}
{R Core Team}.
\newblock {\em R: a language and environment for statistical computing}.
\newblock Vienna, Austria, 2012.
\newblock {ISBN} 3-900051-07-0.

\bibitem{rockafellar1974}
R.~T. Rockafellar.
\newblock {\em Conjugate duality and optimization}.
\newblock Society for Industrial and Applied Mathematics, Philadelphia, Pa.,
  1974.

\bibitem{rockafellar}
R.~T. Rockafellar.
\newblock {\em Convex analysis}.
\newblock Princeton University Press, Princeton, NJ, 1997.

\bibitem{rockafellar2009variational}
R.~T. Rockafellar and R.~J.-B. Wets.
\newblock {\em Variational analysis}, volume 317.
\newblock Springer-Verlag, Berlin, 1998.

\bibitem{Silvey}
S.~D. Silvey.
\newblock The {L}agrangian multiplier test.
\newblock {\em Ann. Math. Statist.}, 30:389--407, 1959.

\bibitem{Smith}
J.~H. Smith.
\newblock Estimation of linear functions of cell proportions.
\newblock {\em Ann. Math. Statist.}, 18:231--254, 1947.

\bibitem{Stirling}
W.~D. Stirling.
\newblock Testing linear hypotheses in contingency tables with zero cell
  counts.
\newblock {\em Comput. Statist. Data Anal.}, 4(1):1--13, 1986.

\bibitem{Wets}
R.~J.-B. Wets.
\newblock Statistical estimation from an optimization viewpoint.
\newblock {\em Ann. Oper. Res.}, 85:79--101, 1999.

\bibitem{Ye}
Y.~Ye.
\newblock {\em Interior algorithms for linear, quadratic, and linearly
  constrained non-linear programming}.
\newblock PhD thesis, Department of {ESS}, Stanford University, 1987.

\bibitem{ZhangBootstrap}
B.~Zhang.
\newblock Bootstrapping with auxiliary information.
\newblock {\em Canad. J. Statist.}, 27(2):237--249, 1999.

\bibitem{Zhang}
Z.~Zhang.
\newblock Interpreting statistical evidence with empirical likelihood
  functions.
\newblock {\em Biom. J.}, 51(4):710--720, 2009.

\end{thebibliography}


\vspace{0.8cm}
\noindent
Authors' affiliations

\medskip\smallskip
\noindent
M. Grend\'ar,
Slovanet a.s., Z\'ahradn\'icka 151, 821 08 Bratislava, Slovakia;
Department of~Mathematics, Faculty of~Natural Sciences, Matej Bel University,
Bansk\'a Bystrica, Slovakia;
Institute of~Measurement Science, Slovak Academy of~Sciences (SAS), Bratislava, Slovakia;
Institute of~Mathematics, SAS, Bratislava, Slovakia; \texttt{marian.grendar@savba.sk}

\medskip\smallskip
\noindent
V. \v Spitalsk\'y,
Slovanet a.s., Z\'ahradn\'icka 151, 821 08 Bratislava, Slovakia;
Department of~Mathematics, Faculty of~Natural Sciences, Matej Bel University,
Bansk\'a Bystrica, Slovakia; \texttt{vladimir.spitalsky@umb.sk}

\end{document}